\documentclass[10pt]{article}

\usepackage[english]{babel}
\usepackage{bbold}
\usepackage{stmaryrd}
\usepackage{verbatim}
\usepackage[utf8x]{inputenc}
\usepackage{amsmath}
\numberwithin{equation}{subsection}
\usepackage{amssymb}
\usepackage{amsthm}
\usepackage{mathtools}
\usepackage[T1]{fontenc}
\usepackage{setspace}
\usepackage{amsfonts}
\usepackage{eucal}
\usepackage{mathrsfs}
\usepackage[colorinlistoftodos]{todonotes} 
\usepackage{pictexwd,dcpic}
\usepackage{tikz}
\usepackage{tqft}
\usepackage{pigpen}
\usepackage{xcolor}
\usepackage{pdfpages}
\usepackage{hyperref}

\newtheorem*{theorem-non}{Theorem}
\newtheorem*{theorem-nonf}{Th\'eor\`eme}
\newtheorem*{corollary-non}{Corollary}
\newtheorem*{corollary-nonf}{Corollaire}
\newtheorem{lmm}[subsubsection]{Lemma}
\newtheorem{thm}[subsubsection]{Theorem}
\newtheorem{prop}[subsubsection]{Proposition}
\newtheorem{cor}[subsubsection]{Corollary}

\theoremstyle{definition}
\newtheorem{defn}[subsubsection]{Definition}
\newtheorem{context}[subsubsection]{Context}
\newtheorem{ex}[subsubsection]{Example}
\newtheorem{rmk}[subsubsection]{Remark}
\newtheorem*{rmk-non}{Remark}

\newtheorem{notation}[subsubsection]{Notation}
\newtheorem{construction}[subsubsection]{Construction}

\newcommand{\ital}[1]{\textit{#1}}

\newcommand{\hm}[0]{\textup{Hom}}

\newcommand{\oo}{\infty}
\newcommand{\Z}[0]{\mathbb{Z}}
\newcommand{\N}[0]{\mathbb{N}}
\newcommand{\C}[0]{\mathbb{C}}
\newcommand{\Q}[0]{\mathbb{Q}}
\newcommand{\V}[0]{\mathbb{V}}
\newcommand{\aff}[2]{\mathbb{A}^{#1}_{#2}}
\newcommand{\proj}[2]{\mathbb{P}^{#1}_{#2}}
\newcommand{\sch}[0]{\textup{\textbf{Sch}}_{S}}
\newcommand{\Cc}[0]{\mathcal{C}}

\newcommand{\sm}[0]{\textup{\textbf{Sm}}_S}
\newcommand{\ag}[1]{\mathbb{A}^1_{#1}/\mathbb{G}_{m,#1}}
\newcommand{\gm}[1]{\mathbb{G}_{m,#1}}
\newcommand{\bgm}[1]{\textup{B}\mathbb{G}_{m,#1}}
\newcommand{\Ec}[0]{\mathcal{E}}
\newcommand{\Lc}[0]{\mathcal{L}}
\newcommand{\Lb}[0]{\mathbb{L}}
\newcommand{\Mcor}[0]{\mathcal{M}}
\newcommand{\Oc}[0]{\mathcal{O}}
\newcommand{\Rc}[0]{\mathcal{R}}
\newcommand{\Fc}[0]{\mathcal{F}}

\newcommand{\Map}[0]{\textup{Map}}
\newcommand{\rhom}[0]{\mathbb{R}\textup{Hom}}
\newcommand{\bu}[0]{\textup{B}\mathbb{U}}
\newcommand{\Ql}[1]{\mathbb{Q}_{\ell#1}}
\newcommand{\Rl}[0]{\mathcal{R}^{\ell}}
\newcommand{\Rlv}[0]{\mathcal{R}^{\ell,\vee}}

\newcommand{\rperf}[0]{\mathcal{R}_{\textup{\textbf{Perf}}}}
\newcommand{\hk}[0]{\textup{HK}}
\newcommand{\Mv}[0]{\mathcal{M}^{\vee}}
\newcommand{\HQ}[0]{\textup{H}\mathbb{Q}}
\newcommand{\MB}[1]{\textup{M}\mathbb{B}_{#1}}

\newcommand{\shvl}[0]{\textup{\textbf{Shv}}_{\mathbb{Q}_{\ell}}}
\newcommand{\shv}[0]{\textup{\textbf{Shv}}}
\newcommand{\md}[0]{\textup{\textbf{Mod}}}
\newcommand{\fclgm}[0]{\textup{LG}}
\newcommand{\lgm}[1]{\textup{LG}_{S}(#1)}

\newcommand{\tlgm}[1]{\textup{LG}_{(#1,\Lc_{#1})}}
\newcommand{\tlgmproj}[1]{\textup{LG}_{(\mathbb{P}(#1),\mathcal{O}(1))}}
\newcommand{\tlgmr}[1]{\textup{LG}_{(#1,\Ec_{#1})}}

\newcommand{\dgcat}{\textup{dgCat}_S}
\newcommand{\dgcatlf}{\textup{dgCat}^{\textup{lf}}_S}
\newcommand{\dgcatdk}{\textup{\textbf{dgCat}}_S}
\newcommand{\dgcatm}{\textup{\textbf{dgCat}}^{\textup{idm}}_S}
\newcommand{\dgcatlfo}{\textup{dgCat}^{\textup{lf},\otimes}_S}
\newcommand{\dgcatdko}{\textup{\textbf{dgCat}}^{\otimes}_S}
\newcommand{\dgcatmo}{\textup{\textbf{dgCat}}^{\textup{idm},\otimes}_S}
\newcommand{\dgcatft}{\textup{\textbf{dgCat}}_S^{\textup{ft}}}

\newcommand{\cohb}[1]{\textup{\textbf{Coh}}^{b}(#1)}
\newcommand{\perf}[1]{\textup{\textbf{Perf}}(#1)}
\newcommand{\qcoh}[1]{\textup{\textbf{QCoh}}(#1)}
\newcommand{\cohm}[1]{\textup{\textbf{Coh}}^{-}(#1)}
\newcommand{\cohbp}[2]{\textup{\textbf{Coh}}^b(#1)_{\textup{\textbf{Perf}}(#2)}}

\newcommand{\cohs}[1]{\textup{Coh}^{s}(#1)}
\newcommand{\sing}[1]{\textup{\textbf{Sing}}(#1)}

\newcommand{\sring}{\textup{\textbf{sRing}}}
\newcommand{\dscaff}{\textup{\textbf{dSch}}^{\textup{aff}}}
\newcommand{\dsc}{\textup{\textbf{dSch}}}

\newcommand{\uH}[1]{\textup{\textbf{H}}_{#1}}
\newcommand{\sh}[0]{\textup{\textbf{SH}}_S}
\newcommand{\shm}[0]{\textup{\textbf{SH}}_S^{\otimes}}
\newcommand{\shnc}[0]{\textup{\textbf{SH}}^{\textup{nc}}_S}

\newcommand{\shncm}[0]{\textup{\textbf{SH}}^{\textup{nc}, \otimes}_S}
\newcommand{\shncmop}[0]{\textup{\textbf{SH}}^{\textup{nc,op}, \otimes}_S}
\newcommand{\nc}[1]{\mathcal{N}c_{#1}}

\title{\parbox{\linewidth}{\centering Motivic and \texorpdfstring{$\ell$}{l}-adic realizations of the category of singularities of the zero locus of a global section of a vector bundle}}
\author{Massimo Pippi}
\date{}
\begin{document}
\maketitle
\begin{abstract}
We study the motivic and $\ell$-adic realizations of the dg category of singularities of the zero locus of a global section of a line bundle over a regular scheme. We will then use the formula obtained in this way together with a theorem due to D.~Orlov and J.~Burke - M.~Walker to give a formula for the $\ell$-adic realization of the dg category of singularities of the zero locus of a global section of a vector bundle. In particular, we obtain a formula for the $\ell$-adic realization of the dg category of singularities of the special fiber of a scheme over a regular local ring of dimension $n$. 
\end{abstract}

\tableofcontents

\section{Introduction}

The connection between categories of singularities and vanishing cycles is well known, thanks to works of T.~Dyckerhoff (\cite{dy11}), A.~Preygel (\cite{pr11}), A.~Efimov (\cite{efi18}) and many others. Recently, an instance of this fact has been studied in \cite{brtv}. 

We start by quickly reviewing the main theorem of \cite{brtv}, that served both as a model and as a motivation for the investigations presented later. The main purpose of the above mentioned paper is to identify a classical object of singularity theory, namely the $\ell$-adic sheaf of (inertia invariant) vanishing cycles, with the $\ell$-adic cohomology of a non-commutative space (that is defined in \textit{loc. cit.}), the dg category of singularities of the special fiber. The main result of A.~Blanc, M.~Robalo, B.~To\"en and G.~Vezzosi's paper (\cite[Theorem 4.39]{brtv}) reads as follows:

Let $p:X\rightarrow S$ be a proper, flat, regular scheme over an excellent strictly henselian trait $S$. Let  $i_{\sigma}:\sigma \hookrightarrow S$ be the embedding of the closed point in $S$ and let $p_{\sigma}:X_{\sigma}\rightarrow \sigma$ the pullback of $p$ along $i_{\sigma}$. Fix a prime number $\ell$ different from the characteristic of $\sigma$.
Let $\Phi_p(\Ql{,X}(\beta))$ be the $\ell$-adic sheaf of vanishing cycles associated to $\Ql{,X}(\beta)= \bigoplus_{n\in \Z}\Ql{,X}(n)[2n]$. Denote by $I$ the inertia group (as $S$ is strictly henselian, it coincides with the absolute Galois group of the open point in $S$) and by $(-)^{\textup{h}I}$ the (homotopy) fixed points $\oo$-functor. In \cite{brtv}, the authors define the $\ell$-adic realization of dg categories, denoted by $\Rlv_S$. It is an $\oo$-functor that associates an $\ell$-adic complex to a dg category.
For every (derived) scheme $Z$, let $\sing{Z}$ be the dg category of singularities of $Z$, that is the dg quotient\footnote{notice that if $Z$ is derived, we need to assume that $\Oc_Z \in \cohb{Z}$.}
\begin{equation*}
    \sing{Z}:=\cohb{Z}/\perf{Z},
\end{equation*}
where $\cohb{Z}$ denotes the dg-category of complexes of quasi coherent $\Oc_Z$-modules with coherent total cohomology and $\perf{Z}$ that of perfect complexes of $\Oc_Z$-modules.

A.~Blanc, M.~Robalo, B.~To\"en and G.~Vezzosi show that there is an equivalence of $i_{\sigma}^*\Rlv_S(\sing{S,0})\simeq \Ql{,\sigma}(\beta)\otimes_{\Ql{,\sigma}} \Ql{,\sigma}^{\textup{h}I}$-modules
\begin{equation*}
    i_{\sigma}^*\Rlv_S(\sing{X_{\sigma}})\simeq p_{\sigma *}\Phi_{p}(\Ql{,X}(\beta))^{\textup{h}I}[-1].
\end{equation*}

\begin{rmk-non}
This theorem is of major importance in B.~To\"en - G.~Vezzosi's approach to the Bloch's conductor conjecture (\cite{bl85}). For more details, see \cite{tv17}, \cite{tv19a}
,\cite{tv19b}.
\end{rmk-non}
The content of this theorem is summarized in the following mind map:
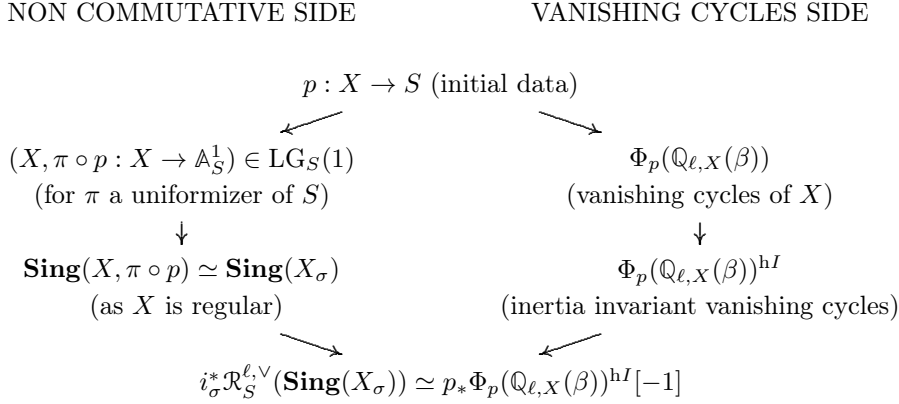
\begin{equation*}
    \begindc{\commdiag}[14]
    \obj(0,80)[1]{$p:X\rightarrow S \text{ (initial data)}$}
    \obj(-70,60)[2]{$(X,\pi \circ p:X\rightarrow \aff{1}{S})\in \lgm{1}$}
    \obj(-70,50)[3]{$(\text{for } \pi \text{ a uniformizer of } S)$}
    \obj(-70,30)[4]{$\sing{X,\pi \circ p}\simeq \sing{X_{\sigma}}$}
    \obj(-70,20)[5]{$\text{ (as } X \text{ is regular)}$}
    \obj(0,0)[6]{$i_{\sigma}^*\Rlv_S(\sing{X_{\sigma}})\simeq p_*\Phi_{p}(\Ql{,X}(\beta))^{\textup{h}I}[-1]$}
    \obj(70,60)[7]{$\Phi_p(\Ql{,X}(\beta))$}
    \obj(70,50)[8]{$\text{(vanishing cycles of } X)$}
    \obj(70,30)[9]{$\Phi_p(\Ql{,X}(\beta))^{\textup{h}I}$}
    \obj(70,20)[10]{$\text{(inertia invariant vanishing cycles)}$}
    \mor(-20,75)(-50,65){$$}
    \mor{3}{4}{$$}
    \mor(-50,15)(-20,5){$$}
    \mor(20,75)(50,65){$$}
    \mor{8}{9}{$$}
    \mor(50,15)(20,5){$$}
    \obj(-70,100)[]{$\text{NON COMMUTATIVE SIDE}$}
    \obj(70,100)[]{$\text{VANISHING CYCLES SIDE}$}
    \enddc
\end{equation*}
We can ask whether the above diagram makes sense in more general situations. 
For example, one can start with the datum of a proper, flat and regular scheme over an excellent local regular ring of dimension $n$. One recovers the case treated in \cite{brtv} when $n=1$.

It is immediate to observe that the left hand side of the diagram makes sense without any change: assume that we are given a proper, flat morphism $p:X\rightarrow S$, with $X$ regular and $S$ local, regular of dimension $n$. Let $\underline{\pi}=(\pi_1,\dots,\pi_n)$ be a collection of generators of the closed point of $S$. Then we can consider the morphism $\underline{\pi}\circ p:X\rightarrow \aff{n}{S}$ and its fiber $X_0$ along the origin $S\rightarrow \aff{n}{S}$. It makes perfectly sense to consider $\Rlv_S(\sing{X_0})$.

It comes up that this generalization is related to the following one: one can consider pairs $(X,s_X)$ where $X$ is a regular scheme over $S$ and $s_X:X\rightarrow \V(\Lc_S)$ is a morphism towards the total space of a line bundle $\Lc_S$ over $S$. One recovers the situation pictured above when $\Lc_S=\Oc_S$ is the trivial line bundle. In this situation, we want to compute $\Rlv_S(\sing{X_0})$, where $X_0$ is the fiber of $s_X:X\rightarrow \V(\Lc_S)$ along the zero section $S\rightarrow \V(\Lc_S)$.

One can view the former generalization as a particular case of the latter thanks to a theorem of D.~Orlov (\cite{orl06}) and J.~Burke-M.~Walker (\cite{bw15}), which tells us that the dg category of singularities of $(X,\underline{\pi}\circ p)$ is equivalent to the dg category of singularities of $(\proj{n-1}{X},(\pi_1\circ p)\cdot T_1+\dots + (\pi_n\circ p)\cdot T_n \in \Oc(1)(\proj{n-1}{X}))$.

Thus, we only need to find the appropriate generalization of (inertia invariant) vanishing cycles. The first thing we can think of are vanishing cycles over general bases, developed by G.~Laumon (\cite{lau82}) following ideas of P.~Deligne and further investigated by many others, including O.~Gabber, L.Illusie, F.~Orgogozo (see \cite{ilo14}, \cite{or06}). However, it seems that this is not the right point of view for our purposes. Instead, we will pursue the following analogy:
\begin{equation*}
    \begindc{\commdiag}[18]
    \obj(-60,70)[]{$\text{TOPOLOGY}$}
    \obj(0,70)[]{$\text{VANISHING CYCLES}$}
    \obj(60,70)[]{$\text{OUR SETTING}$}
    \obj(-60,50)[]{$\mathbb{D}\text{ unital}$}
    \obj(-60,40)[]{$\text{open disk}$}
    \obj(0,50)[]{$S \text{ strictly}$}
    \obj(0,40)[]{$ \text{henselian trait}$}
    \obj(60,50)[]{$\V(\Lc)\text{ total space}$}
    \obj(60,40)[]{$\text{of a line bundle } \Lc \text{ over }$}
    \obj(60,30)[]{$\text{a base scheme } S$}
    \obj(-60,10)[]{$0\hookrightarrow \mathbb{D}$}    \obj(-60,0)[]{$\text{origin}$}
    \obj(0,10)[]{$\sigma \hookrightarrow S$}
    \obj(0,0)[]{$\text{special point}$}
    \obj(60,10)[]{$S\hookrightarrow \V(\Lc)$}
    \obj(60,0)[]{$\text{zero section}$}
    \obj(-60,-20)[]{$\mathbb{D}^*=\mathbb{D}-\{0\}\hookrightarrow \mathbb{D}$}    \obj(-60,-30)[]{$\text{punctured disk}$}
    \obj(0,-20)[]{$\eta \hookrightarrow S$}
    \obj(0,-30)[]{$\text{generic point}$}
    \obj(60,-20)[]{$\mathcal{U}=\V(\Lc)-S \hookrightarrow \V(\Lc)$}
    \obj(60,-30)[]{$\text{open complementary}$}
    \obj(-60,-50)[]{$\widetilde{\mathbb{D}^*} \rightarrow \mathbb{D}^*$}    \obj(-60,-60)[]{$\text{universal cover}$}
     \obj(-60,-70)[]{$\text{of the punctured disk}$}
    \obj(0,-50)[]{$\bar{\eta}\rightarrow \eta$}
    \obj(0,-60)[]{$\text{separable closure}$}
    \obj(0,-70)[]{$\text{of the generic point}$}
    \obj(60,-60)[]{$?$}
    \enddc
\end{equation*}
As we will only need to define the analogous of inertia-invariant vanishing cycles, we will not face the problem of filling the empty spot in the mental map above. Nevertheless, we will come back to this matter at the end of the article, presenting a strategy to complete the picture.

We will define an appropriate generalization of $\Phi_p(\Ql{,X}(\beta))^{\textup{h}I}$ (see Definition \ref{definition monodromy-invarian t vanishing cycles}) and prove a generalization of the formula stated in \cite[Theorem 4.39]{brtv}. Our main theorem will then look as follows:
\begin{theorem-non}{\textup{\ref{main theorem}}}
Let $X$ be a regular scheme and let $s_X$ be a regular global section of a line bundle $\Lc_X$. Denote $X_0$ the zero locus of $s_X$. Then
\begin{equation*}
    \Rlv_{X_0}(\sing{X_0})\simeq \Phi^{\textup{mi}}_{(X,s_X)}(\Ql{}(\beta))[-1].
\end{equation*}
Here $\Phi^{\textup{mi}}_{(X,s_X)}(\Ql{}(\beta))$ is what we call the monodromy-invariant vanishing cycles $\ell$-adic sheaf (see Definition \ref{mododromy-invariant vanishing cycles def}). It coincides with inertia invariant vanishing cycles when we put ourselves in the situation considered in \cite{brtv}.
\end{theorem-non}

Using this theorem combined with the above mentioned result of D.~Orlov and J.~Burke - M.~Walker (Theorem \ref{reduction of codimension theorem}), we will deduce the following formula for the situation in which one starts with a regular scheme $X$ and a regular section $s_X$ of a vector bundle.
\begin{theorem-non}{\textup{\ref{formula for (X,s)}}}
Let $f:X\rightarrow S$ be a flat morphism, $\Ec_S$ be a vector bundle of rank $r$ over $S$ and let $s_X$ be a global section of $\Ec_X=f^*\Ec_S$. Assume that $X$ is a regular scheme and that $s_X$ is regular section. The following equivalence holds in $\md_{\Rlv_X(\sing{X,0})}(\shvl(X))$
\begin{equation*}
    \Rlv_X(\sing{X,s_X})\simeq p_*i_*\Phi^{\textup{mi}}_{(\mathbb{P}(\Ec_X),W_{s_X})}(\Ql{}(\beta))[-1],
\end{equation*}
where $W_{s_X}$ is the global section of $\Oc_{\mathbb{P}(\Ec_X)}(1)$ associated to $s_X$, $i:V(W_{s_X})\rightarrow \mathbb{P}(\Ec_X)$ is the closed embedding of the zero locus of $W_{s_X}$ in $\mathbb{P}(\Ec_X)$  and $p:\mathbb{P}(\Ec_X)\rightarrow X$ is the canonical projection.
\end{theorem-non}

In particular, the above theorem applies to the special case $\Ec_S \simeq \Oc_S^n$ and allows us to compute the $\ell$-adic realisation of the special fiber of a regular scheme over a local regular noetherian ring of dimension $n$. More precisely:

\begin{corollary-non}{\textup{\ref{answer intial question}}}
Assume that $S=Spec(A)$ is a noetherian regular local ring of dimension $n$ and let $\pi_1,\dots,\pi_n$ be generators of the maximal ideal. Let $p:X\rightarrow S=Spec(A)$ be a regular, flat $S$-scheme of finite type. Let $\underline{\pi}:S\rightarrow \aff{n}{S}$ be the closed embedding associated to $\pi_1,\dots,\pi_n$. Then $\underline{\pi}\circ p$ is a regular global section of $\Oc_X^{n}$. Then the equivalence
\begin{equation*}
    \Rlv_X(\sing{X,\underline{\pi}\circ p})\simeq q_*i_*\Phi^{\textup{mi}}_{(\proj{n-1}{X},W_{\underline{\pi}\circ p})}(\Ql{}(\beta))[-1]
\end{equation*}
holds in $\md_{\Rlv_X(\sing{X,\underline{0}})}(\shvl(X))$. 

Here $q:\proj{n-1}{X}=Proj_X(\Oc_X[T_1,\dots,T_n])\rightarrow X$ is the canonical projection and $i:V(W_{\underline{\pi}\circ p})\rightarrow \proj{n-1}{X}$ is the closed embedding determined by the equation 
\begin{equation*}
    W_{\underline{\pi}\circ p}=p^*(\pi_1)\cdot T_1+\dots + p^*(\pi_n)\cdot T_n=0.
\end{equation*}
\end{corollary-non}

We will end this article with some remarks on the following two problems:
\begin{enumerate}
    \item It seems possible to define a formalism of vanishing cycles in twisted situations, i.e. in the situation in which we have a morphism $s_X:X\rightarrow \V(\Lc_X)$ for $\Lc_X \in \textup{\textbf{Pic}}(X)$. This is all about completing the empty slot in the mind map above. One should then be able to find $\Phi^{\textup{mi}}_{(X,s_X)}(\Ql{}(\beta))$ via a procedure that corresponds to taking homotopy fixed points in the usual situation. A complete account on this formalism will appear in a forthcoming paper in collaboration with D.-C.~Cisinski.
    \item We will comment the regularity hypothesis that appears both in A.~Blanc - M.~Robalo - B.~To\"en - G.~Vezzosi's theorem and in the generalization we provide.
\end{enumerate}
\begin{rmk-non}
The main body of this article corresponds to \cite[Chapter \S3]{pi20}, while the "preliminaries" section corresponds Chapter \S1 in \textit{loc. cit}. The only difference lies in section \S \ref{the l-adic realization of the dg category of singularities of a twisted LG model of rank r}, where we have considered a more general setting in the present text.
\end{rmk-non}
\begin{rmk-non}
Even if the theorems are stated in the category $\shvl(Z)$ of $\ell$-adic sheaves, they are nevertheless true (and the proves we provide work mutatis mutandis) as $\textup{B}\mathbb{U}_{Z,\Q}$-modules (see section \ref{the l-adic realization of dg categories}).
\end{rmk-non}

\section{Preliminaries and notation}

In this preliminary section we will briefly recall some of the mathematical tools that we will need later. We will also fix the notation that we will use in the other sections.
\subsection{Some notation and convention}
\begin{itemize}
    \item Even when not explicitly stated, $S$ will always be a regular scheme of finite type over a strictly local noetherian scheme.
    \item $\sm$ denotes the category of smooth schemes of finite type over $S$.
    \item $\sch$ denotes the category of separated schemes of finite type over $S$.
    \item We will freely use the language of $\oo$-categories which has been developed in \cite{htt}, \cite{ha}. $\oo$-category will always mean $(\oo,1)$-category for us.
    \item $\mathcal{S}$ denotes the $\oo$-category of spaces.
    \item We write dg instead of "differential graded".
    \item We will use cohomological notations. In particular, the differential of a complex increases the degree.
    \item If we are given a morphism of (derived) schemes $f:X\rightarrow Y$ and an object $\Ec_Y \in \qcoh{Y}$, we will write $\Ec_X$ instead of $f^*\Ec_Y$.
\end{itemize}
\subsection{Reminders on dg categories}
\begin{rmk}
For more details on the theory of dg categories, we invite the reader to consult \cite{ke06}, \cite{to11} and/or \cite{ro14}.
\end{rmk}
Let $A$ be a commutative ring.

Let $\dgcat$ be the category of small $A$-linear dg categories together with $A$-linear dg functors. This category can be endowed with a  cofibrantly generated model category structure, where weak equivalences are DK equivalences (see \cite{tab05}). The underlying $\oo$-category of this model category coincides with the $\oo$-localization of $\dgcat$ with respect to the class of DK equivalences. We will denote this $\oo$-category by $\dgcatdk$. 

Every DK equivalence is a Morita equivalence. We can therefore endow $\dgcat$ with a second cofibrantly generated model category structure by using the theory of Bousfield localizations. In this case weak equivalences are Morita equivalences. Similarly to the previous case, the underlying $\oo$-category of this model category coincides with the $\oo$-localization of $\dgcat$ with respect to Morita equivalences. We will label this $\oo$-category by $\dgcatm$. 

Let  $\hat{\Cc}_c$ denote the dg category of perfect $\Cc^{\textup{op}}$-dg modules. Then $\dgcatm$ is equivalent to the full subcategory of $\dgcatdk$ spanned by dg categories $\Cc$ for which the Yoneda embedding $\Cc \hookrightarrow \hat{\Cc}_c$ is a DK equivalence.

To summarize, we have the following pair of composable $\oo$-localizations
\begin{equation}
\dgcat \rightarrow \dgcatdk \rightarrow \dgcatm.
\end{equation}
$\dgcatdk \rightarrow \dgcatm$ is a left adjoint to the inclusion $\dgcatm\hookrightarrow \dgcatdk$ under the identification mentioned above. At the level of objects, it is defined by the assignment $T\mapsto \hat{T}_c$.

It is possible to enhance $\dgcatdk$ and $\dgcatm$ with symmetric monoidal structures. Furthermore, if we restrict to the full subcategory $\dgcatlf \subseteq \dgcat$ of locally flat (small) dg categories, we get two composable symmetric monoidal $\oo$-functors
\begin{equation}
    \dgcatlfo \rightarrow \dgcatdko \rightarrow \dgcatmo.
\end{equation}
More details on the Morita theory of dg categories can be found in \cite{to07}.

One of the most recurrent operations that occur in this work is that of forming quotients of dg categories:
given a dg category $\Cc$ together with a full sub dg category $\Cc'$, both of them in $\dgcatm$, we can consider  the pushout $\Cc \amalg_{\Cc'}0$ in $\dgcatm$. Here $0$ denotes the final object in $\dgcatm$, i.e. the dg category with only one object whose endomorphisms are given by the zero hom-complex. We denote this pushout by $\Cc/\Cc'$ and refer to it as the dg quotient of $\Cc' \hookrightarrow \Cc$. Equivalently, The dg category $\Cc/\Cc'$ can also be obtained as the image in $\dgcatm$ of the pushout $\Cc \amalg_{\Cc'}0$ formed in $\dgcatdk$. Its homotopy category coincides with (the idempotent completion of) the Verdier quotient of $\textup{H}^0(\Cc)$ by the full subcategory $\textup{H}^0(\Cc')$ (see \cite{dri}). 

Compact objects in $\dgcatm$ are dg categories of finite type over $A$, as defined in \cite{tv07}. In particular,
\begin{equation}
    \textup{Ind}(\dgcatft)\simeq \dgcatm.
\end{equation}

Moreover, $\dgcatm$ is equivalent to the $\oo$-category of small, idempotent complete, $A$-linear stable $\oo$-categories (\cite{co13}).

When $S$ is a non affine scheme, the symmetric monoidal $\oo$-category of dg-cateogories over $S$ is defined as the limit
\begin{equation}
 \dgcatmo = \varprojlim_{Spec(A)\rightarrow S}\textup{\textbf{dgCat}}_A^{\textup{idm},\otimes}.
\end{equation}

Here there are some dg categories that we will use: let $X$ be a derived scheme (stack).
\begin{itemize}
    \item $\qcoh{X}$ will denote the dg category of quasi-coherent $\Oc_X$-modules. It can be defined as follows. If $X=Spec(B)$, then $\qcoh{X}=\md_B$, the dg category of dg modules over the dg algebra associated to $B$ via the Dold-Kan equivalence. In the general case, $\qcoh{X}=\varprojlim_{Spec(B)\rightarrow X}\md_B$ (the functoriality being that induced by base change).
    \item $\perf{X}$ will denote the full sub dg category of $\qcoh{X}$ spanned by perfect complexes. If $X=Spec(B)$, then $\perf{B}$ is the smallest subcategory of $\md_B$ which contains $B$ and that is stable under the formation of finite colimits and retracts. More generally, an object $\Ec \in \qcoh{X}$ is perfect if, for any $g:Spec(B)\rightarrow X$, the pullback $g^*\Ec \in \perf{B}$. Perfect complexes coincide with dualizable objects in $\qcoh{X}$ and, under some mild additional assumptions (that will always be verified in our examples), with compact objects (see \cite{bzfn}).
\end{itemize}
Moreover, as we assume to work in the noetherian setting, we can consider:
\begin{itemize}
    \item $\cohb{X}$ will denote the full sub dg category of $\qcoh{X}$ spanned by those cohomologically bounded complexes $\Ec$ (i.e. $\textup{H}^i(\Ec)\neq 0$ only for a finite number of indexes) such that $\textup{H}^*(\Ec)$ is a coherent $\textup{H}^0(\Oc_X)$-module.
    \item $\cohm{X}$ will denote the full dg category of $\qcoh{X}$ spanned by those cohomologically bounded above complexes $\Ec$ (i.e. $\textup{H}^i(\Ec)=0$ for $i>>0$) such that $\textup{H}^*(\Ec)$ is a coherent $\textup{H}^0(\Oc_X)$-modules. These are also known as pseudo-coherent complexes.
    \item Let $p : X\rightarrow Y$ be a proper morphism locally almost of finite type. By \cite[Cpt.4 Lemma 5.1.4]{gr17}, we have an induced $\oo$-functor $p_*:\cohb{X}\rightarrow \cohb{Y}$. We denote $\cohbp{X}{Y}$ the full subcategory of $\cohb{X}$ spanned by those objects $\Ec$ such that $p_*\Ec\in \perf{Y}$.
\end{itemize}
\subsection{Derived algebraic geometry}
Derived algebraic geometry is a broad generalization of algebraic geometry whose building blocks are simplicial commutative rings, rather then commutative rings. It is usually better behaved in the situations that are typically defined bad in the classical context, e.g. non-transversal intersections. The main idea is to develop algebraic geometry in an homotopical context: instead of saying that two elements are equal we rather say that they are homotopic, the homotopy being part of the data. In this article derived schemes appear exclusively as (homotopy) fiber products of ordinary schemes. It is necessary to allow certain schemes to be derived, as if we restrict ourselves to work with discrete (i.e. classical) schemes, some important characters do not appear (e.g. the algebra which acts on certain dg categories of (relative) singularities).

The ideas and motivations that led to derived algebraic geometry go back to J.-P.~Serre (Serre's intersection formula, \cite{se65}), P.~Deligne (algebraic geometry in a symmetric monoidal category, \cite{de90}), L.~Illusie, A.~Grothendieck, M.~Andr\'e, D.~Quillen (the cotangent complex, \cite{an74}, \cite{ill71}, \cite{qu70}, \cite{sga6}) $\dots$ however, the theory nowadays relies on solid roots thanks to the work of J.~Lurie (\cite{sag}) and B.~To\"en - G.~Vezzosi (\cite{tv05}, \cite{tv08}).
For a brief introduction to derived algebraic geometry we refer to \cite{to14}. We will denote the $\oo$-category of derived schemes by $\dsc$.

By definition, each derived scheme $X$ has an underlying scheme $t_0(X)$ (its truncation). Indeed, the assignment $X\mapsto t_0(X)$ is part of an adjunction 
\begin{equation}
  \iota: \textup{\textbf{Sch}} \rightleftarrows \dsc :  t_0 .
\end{equation}
A derived scheme is affine if its underlying scheme is so. There is an equivalence of $\oo$-categories
\begin{equation}
    \dscaff\simeq \sring^{\textup{op}} 
\end{equation}
where $\dscaff$ is the full subcategory of $\dsc$ spanned by affine derived schemes and $\sring$ is the $\oo$-category of simplicial rings. For any simplicial commutative ring $A$, we denote $Spec(A)$ the associated derived scheme.

The $\oo$-category $\dsc$ has all finite limits. In particular, it has fiber products. For example, if we consider a diagram
\begin{equation*}
    Spec(B)\rightarrow Spec(A) \leftarrow Spec(C)
\end{equation*}
of affine derived schemes, the fiber product is equivalent to $Spec(B\otimes^{\mathbb{L}}_AC)$, the spectrum of the derived tensor product. 

Consider two (underived) $S$-schemes $X,Y$ ($S$ being underived itself). Then the fiber product computed in $\dsc$ (denoted $X\times^h_SY$) might differ from the one computed in $\textup{\textbf{Sch}}$ (denoted $X\times_S Y$). However, they are related by the formula
\begin{equation}
    t_0(X\times_S^hY)\simeq X\times_SY.
\end{equation}

\subsection{\texorpdfstring{$\ell$}{l}-adic sheaves}\label{l-adic sheaves}
We shall briefly introduce the $\oo$-category of $\ell$-adic sheaves, following \cite{gl19}.

Fix a prime number $\ell$, which is invertible in each residue field of our base scheme $S$,
where $S$ is a regular scheme of finite type over a strictly local noetherian scheme\footnote{This assumption is needed in order to consider schemes of finite \'etale cohomological dimension, see \cite[Theorem 1.1.5]{cd16}.}. Let $\shv(X,\Z/\ell^d\Z)$ be the full subcategory of the $\oo$-category $\textup{Fun}(\textup{\textbf{Sch}}_{\textup{et}}^{op},\textup{\textbf{Mod}}_{\Z/\ell^d\Z})$ spanned by (hypercomplete) \'etale sheaves. Here $\textup{\textbf{Sch}}_{\textup{et}}$ denotes the category of \'etale sheaves and $\textup{\textbf{Mod}}_{\Z/\ell^d\Z}$ the $\oo$-category of $\Z/\ell^d\Z$-modules. 
The $\oo$-category $\shv(X,\Z/\ell^d\Z)$ is compactly-generated, its compact objects being constructible sheaves (see \cite[Proposition 3.38]{brtv}). In what follows, we will denote $\shv^c(X,\Z/\ell^d\Z)$ the full subcategory of $\shv(X,\Z/\ell^d\Z)$ spanned by compact objects.

The ring homomorphisms $\Z/\ell^d\Z \rightarrow \Z/\ell^{d-1}\Z $ induce a  sequence of $\oo$-functors
\begin{equation}
\shv(X,\Z/\ell^{d}\Z)\rightarrow \shv(X,\Z/\ell^{d-1}\Z)
\end{equation}
and it follows from \cite[Proposition 2.2.8.4]{gl19} that the image of a constructible sheaf is again constructible, yielding 
\begin{equation}
\shv^c(X,\Z/\ell^{d}\Z)\rightarrow \shv^c(X,\Z/\ell^{d-1}\Z).
\end{equation}
We can then consider the limit of the diagram of $\oo$-categories
\begin{equation}
\shv^c_{\ell}(X):=\varprojlim \shv^c(X,\Z/\ell^{d}\Z).
\end{equation}
This $\oo$-category can be identified with the full subcategory of $\shv(X,\Z)$ generated by $\ell$-complete constructible sheaves, i.e. by those objects $\mathcal{F} \in \shv(X,\Z)$ such that
\begin{enumerate}
\item $\mathcal{F}\simeq \varprojlim \mathcal{F}/\ell^d\mathcal{F}$,
\item for any $d\geq 1$, $\mathcal{F}/\ell^d\mathcal{F}$ is constructible. 
\end{enumerate}
We will refer to this $\oo$-category with the $\oo$-category of constructible $\ell$-adic sheaves.

The pushforward for a morphism $f : X\rightarrow Y$ of $S$-schemes of finite type induces an $\oo$-functors at the level of constructible $\ell$-adic
sheaves
\begin{equation}
f_* : \shv^c_{\ell}(X)\rightarrow \shv^c_{\ell}(Y).
\end{equation}
It admits a left adjoint
\begin{equation}
f^*:\shv^c_{\ell}(Y)\rightarrow \shv^c_{\ell}(X)
\end{equation}
that, at the level of objects, takes a constructible $\ell$-adic sheaf to the $\ell$-completion of its pullback.

We next consider the ind-completion of such categories:
\begin{equation}
\shv_{\ell}(X):= \textup{Ind}\bigl ( \shv^c_{\ell}(X) \bigr ),
\end{equation}
the $\oo$-category of $\ell$-adic sheaves. It is then a formal fact that we have a couple of adjoint functors, also called the pushforward and the pullback, defined at the level of $\ell$-adic sheaves.

Finally, we consider the localization of $\shv_{\ell}(X)$ with respect to the class of morphisms $\{ \mathcal{F}\rightarrow \mathcal{F}[\ell^{-1}] \}$, obtaining the $\oo$-category of $\mathbb{Q}_{\ell}$-adic sheaves $\shvl(X)$.

\subsection{Stable homotopy categories}
In this section we will briefly recall the constructions and main properties of $\sh$ and of $\shnc$, the stable $\oo$-category of schemes and the stable $\oo$-category of non-commutative spaces (a.k.a. dg categories).
\subsubsection{The stable homotopy category of schemes}\label{stable homotopy category of schemes}
The stable homotopy category of schemes was first introduced by F.~Morel and V.~Voevodsky in their celebrated paper \cite{mv99}. The main idea is to develop an homotopy theory for schemes, where the role of the unit interval - which is not available in the world of schemes - is played by the affine line. It was first developed using the language of model categories. We will rather use that of $\oo$-categories, following \cite{ro14} and \cite{ro15}. The two procedures are compatible, as shown in \cite{ro14} and \cite{ro15}.

Let $S=Spec(A)$ denote an affine scheme.
One can produce the unstable homotopy category of schemes as follows: one considers the $\oo$-category $\textup{Fun}(\sm^{\textup{op}},\mathcal{S})$ of presheaves (of spaces) on $\sm$. 
Its full subcategory spanned by Nisnevich sheaves, $\textup{Sh}_{Nis}(\sm)$ is an example of an $\oo$-topos (in the sense of \cite{htt}). 
Then one has to consider its hypercompletion $\textup{Sh}_{Nis}(\sm)^{hyp}$, which coincides with the localization of $\textup{Fun}(\sm^{\textup{op}},\mathcal{S})$ spanned by objects that are local with respect to Nisnevich hypercovers. 
If we further localize with respect to the projections $\{\aff{1}{X}\rightarrow X\}$, we obtain the unstable homotopy $\oo$-category of schemes, which we will denote by $\uH{S}$. 
It is also important that the canonical $\oo$-functor $\sm \rightarrow \uH{S}$ can be promoted to a symmetric monoidal $\oo$-functor with respect to the Cartesian structures.
One then considers the pointed version of $\uH{S}$, $\uH{S*}$: it comes equipped with a canonical symmetric monoidal structure $\uH{S*}^{\wedge}$ and there is a symmetric monoidal $\oo$-functor $\uH{S}^{\times}\rightarrow \uH{S*}^{\wedge}$.
The final step consists in stabilization: in classical stable homotopy theory one forces stabilization by inverting $S^1$. However, in this context, there exist two circles, the topological circle $S^1:=\Delta^1/\partial \Delta^1$ and the algebraic circle $\gm{S}$. One then stabilizes $\uH{S*}$ by inverting $S^1\wedge \gm{S}\simeq(\proj{1}{S},\oo):=cofib(S\xrightarrow{\oo}\proj{1}{S})$\footnote{clearly, this cofiber is taken in $\uH{S*}$.} - this can be done using the machinery developed in \cite[\S2.1]{ro15}. As a result, we obtain the presentable, symmetric monoidal, stable $\oo$-category $\sh^{\otimes}:=\uH{S*}^{\wedge}[(\proj{1}{S},\oo)^{-1}]$, called the stable homotopy $\oo$-category of schemes. It is moreover characterized by the following universal property (see \cite[Corollary 2.39]{ro15}): there is a symmetric monoidal $\oo$-functor $\Sigma_+^{\oo}:\sm^{\times}\rightarrow \sh^{\otimes}$ and for any presentable, symmetric monoidal pointed $\oo$-category $\mathcal{D}^{\otimes}$, the map
\begin{equation}
    \textup{Fun}^{\otimes,\textup{L}}(\sh^{\otimes},\mathcal{D}^{\otimes})\rightarrow \textup{Fun}^{\otimes}(\sm^{\times},\mathcal{D}^{\otimes})
\end{equation}
induced by $\Sigma_+^{\oo}$ is fully faithful and its image coincides with those symmetric monoidal $\oo$-functors $F:\sm^{\times}\rightarrow \mathcal{D}^{\otimes}$ that satisfy
\begin{itemize}
    \item Nisnevich descent,
    \item $\aff{1}{}$-invariance,
    \item $cofib(F(S\xrightarrow{\oo}\proj{1}{S}))$ is an invertible object in $\mathcal{D}$.
\end{itemize}
\begin{rmk}
It can be shown, using results of Ayoub (\cite{ay08i}, \cite{ay08ii}), Cisinski-D\'eglise (\cite{cd19}) and the machinery developed by Gaitsgory-Rozenblyum (\cite{gr17}) and Liu-Zheng (\cite{lz15}, \cite{lz17i}, \cite{lz17ii}), that the assignment $S\mapsto \sh$ defines a sheaf of $\oo$-categories enhanced with a Grothendieck $6$-functors formalism. See \cite{ro14} and \cite{ro15} and the appendix in \cite{brtv}.
\end{rmk}
Among the objects of $\sh$, the spectrum of homotopy invariant non-connective K-theory $\bu$\footnote{more commonly, this spectrum is denoted $KGL_S$, see \cite{ci13}. We use the notation $\bu_S$, which comes from topology, following the lead of \cite{brtv}, which is our main source of inspiration.} will play a crucial role in what follows. Recall that it is an object in $\sh$ such that, for every object $X\in \sm$,
\begin{equation}
    \Map_{\sh}(\Sigma^{\oo}_+X,\bu_S)\simeq \hk(X).
\end{equation}
Moreover, $\bu_S$ satisfies the algebraic Bott periodicity:
\begin{equation}
    \bu_S\simeq \bu_S(1)[2].
\end{equation}
\subsubsection{The stable homotopy category of non-commutative spaces}
\begin{rmk}
This section is an exposition of the ideas and results of \cite{ro14} and \cite{ro15}. However, a theory of non-commutative motives was also proposed by D.-C.~Cisinski and G.~Tabuada in \cite{ct11}, \cite{ct12}, \cite{ta08}. Their theory is dual to the one constructed by M.~Robalo. For more details about this, we refer to \cite[Appendix A]{ro15} and \cite[Remark 3.4]{brtv}.
\end{rmk}
It is possible to mimic the procedure described in the previous paragraph to construct a presentable, symmetric monoidal, stable $\oo$-category $\shnc$ defined starting with non-commutative spaces rather than schemes, a.k.a. dg categories. The role of smooth schemes is played by dg categories of finite type in this context. Moreover, there exists an analog of Nisnevich square of non-commutative spaces (see \cite{ro14} and \cite{ro15}). 
\begin{rmk}\label{convention emptyset}
By convention, $\emptyset$ is a Nisnevich covering of the zero object in $\dgcat^{\textup{ft}}$.
\end{rmk}
\begin{rmk}
The notion on Nisnevich squares for smooth non-commutative spaces is compatible with the classical one: if
\begin{equation}
    \begindc{\commdiag}[18]
    \obj(-15,10)[1]{$V\times_X U$}
    \obj(15,10)[2]{$V$}
    \obj(-15,-10)[3]{$U$}
    \obj(15,-10)[4]{$X$}
    \mor{1}{2}{$$}
    \mor{1}{3}{$$}
    \mor{2}{4}{$$}
    \mor{3}{4}{$$}
    \enddc
\end{equation}
is an elementary Nisnevich square, then its image via $\perf{\bullet}$ is a Nisnevich square of non-commutative spaces (\cite[Proposition 3.21]{ro15}).
\end{rmk}
One then considers the category of presheaves $\textup{Fun}(\nc{S}^{\textup{op}},\mathcal{S})$ and its localization with respect to the class of morphisms $\{j(U)\amalg_{j(X)}j(V)\rightarrow j(W)\}$ determined by Nisnevich squares of non-commutative spaces, where $j:\nc{S} \rightarrow \textup{Fun}(\nc{S}^{\textup{op}},\mathcal{S})$ is the Yoneda embedding. We further localize with respect to the morphisms $\{X\otimes \perf{S}\rightarrow X\otimes \perf{\aff{1}{S}}\}$\footnote{Since the Yoneda embedding is symmetric monoidal (when we consider the Day convolution product on the $\oo$-category of presheaves), then it suffices to localize with respect to the morphism $j(\perf{S})\rightarrow j(\perf{\aff{1}{S}})$.} and obtain a presentable symmetric monoidal $\oo$-category $\uH{S}^{\textup{nc},\otimes}$. Pursuing the analogy with the commutative case, we should now force the existence of a zero object in $\uH{S}^{\textup{nc}}$ and then stabilize with respect to the topological and algebraic circles. It comes out that the situation is simpler in the non-commutative case: let 
\begin{equation*}
    \psi^{\otimes}:\nc{S}^{\otimes}\rightarrow \uH{S}^{\textup{nc},\otimes}\rightarrow \shncm:= \uH{S}^{\textup{nc},\otimes}[(S^1)^{-1}]
\end{equation*} 
denote the $\oo$-functor that we obtain if we force the topological circle to be invertible. Then 
\begin{itemize}
    \item $\shnc$ is pointed due to the convention of Remark \ref{convention emptyset}.
    \item The non-commutative motive 
    \begin{equation*}
        \psi(\proj{1}{S},\oo)=cofib(\psi(\perf{S} \xrightarrow{\oo^*} \perf{\proj{1}{S}}))
    \end{equation*} 
    is invertible in $\shncm$ (\cite[Proposition 3.24]{ro15}): indeed, it is equivalent to $\perf{S}\simeq 1^{\textup{nc}}$ and therefore it is not necessary to force $\gm{S}$ to be invertible.
\end{itemize}
\subsection{The bridge between motives and non-commutative motives}\label{the bridge between motives and non commutative motives}
We have now at our disposal the following picture:
\begin{equation}
    \begindc{\commdiag}[18]
    \obj(-20,20)[1]{$\sm^{\times}$}
    \obj(20,20)[2]{$\nc{S}^{\otimes}$}
    \obj(-20,-20)[3]{$\sh^{\otimes}$}
    \obj(20,-20)[4]{$\shncm$}
    \mor{1}{2}{$\perf{\bullet}$}
    \mor{1}{3}{$\Sigma^{\oo}_+$}[\atright,\solidarrow]
    \mor{2}{4}{$\psi^{\otimes}$}
    \mor{3}{4}{$\rperf$}[\atright,\dashArrow]
    \enddc
\end{equation}
The existence of the dotted map, which we name perfect realization, is granted by the universal property of $\sh^{\otimes}$. Indeed
\begin{itemize}
    \item $\psi \circ \perf{\bullet}$ sends (ordinary) Nisnevich squares to pushout squares in $\shnc$: this is a consequence of the compatibility of non-commutative Nisnevich squares with the classical ones and of the definition of $\shnc$,
    \item $\aff{1}{}$-invariance is forced by construction,
    \item $\psi(\proj{1}{S},\oo)$ is an invertible object in $\shncm$.
\end{itemize}
The fact that $\rperf$ commutes with colimits is also guaranteed by the universal property of $\sh^{\otimes}$. As both $\sh$ and $\shnc$ are presentable $\oo$-categories, the adjoint functor theorem \cite[Corollary 5.5.2.9]{htt} implies the existence of a lax-monoidal right adjoint
\begin{equation}
    \mathcal{M}^{\otimes}:\shncm \rightarrow \sh^{\otimes}.
\end{equation}
For our purposes, it will be very important the following result
\begin{thm}{\textup{\cite[Corollary 4.10]{ro15}}}
The image of the monoidal unit $1^{\textup{nc}}_S$ via $\mathcal{M}$ is equivalent to $\bu_S$.
\end{thm}
In particular, the $\oo$-functor $\mathcal{M}$ factors trough the full subcategory of $\sh$ spanned by modules over $\bu_S$:
\begin{equation}
    \mathcal{M}^{\otimes}: \shncm\rightarrow \md_{\bu_S}(\sh)^{\otimes}.
\end{equation}
In \cite{brtv}, the authors introduced a dual version of this $\oo$-functor.
Consider the endomorphism of $\shnc$ induced by the internal hom:
\begin{equation}
    \rhom_{\shnc}(-,1^{\textup{nc}}_S):\shncmop\rightarrow \shncm.
\end{equation}
Then consider
\begin{equation}
    \begindc{\commdiag}[14]
    \obj(-50,20)[1]{$\dgcat^{\textup{ft},\otimes}$}
    \obj(0,20)[2]{$\shncmop$}
    \obj(90,20)[3]{$\shncm$}
    \obj(140,20)[4]{$\md_{\bu_S}(\sh^{\otimes})$}
    \obj(-50,-20)[5]{$\dgcatmo$}
    \mor{1}{2}{$\psi^{\otimes}$}
    \mor{2}{3}{$\rhom_{\shnc}(-,1^{\textup{nc}_S})$}
    \mor{3}{4}{$\mathcal{M}^{\otimes}$}
    \mor{1}{5}{$$}[\atright,\injectionarrow]
    
    \mor{5}{2}{$\psi^{\otimes}$}[\atright,\solidarrow]
    \cmor((-30,-20)(80,-20)(100,-20)(130,-12)(140,10)) \pup(50,-24){$\mathcal{M}^{\vee,\otimes}$}
    \enddc
\end{equation}
where the vertical map on the left is given by the inclusion of $\dgcat^{\textup{ft},\otimes}$ in its Ind-completion $\dgcatmo$ and the oblique $\psi^{\otimes}$ is induced by the universal property of the Ind-completion.
Let $T \in \dgcatm$. Then $\Mv(T)$ is the sheaf of spectra $X \in \sm \mapsto \hk(\perf{X}\otimes_{S} T)$:
\begin{equation}
    \Mv(T)(X)=\Map_{\sh}(\Sigma^{\oo}_+X,\Mv(T))\simeq
\end{equation}
\begin{equation*}
    \Map_{\shnc}(\rperf(\Sigma^{\oo}_+X),\rhom_{\shnc}(T,1^{\textup{nc}}_S))
\end{equation*}
\begin{equation*}
    \simeq \Map_{\shnc}(\rperf(\Sigma^{\oo}_+X)\otimes_{S} T,1^{\textup{nc}}_S)\simeq \hk(\perf{X}\otimes_{S} T).
\end{equation*}
Moreover, $\Mv$ has the following nice properties
\begin{itemize}
    \item it is lax monoidal (it is a composition of lax monoidal $\oo$-functors),
    \item it commutes with filtered colimits (see \cite[Remark 3.4]{brtv}),
    \item it sends exact sequences of dg categories to fiber-cofiber sequences in $\md_{\bu_S}(\sh)^{\otimes}$ (see \cite[Corollary 3.3]{brtv}).
\end{itemize}
We will refer to $\Mv$ as the motivic realization of dg categories.
\subsection{\texorpdfstring{$\ell$}{l}-adic realization of dg categories}\label{the l-adic realization of dg categories}
We will need a way to associate an $\ell$-adic sheaf to a dg category. We follow the construction given in \cite[\S3.6, \S3.7]{brtv}, which relies on results of J.~Ayoub and D.-C.~Cisinski - F.~D\'eglise.
Let $\HQ$ be the Eilenberg-MacLane spectrum of rational homotopy theory. Then we get
\begin{equation}
    -\otimes \HQ:\sh^{\otimes}\rightarrow \md_{\HQ}(\sh)^{\otimes},
\end{equation}
which identifies the $\oo$-category on the right hand side with non-torsion objects of $\sh$.
Similarly, if one puts $\bu_{S,\Q}:=\HQ\otimes \bu_S$, then one gets
\begin{equation}
    -\otimes \HQ:\md_{\bu_S}(\sh)^{\otimes}\rightarrow \md_{\bu_{S,\Q}}(\sh)^{\otimes},
\end{equation}
where the right hand side identifies with non-torsion $\bu_S$-modules. This $\oo$-functor is strongly compatible with the $6$-functors formalism and Tate twists.
Moreover, we state for future reference the following crucial fact:
\begin{thm}{\textup{\cite[Theorem 5.3.10]{ri10}, \cite[\S14.1]{cd19}, \cite[Proposition 3.35]{brtv}}}\label{thm riou}
Let $X$ be a scheme of finite Krull dimension an let $\MB{X}\in \textup{CAlg}(\textup{\textbf{SH}}_X)$ denote the spectrum of Beilinson motivic cohomology. Then the canonical morphism $1_X(1)[2]\rightarrow \bu_X\otimes \HQ=\bu_{X,\Q}$ induces an equivalence of commutative algebra objects
\begin{equation}
    \MB{X}(\beta):=Sym(\MB{X}(1)[2])[\nu^{-1}]\simeq \bu_{X,\Q}=\bu_X\otimes \HQ,
\end{equation}
where $\nu$ is the free generator in degree $(1)[2]$.
\end{thm}
By using the theory of h-motives developed by the authors in \cite{cd16} one can define an $\ell$-adic realization $\oo$-functor
\begin{equation}\label{l-adic realization from beilinson motives}
    \Rl: \md_{\MB{}}(\textup{\textbf{SH}})\rightarrow \shvl(-)
\end{equation}
strongly compatible with the $6$-functors formalism and Tate twists, at least for noetherian schemes of finite Krull dimension.
Then, using the equivalence
\begin{equation}
    \md_{\bu}(\md_{\MB{}}(\textup{\textbf{SH}}))\simeq \md_{\bu}(\textup{\textbf{SH}})
\end{equation}
we obtain an $\ell$-adic realization $\oo$-functor
\begin{equation}
\begindc{\commdiag}[16]
    \obj(-70,5)[1]{$\Rl: \md_{\bu}(\textup{\textbf{SH}})$}
    \obj(0,5)[2]{$\md_{\bu_{\Q}}(\textup{\textbf{SH}})$}
    \obj(70,5)[3]{$\md_{\Rl(\bu)}(\shvl(-))$} 
    \obj(70,-5)[4]{$\simeq \md_{\Ql{}(\beta)}(\shvl(-))$}
    \mor{1}{2}{$-\otimes \HQ$}
    \mor{2}{3}{$$}
\enddc
\end{equation}

strongly compatible with the $6$-functors formalism and Tate twists. The last equivalence above holds as (\ref{l-adic realization from beilinson motives}) is symmetric monoidal and commutes with Tate twists (see \cite[Remark 3.43]{brtv}):
\begin{equation}
    \Rl_X(\bu_X)\underbracket{\simeq}_{\textbf{Theorem }\ref{thm riou}}\Rl_X(Sym(\MB{X}(1)[2])[\nu^{-1}])
    \end{equation}
\begin{equation*}
\simeq Sym(\Rl_X(\MB{X})(1)[2])[\nu^{-1}]
\simeq \Ql{,X}(\beta)=Sym(\Ql{,X}(1)[2])[\nu^{-1}].
\end{equation*}
We will use the notation
\begin{equation}
    \Rlv_S:=\Rl_S \circ \Mv_S : \dgcatmo \rightarrow \md_{\Ql{,S}(\beta)}(\shvl(S))
\end{equation}
and refer to this $\oo$-functor as the $\ell$-adic realization of dg categories.
\section{Motivic realization of twisted LG models}
\begin{notation}\label{context non affine base}
Let $S$ be a noetherian (not necessarely affine) regular scheme. We will label $\sch$ the category of separated $S$-schemes of finite type.
\end{notation}
\subsection{The category of twisted LG models}\label{the category of twisted LG models}
Consider the category $\sch$ and let $\Lc_S$ be a line bundle on $S$. Then we can consider the category of Landau-Ginzburg models over $(S,\Lc_S)$, defined as follows:
\begin{itemize}
    \item Objects consists of pairs $(X,s_X)$, where $p:X\rightarrow S$ is a (flat) $S$-scheme and $s_X$ is a section of $\Lc_X:=p^*\Lc_S$.
    \item Morphisms $(X,s_X)\xrightarrow{f}(Y,s_Y)$ consist of morphisms $f:X\rightarrow Y$ in $\sch$ such that $s_X=f^*s_Y$.
    \item Composition and identity morphisms are clear.
\end{itemize}
We will denote this category by $\tlgm{S}$.
\begin{rmk}\label{tlgm trivial line bundle}
If $\Lc_S$ is the trivial line bundle, then $\tlgm{S}$ coincides with the category of usual Landau-Ginzburg models over $S$ (as defined for example in \cite{brtv}). Indeed, in this case $\V(\Lc_S)=Spec_{\Oc_S}(\Oc_S[t])$. Therefore, for any $X\in \sch$, a section of $\Lc_X$ consists of a morphism $\Oc_X[t]\rightarrow \Oc_X$, i.e. of a global section of $\Oc_X$.
\end{rmk}
\begin{construction}
It is possible to endow the category $\tlgm{S}$ with a symmetric monoidal structure 
\begin{equation}
    \boxplus: \tlgm{S}\times \tlgm{S}\rightarrow \tlgm{S}
\end{equation}
\[
\bigl ((X,s_X),(Y,s_Y) \bigr )\mapsto (X,s_X)\boxplus (Y,s_Y)=(X\times_S Y,s_X\boxplus s_Y),
\]
where $s_x\boxplus s_Y=p_X^*s_X+p_Y^*s_Y$. This is clearly (weakly) associative and the unit is given by $(S,0)$, where $0$ stands for the zero section of $\Lc_S$.

Notice that this tensor product is induced by the abelian group structure on $\V(\Lc_S)=Spec_S(Sym_{\Oc_S}(\Lc_S^{\vee}))$.
\end{construction}
\begin{rmk}
If $\Lc_S$ is the trivial line bundle, then $\boxplus$ coincides with the tensor product on the category of LG modules over $S$.
\end{rmk}
We will now exhibit twisted LG models as a fibered category.
\begin{defn}
Let $\fclgm$ be the category defined as follows:
\begin{itemize}
    \item Objects are triplets $(f:Y\rightarrow X,\Lc_X,s_Y)$ where $f$ is a flat morphism between $S$-schemes, $\Lc_X$ is a line bundle on $X$ and $s_Y$ is a global section of $f^*\Lc_X$.
    \item Given two objects $(f_i:Y_i\rightarrow X_i,\Lc_{X_i},s_{Y_i})$ ($i=1,2$), a morphism from the first to the second is the datum of a commutative diagram
    \begin{equation}
        \begindc{\commdiag}[16]
        \obj(-15,15)[1]{$Y_1$}
        \obj(15,15)[2]{$X_1$}
        \obj(-15,-15)[3]{$Y_2$}
        \obj(15,-15)[4]{$X_2$}
        \mor{1}{2}{$f_1$}
        \mor{1}{3}{$g_Y$}
        \mor{2}{4}{$g_X$}
        \mor{3}{4}{$f_2$}
        \enddc
    \end{equation}
    and of an isomorphism $\alpha: g_X^*\Lc_{X_2}\rightarrow \Lc_{X_1}$ such that $s_{Y_1}$ corresponds to $s_{Y_2}$ under the isomorphism
    \begin{equation}
        g_{Y}^*f_2^*\Lc{X_2}\simeq f_1^*g_{X}^*\Lc_{X_2}\underbracket{\simeq}_{\alpha} f_1^*\Lc_{X_1}.
    \end{equation}
    By abuse of notation, we will say in the future that $g_Y^*(s_{Y_2})=s_{Y_1}$ if this condition is satisfied. We will denote such a morphism by $(g,\alpha)$.
    \item Composition and identities are defined in an obvious way.
\end{itemize}
We will refer to this category as the category of twisted Landau-Ginzburg models of rank $1$ over $(S,\Lc_S)$ (twisted LG models for short).
\end{defn}
Notice that there is a functor
\begin{equation}\label{fibration LG}
    \pi:\fclgm \rightarrow \int_{X\in \sch} \bgm{S}(X)
\end{equation}
defined as
\begin{equation*}
    \begindc{\commdiag}[20]
    \obj(0,15)[1]{$(f_1:Y_1\rightarrow X_1,\Lc_{X_1},s_{Y_1})$}
    \obj(60,15)[2]{$(X_1,\Lc_{X_1})$}
    \obj(0,-15)[3]{$(f_2:Y_2\rightarrow X_2,\Lc_{X_2},s_{Y_2})$}
    \obj(60,-15)[4]{$(X_2,\Lc_{X_2})$}
    \mor{1}{3}{$(g,\alpha)$}
    \mor{2}{4}{$(g_X,\alpha)$}
    \obj(38,0)[5]{$\mapsto$}
    \enddc
\end{equation*}
\begin{lmm}
The functor $(\ref{fibration LG})$ exhibits $\fclgm$ as a fibered category over \newline $\int_{X\in \sch} \bgm{S}(X)$.
\end{lmm}
\begin{proof}
Consider a map $(g_X,\alpha):(X_1,\Lc_{X_1})\rightarrow (X_2,\Lc_{X_2})$ in $\int_{X\in \sch} \bgm{S}(X)$ and let $(f_2:Y_2\rightarrow X_2,\Lc_{X_2},s_{Y_2})$ be an object of $\fclgm$ over $(X_2,\Lc_{X_2})$. Consider the morphism 
\[
(g,\alpha):(f_1:Y_1:=X_1\times_{X_2}Y_2\rightarrow X_1,\Lc_{X_1},s_{Y_1})\rightarrow (f_2:Y_2\rightarrow X_2,\Lc_{X_2},s_{Y_2})
\]
where $g_Y$ is the projection $Y_1\rightarrow Y_2$ (which is flat as it is the pullback of a flat morphism), $f_1$ is the projection $Y_1\rightarrow Y_2$ and $s_{Y_1}$ is $g_Y^*(s_{Y_2})$. It is clear that it is a morphism of $\fclgm$ over $(g_X,\alpha)$. We need to show that it is cartesian. Consider
\begin{equation}
    \begindc{\commdiag}[20]
    \obj(-10,15)[1]{$(f_1:Y_1\rightarrow X_1,\Lc_{X_1},s_{Y_1})$}
    \obj(-10,-15)[2]{$(X_1,\Lc_{X_1})$}
    \obj(65,15)[3]{$(f_2:Y_2\rightarrow X_2,\Lc_{X_2},s_{Y_2})$}
    \obj(65,-15)[4]{$(X_2,\Lc_{X_2})$}
    \mor{1}{3}{$(g,\alpha)$}[\atright,\solidarrow]
    \mor{2}{4}{$(g_X,\alpha)$}
    \mor{1}{2}{$$}
    
    \mor{3}{4}{$$}
    \obj(-55,45)[5]{$(h:Z\rightarrow W,\Lc_{W},s_{Z})$}
    
    \obj(-55,0)[6]{$(W,\Lc_{W})$}
    
    \mor{5}{6}{$$}
    
    \mor{6}{2}{$(p,\beta)$}
    \mor{5}{3}{$(q,\gamma)$}
    \mor(-40,40)(-5,20){$((r,p),\beta)$}[\atright,\dashArrow]
    \enddc
\end{equation}
Then the universal property of $Y_1$ gives us an unique morphism $r:Z\rightarrow Y_1$ such that the compositions with $f_1$ and $g_Y$ are $p\circ h$ and $q$ respectively. We just need to show that $r^*(s_{Y_1})=s_Z$. But this is clear since $s_{Y_1}=g_Y^*(s_{Y_2})$, $q_Y^*(s_{Y_2})=s_Z$ and $g_Y\circ r=q_Y$.
\end{proof}
\begin{rmk}
Let $(X,\Lc_X)$ be an object of $\int_{X\in \sch} \bgm{S}(X)$. Then the fiber of $(X,\Lc_X)$ along $\pi$ is $\tlgm{X}$.
\end{rmk}
\begin{defn}
We say that a collection of maps $\{(g_i,\alpha_i):(U_i,\Lc_{U_i})\rightarrow (X,\Lc_{X})\}_{i\in I}$ in $\int_{X\in \sch} \bgm{S}(X)$ is a Zariski covering if $\{g_i:U_i\rightarrow X\}_{i\in I}$ is so. They clearly define a pre-topology on $\int_{X\in \sch} \bgm{S}(X)$\footnote{Notice that pullbacks exists in $\int_{X\in \sch} \bgm{S}(X)$: the pullback of $(Y_1,\Lc_{Y_1})\xrightarrow{(f_1,\alpha_1)}(X,\Lc_X)\xleftarrow{(f_2,\alpha_2)}(Y_2,\Lc_{Y_2})$ is $(Y_1\times_X Y_2,\Lc_{Y_1\times_X Y_2})$ with the projections defined in an obvious way}. We will refer to the corresponding topology by Zarisky topology on $\int_{X\in \sch} \bgm{S}(X)$. 
\end{defn}
\begin{lmm}
$\fclgm$ is a stack over $\int_{X\in \sch} \bgm{S}(X)$ endowed with the Zariski topology.
\end{lmm}
\begin{proof}
This is a simple consequence of the fact that a morphism of schemes is uniquely determined by its restriction to a Zariski covering and that line bundles are Zariski sheaves.
\end{proof}
We conclude this section with the following observation:
\begin{lmm}\label{zariski descent tlgmm}
Let $\{(g_i,\alpha_i):(U_i,\Lc_{U_i})\rightarrow (X,\Lc_{X})\}_{i\in I}$ be a Zariski covering in $\int_{X\in \sch} \bgm{S}(X)$. Then the canonical functor
\begin{equation}\label{compatibility limits with tlgmo}
    \tlgm{X}^{\boxplus}\rightarrow \varprojlim \tlgm{U_i}^{\boxplus}
\end{equation}
is a symmetric monoidal equivalence.
\end{lmm}
\begin{proof}
Consider the functors
\begin{equation}
    \tlgm{X}\rightarrow \tlgm{U_i}
\end{equation}
\begin{equation*}
    (f:Y\rightarrow X,s_Y)\mapsto (f_i:Y\times_X U_i\rightarrow U_i,s_{Y|U_i}).
\end{equation*}
It is easy to see that they respect the tensor structure. Therefore, we get the desired symmetric monoidal functor $(\ref{compatibility limits with tlgmo})$ in the (big) category of symmetric monoidal categories and symmetric monoidal functors. Then, in order to prove that is a symmetric monoidal equivalence, it suffices to show that the underlying functor 
\begin{equation}
     Forget(\tlgm{X}^{\boxplus}\rightarrow \varprojlim \tlgm{U_i}^{\boxplus})\simeq (\tlgm{X}\rightarrow Forget(\varprojlim \tlgm{U_i}^{\boxplus})).
\end{equation}
As the functor which forgets the symmetric monoidal structure of a category is a right adjoint, it preserves limits. Thus,
\begin{equation}
    Forget(\varprojlim \tlgm{U_i}^{\boxplus})\simeq \varprojlim \tlgm{U_i}.
\end{equation}
The assertion now follows from the previous lemma.
\end{proof}
\subsection{The dg category of singularities of a twisted LG model}\label{The dg category of singularities of a twisted LG model}
Let $(X,s_X)$ be a twisted LG model over $(S,\Lc_S)$. The section $s_X$ defines a closed sub-scheme of $X$. Since we are not assuming that the section is regular, some derived structure might appear. More precisely: let $\Oc_X\rightarrow \Lc_X$ be the morphism of $\Oc_X$-modules associated to $s_X$. Then, taking duals, it defines a morphism $\Lc_X^{\vee}\rightarrow \Oc_X$. If we apply the relative spectrum functor, we get a morphisms $Spec_{\Oc_X}(\Oc_X)=X\rightarrow Spec_{\Oc_X}(Sym_{\Oc_X}(\Lc_X^{\vee}))=\V(\Lc_X)$, i.e. a section of the vector bundle associated to $\Lc_X$. By abuse of notation, we will label this morphism by $s_X$. We can also consider the zero section $X\rightarrow \V(\Lc_X)$, which is the morphism associated to $0\in \Lc_X(X)$. 
\begin{defn}\label{derived zero locus global section line bundle}
The derived zero locus of $s_X$ is defined as the derived pullback
\begin{equation}\label{pullback square derived zero locus}
    \begindc{\commdiag}[20]
    \obj(0,15)[1]{$X_0$}
    \obj(30,15)[2]{$X$}
    \obj(0,-15)[3]{$X$}
    \obj(30,-15)[4]{$\V(\Lc_X).$}
    \mor{1}{2}{$\mathfrak{i}$}
    \mor{1}{3}{$$}
    
    \mor{2}{4}{$s_X$}
    \mor{3}{4}{$0$}
    \enddc
\end{equation}

\end{defn}
\begin{rmk}
Notice that the zero section $X\xrightarrow{0}\V(\Lc_X)$ is an lci morphism (Zariski locally, it is just the zero section a scheme $Y$ in the affine line $\aff{1}{Y}$). It is well known that this class of morphisms is stable under derived pullbacks. In particular, $\mathfrak{i}:X_0\rightarrow X$ is a derived lci morphism.
\end{rmk}
\begin{rmk}
There is a truncation morphism $\mathfrak{t}:\pi_0(X_0)\rightarrow X$, where $\pi_0(X_0)$ is the classical scheme associated to $X_0$. Whenever $s_X$ is a regular section, the truncation morphism is an equivalence in the $\oo$-category of derived schemes.
\end{rmk}

\begin{rmk}
Of major importance for our purposes is that $\mathfrak{i}$ is an lci morphism. Indeed, by \cite{to12}, if $f:Y\rightarrow Z$ is an lci morphism of derived schemes and $E\in \perf{Y}$, then $f_*E\in \perf{Z}$. In particular, we get that $\perf{X_0}$ is a full subcategory of $\cohbp{X_0}{X}$, the full subcategory of $\cohb{X_0}$ spanned by those objects $E \in \cohb{X_0}$ such that $f_*E\in \perf{X}$.
\end{rmk}

The dg category of absolute singularities of a derived scheme is a non-commutative invariant (in the sense of Kontsevich) which captures the singularities of the scheme. 
\begin{defn}
Let $Z$ be a derived scheme of finite type over $S$ whose structure sheaf is cohomologically bounded. Then the dg category of absolute singularities is the dg quotient in $\dgcatm$
\begin{equation}
    \sing{Z}:=\cohb{Z}/\perf{Z}.
\end{equation}
\end{defn}
\begin{rmk}
It is a classical theorem due to M.~Auslander - D.~A.~Buchsbaum (\cite[Theorem 4.1]{ab56}) and J.-P.~Serre (\cite[Th\'eor\`eme 3]{se55}) that a noetherian ring $A$ is regular if and only if it has finite global dimension. This means that the category of perfect complexes $\perf{A}$ coincides with that of coherent bounded complexes $\cohb{A}$, i.e. $\sing{A}\simeq 0$. This explains the name of this object.
\end{rmk}
\begin{rmk}
Notice that the hypothesis on $Z$ are crucial, as in general the structure sheaf of a derived scheme might not be cohomologically bounded. As an example of a derived scheme that doesn't sit in this class of objects, consider 
\[
Spec(\C[x]\otimes^{\mathbb{L}}_{\C[x,y]/(xy)}\C[y])\simeq Spec(\C[x])\times^h_{Spec(\C[x,y]/(xy))}Spec(\C[y]).
\]
By tensoring the projective $\C[x,y]/(xy)$-resolution of $\C[x  ]$
\[
\dots \xrightarrow{y} \C[x,y]/(xy)\xrightarrow{x}\C[x,y]/(xy)\xrightarrow{y}\C[x,y]/(xy)\rightarrow 0
\]
with $\C[y]$ over $\C[x,y]/(xy)$ we get that this derived scheme is the spectrum of the cdga
\[
\dots \xrightarrow{y} \C[y]\xrightarrow{0}\C[y]\xrightarrow{y}\C[y]\rightarrow 0,
\]
which has nontrivial cohomology in every even negative degree.
\end{rmk}
Since $\mathfrak{i}$ is an lci morphism of derived schemes, by \cite{to12} and by \cite{gr17} the pushforward induces a morphism 
\begin{equation}\label{push sing}
\mathfrak{i}_*: \sing{X_0}\rightarrow \sing{X}.
\end{equation}
\begin{defn}
Let $(X,s_X)$ be a twisted LG model over $(S,\Lc_S)$. Its dg category of singularities is defined as the fiber in $\dgcatm$ of $(\ref{push sing})$:
\begin{equation}
    \sing{X,s_X}:=fiber\bigl(\mathfrak{i}_*:\sing{X_0}\rightarrow \sing{X} \bigr).
\end{equation}
\end{defn}
\begin{rmk}
It is a consequence of the functoriality properties of the pullback and that of taking quotients and fibers in $\dgcatm$ that the assignments $(X,s_X)\mapsto \sing{X,s_X}$ can be organized in an $\oo$-functor
\begin{equation}\label{functor sing tlgm}
    \sing{\bullet,\bullet}:\tlgm{S}^{\textup{op}}\rightarrow \dgcatm.
\end{equation}
\end{rmk}
We will need to endow the $\oo$-functor $(\ref{functor sing tlgm})$ with the structure of a lax monoidal $\oo$-functor. In order to do so, we will introduce a strict model for $\cohbp{X_0}{X}$, following the strategy exploited in \cite{brtv}.
\begin{construction}\label{strict model for the derived zero locus} Let $(X,s_X)$ be a twisted LG model over $(S,\Lc_S)$. The identity morphism on $\Lc^{\vee}_X$ induces a $Sym_{\Oc_X}(\Lc^{\vee}_X)$-projective resolution of $\Oc_X$, namely
\begin{equation}
    \Lc^{\vee}_X\otimes_{\Oc_X}Sym_{\Oc_X}(\Lc^{\vee}_X)\rightarrow Sym_{\Oc_X}(\Lc^{\vee}_X).
\end{equation}
This is a morphism of $Sym_{\Oc_X}(\Lc^{\vee}_X)$-modules (induced by the multiplication of $Sym_{\Oc_X}(\Lc_X^{\vee})$).
Base changing along the morphism $Sym_{\Oc_X}(\Lc^{\vee}_X)\rightarrow \Oc_X$ induced by $s_X$ we obtain that the cdga associated to $X_0$ is the spectrum of $\Lc^{\vee}_X\xrightarrow{s_X} \Oc_X$.
\end{construction}
\begin{ex}\label{derived zero locus zero section}
Consider the zero section of $\Lc_X$. Then the cdga associated to the structure sheaf of $X\times^h_{\V(\Lc_X)}X$ is the Koszul algebra in degrees $[-1,0]$ (recall that we use cohomological conventions)
\begin{equation}
    \Lc^{\vee}_X\xrightarrow{0}\Oc_X.
\end{equation}
\end{ex}
\begin{construction}\label{cohs section line bundle}
Let $S=Spec(A)$ be a noetherian, regular affine scheme and let $L$ be a projective $A$-module of rank $1$. Let $\tlgm{S}^{\textup{aff}}$ be the category of affine twisted LG models over $(S,L)$, i.e. of pairs $(Spec(B),s)$ where $B$ is a flat $A$-algebra of finite type and $s\in L_B=L\otimes_A B$. For any object $(X,s)=(Spec(B),s)$, the derived zero locus of $s$ is given by the spectrum of the Koszul algebra $K(B,s)$. The associated cdga is
\begin{equation*}
    K(B,s)=L_B^{\vee}\xrightarrow{s}B
\end{equation*}
concentrated in degrees $-1$ and $0$, where we label $s:L_B^{\vee}\rightarrow B$ the morphism induced by $s:B\rightarrow L_B$. Then, similarly to \cite[Remark 2.30]{brtv}, there is an equivalence between cofibrant $K(B,s)$-dg modules (denoted $\widehat{K(B,s)}$) and $\qcoh{X_0}$. Under this equivalence, $\cohb{X_0}$ corresponds to the full sub-dg category of $\widehat{K(B,s)}$ spanned by those cohomologically bounded dg modules with coherent cohomology (over $B/s=coker(L_B^{\vee}\xrightarrow{s}B)$) and $\perf{X_0}$ corresponds to homotopically finitely presented $K(B,s)$-dg modules. Moreover, the pushforward
\begin{equation*}
\mathfrak{i}_*:\qcoh{X_0}\rightarrow \qcoh{X}
\end{equation*}
corresponds to the forgetful functor
\begin{equation*}
    \widehat{K(B,s)}\rightarrow \widehat{B}.
\end{equation*}

Define $\cohs{B,s}$ as the dg category of $K(B,s)$-dg modules whose underlying $B$-dg module is perfect. Since $X$ is an affine scheme, every perfect $B$-dg module is equivalent to a strictly perfect one. Therefore, we can take $\cohs{B,s}$ as the dg category of $K(B,s)$-dg moodules whose underlying $B$-module is strictly perfect. More explicitly, such an object corresponds to a triplet $(E,d,h)$ where $(E,d)$ is a strictly perfect $B$-dg module and 
\begin{equation*}
    h:E\rightarrow E\otimes_B L_B[-1]
\end{equation*}
is a morphism such that 
\begin{equation*}
    (h\otimes id_{L_B}[-1])\circ h=0 \hspace{0.5cm} 
    [d,h]=h\circ d+(d\otimes id[-1])\circ h=id \otimes s.
\end{equation*}
Notice that $\cohs{B,s}$ is a locally flat $A$-linear dg category. This follows immediately from the fact that the underlying complexes of $B$-modules are strictly perfect and from the fact that $B$ is a flat $A$-algebra.

\end{construction}

This strict dg category is analogous to the one that has been introduced in \cite{brtv}. It gives us a strict model for the $\oo$-category $\cohbp{X_0}{X}$. Indeed, the following result holds:
\begin{lmm}\textup{\cite[Lemma 2.33]{brtv}}\\
Let $\cohs{B,s}^{\textup{acy}}$ be the full sub-category of $\cohs{B,s}$ spanned by acyclic dg modules. Then the cofibrant replacement induces an equivalence of dg categories
\begin{equation}
    \cohs{B,s}[\textup{q.iso}^{-1}]\simeq \cohs{B,s}/\cohs{B,s}^{\textup{acy}}\simeq \cohbp{X_0}{X}.
\end{equation}
As a consequence, if we label $\textup{Perf}^s(B,s)$ the full subcategory of $\cohs{B,s}$ spanned by perfect $K(B,s)$-dg modules, there are equivalances of dg categories
\begin{equation}
   \cohs{B,s}/\textup{Perf}^s(B,s)\simeq \cohbp{X_0}{X}/\perf{X_0}\simeq \sing{X,s}.
   \end{equation}
\end{lmm}

\begin{construction}\label{functorial properties cohs line bundle}
We can endow the assignment $(B,s)\rightarrow \cohs{B,s}$ with the structure of a pseudo-functor: let $f:(B,s)\rightarrow (C,t)$ be a morphism of affine twisted LG models, i.e. a morphism of $A$-algebras $f:B\rightarrow C$ such that the induced morphism $id\otimes f:L_B\rightarrow L_C$ sends $s$ to $t$. Then we can define the dg functor
\begin{equation}
    -\otimes_B C :\cohs{B,s}\rightarrow \cohs{C,t}
\end{equation}
\begin{equation*}
    (E,d,h)\mapsto (E\otimes_B C,d\otimes id, h\otimes id)
\end{equation*}
Indeed, it is clear that if $(E,d,h)$ is a $K(B,s)$-dg module whose underling $B$-module is perfect, then the $K(C,t)$-dg module $(E\otimes_BC,d\otimes id,h\otimes id)$ is levelwise $C$-projective and strictly bounded.

This yields a pseudo-functor
\begin{equation}\label{cohs}
    \cohs{\bullet,\bullet}:\textup{LG}^{\textup{aff,op}}_{(S,L)}\rightarrow \dgcatlf.
\end{equation}
\end{construction}
\begin{construction}
We can endow the pseudo-functor $(B,s)\mapsto \cohs{B,s}$ with a weakly associative and weakly unital lax monoidal structure. For any pair of twisted affine LG models $(B,s)$, $(C,t)$, we need to construct a morphism
\begin{equation}
    \cohs{B,s}\otimes \cohs{C,t}\rightarrow \cohs{B\otimes_AC,s\boxplus t}
\end{equation}
For a pair $(B,s)$, let $Z(s)$ denote $Spec(K(B,s))$. Moreover, let \begin{equation*}
    \V(L)=Spec(Sym(L^{\vee})).
\end{equation*} 
Then consider the following diagram
\begin{equation}
    \begindc{\commdiag}[16]
    
    \obj(0,30)[1]{$Z(s)\times^h_{S}Z(t)$}
    \obj(60,30)[2]{$Z(s\boxplus t)$}
    \obj(0,0)[3]{$S$}
    \obj(60,0)[4]{$\V(L)$}
    \obj(120,30)[5]{$Spec(B\otimes_A C)$}
    \obj(120,0)[6]{$\V(L)\times_{S}\V(L)$}
    \obj(60,-30)[7]{$S$}
    \obj(120,-30)[8]{$\V(L)$}
    \mor{1}{2}{$\phi$}
    \mor{1}{3}{$$}
    \mor{2}{4}{$$}
    \mor{3}{4}{$\textup{zero}$}
    \mor{2}{5}{$$}
    \mor{4}{6}{$(id,-id)$}
    \mor{5}{6}{$s\times t$}
    \mor{4}{7}{$$}
    \mor{6}{8}{$+$}
    \mor{7}{8}{$\textup{zero}$}
    \enddc
\end{equation}

where $\phi$ is the morphism corresponding to
\begin{equation}
    \begindc{\commdiag}[16]
    \obj(-100,20)[a]{$K(B,s)\otimes_A K(C,t)$}
    \obj(-50,20)[1]{$L_{B\otimes_AC}^{\vee}$}
    \obj(0,20)[2]{$L_{B\otimes_AC}^{\vee}\oplus L_{B\otimes_AC}^{\vee}$}
    \obj(50,20)[3]{$B\otimes_AC$}
    \obj(-100,-20)[b]{$K(B\otimes_AC,s\boxplus t)$}
    \obj(0,-20)[4]{$L_{B\otimes_AC}^{\vee}$}
    \obj(50,-20)[5]{$B\otimes_AC$}
    \mor{1}{2}{$$}
    \mor{2}{3}{$$}
    \mor{4}{5}{$$}
    \mor{4}{2}{$\begin{bmatrix}
    1\\ 1
    \end{bmatrix}$}
    \mor{5}{3}{$1$}
    \mor{b}{a}{$\phi$}
    \enddc
\end{equation}
\end{construction}

Consider the two projections
\begin{equation}
    \begindc{\commdiag}[20]
    \obj(0,0)[1]{$Z(s)\times^h_{S}Z(t)$}
    \obj(-30,0)[2]{$Z(s)$}
    \obj(30,0)[3]{$Z(t).$}
    \mor{1}{2}{$p_s$}[\atright,\solidarrow]
    \mor{1}{3}{$p_t$}
    \enddc
\end{equation}
Given $(E,d,h)\in \cohs{B,s}$ and $(E',d',h')\in \cohs{C,t}$, define 
\begin{equation}\label{multiplication lax monoidal structure cohs}
    (E,d,h)\boxtimes (E',d',h')=\phi_*\bigl ( p_s^*(E,d,h)\otimes p_t^*(E',d',h') \bigr ).
\end{equation}
This is a strictly perfect complex of $B\otimes_{A}C$-modules: its underlying complex of $A$-modules is $(E,d)\otimes_A(E',d')$, which immediately implies that it is strictly bounded. In order to see that each component $(E\otimes_A E')^n=\oplus_{i=0 }^nE^i\otimes_AE'^{n-i}$ is a projective $B\otimes_A C$-module, the same argument given in \cite[Construction 2.35]{brtv} applies. The morphism
\begin{equation*}
    E\otimes_AE'\rightarrow E\otimes_AE'\otimes_{B\otimes_AC}L_{B\otimes_AC}[-1] 
\end{equation*}
is $h\otimes 1+ 1\otimes h'$. In particular, $\phi_*\bigl ( p_s^*(E,d,h)\otimes p_t^*(E',d',h') \bigr )$ lives in $\cohs{B\otimes_AC,s\boxplus t}$.
The map 
\begin{equation}\label{unit lax monoidal structure cohs}
    \underline{A}\rightarrow \cohs{A,0},
\end{equation}
where $\underline{A}$ is the unit in $\textup{dgCat}^{\textup{lf}}_{A}$, i.e. the dg category with only one object whose endomorphism algebra is $A$, is defined by the assignment
\begin{equation*}
    \bullet \mapsto A.
\end{equation*}
It is clear that (\ref{multiplication lax monoidal structure cohs}) and (\ref{unit lax monoidal structure cohs}) satisfy the associativity and unity axioms, i.e. they enrich the pseudo-functor (\ref{cohs}) with a lax monoidal structure
\begin{equation}\label{cohs lax monoidal}
    \cohs{\bullet,\bullet}^{\otimes}:\textup{LG}_{(S,L)}^{\textup{aff,op},\boxplus}\rightarrow \textup{dgCat}_{A}^{\textup{lf},\otimes}.
\end{equation}

As in \cite[Construction 2.34, Construction 2.37]{brtv}, consider $\textup{Pairs-dgCat}^{\textup{lf}}_{A}$, the category of pairs $(T,F)$, where $T\in \textup{dgCat}^{\textup{lf}}_{A}$ and $F$ is a class of morphisms in $T$. A morphisms $(T,F)\rightarrow (T',F')$ is a dg functor $T\rightarrow T'$ sending $F$ in $F'$. There is an obvious functor $\textup{Pairs-dgCat}^{\textup{lf}}_{A}\rightarrow \textup{dgCat}^{\textup{lf}}_{A}$ defined as $(T,F)\mapsto T$. We say that a morphism $(T,F)\rightarrow (T',F')$ in $\textup{Pairs-dgCat}^{\textup{lf}}_{A}$ is a DK-equivalence if the dg functor $T\rightarrow T'$ is a DK-equivalence in $\textup{dgCat}^{\textup{lf}}_{A}$. Denote this class of morphisms in $\textup{Pairs-dgCat}^{\textup{lf}}_{A}$ by $W_{DK}$.

Also notice that the symmetric monoidal structure on $\textup{dgCat}^{\textup{lf}}_{A}$ induces a symmetric monoidal structure on $\textup{Pairs-dgCat}^{\textup{lf}}_{A}$
\begin{equation}
    (T,F)\otimes (T',F'):= (T\otimes T',F\otimes F').
\end{equation}
It is clearly associative and unital, where the unit is $(\underline{A},\{id\})$. We will denote this symmetric monoidal structure by $\textup{Pairs-dgCat}^{\textup{lf},\otimes}_{A}$. Note that this symmetric monoidal structure is compatible to the class of morphisms in $W_{DK}$, as we are working with locally-flat dg categories.

Then  as in \cite[Construction 2.34, Construction 2.37]{brtv}, we can construct a symmetric monoidal $\oo$-functor
\begin{equation}
    \textup{loc}_{dg}^{\otimes}: \textup{Pairs-dgCat}^{\textup{lf},\otimes}_{A}[W_{DK}^{-1}]\rightarrow \textup{dgCat}^{\textup{lf},\otimes}_{A}[W_{DK}^{-1}]\simeq \dgcatdko
\end{equation}
which, on objects, is defined by sending $(T,F)$ to $T[F^{-1}]_{dg}$, the dg localization of $T$ with respect to $F$. Let $\{q.iso\}$ be the class of quasi-isomorphisms in $\textup{Coh}^s(R_i,s_{R_i})$. Consider the functor
\begin{equation}
    \textup{LG}^{\textup{aff,op},\boxplus}_{A}\rightarrow \textup{Pairs-dgCat}^{\textup{lf},\otimes}_{A}
\end{equation}
\begin{equation*}
    (B,s)\mapsto \bigl (\textup{Coh}^{s}(B,s), \{q.iso\} \bigr )
\end{equation*}

and compose it with the functor
\begin{equation}
\begindc{\commdiag}[18]
\obj(-60,10)[1]{$\textup{Pairs-dgCat}^{\textup{lf},\otimes}_{A}$}
\obj(0,10)[2]{$\textup{Pairs-dgCat}^{\textup{lf},\otimes}_{A}[W_{DK}^{-1}]$}
\obj(60,10)[3]{$\textup{dgCat}^{\textup{lf},\otimes}_{A}[W_{DK}^{-1}]$}
\obj(-40,-10)[4]{$\textup{\textbf{dgCat}}^{\otimes}_{A}$}
\obj(40,-10)[5]{$\textup{\textbf{dgCat}}^{\textup{idm},\otimes}_{A}$}
\mor{1}{2}{$\textup{loc}$}
\mor{2}{3}{$\textup{loc}_{dg}$}
\mor{4}{5}{$\textup{loc}$}
\cmor((58,6)(58,2)(56,0)(10,0)(-36,0)(-38,-2)(-38,-6)) \pdown(0,3){$\simeq$}
\enddc
\end{equation}
We get, after a suitable monoidal left Kan extension, the desired lax monoidal $\oo$-functor
\begin{equation}\label{cohbp monoidal}
    \cohbp{\bullet}{\bullet}^{\otimes}:\textup{LG}^{\textup{op},\boxplus}_{A}\rightarrow \textup{\textbf{dgCat}}^{\textup{idm},\otimes}_{A}.
\end{equation}

\begin{lmm}\label{RHom(A,A)=Sym_A(L[-2])}
Let $S=Spec(A)$ be a regular noetherian ring. Let $L$ be a line bundle over $S$. The following equivalence holds in $\textup{CAlg}(\dgcatmo)$
\begin{equation}\label{equivalence cohb(S_0) perf(R[u])}
    \cohbp{S_0}{S}^{\otimes}\simeq \cohb{S_0}^{\otimes}\xrightarrow{\sim} \perf{Sym_{A}(L[-2])}^{\otimes}.
\end{equation}
\end{lmm}
\begin{proof}
The first equivalence is an immediate consequence of the regularity hypothesis on $S$.
As we have remarked in the Example (\ref{derived zero locus zero section}), the cdga assocaiated via the Dold-Kan correspondence to $S_0$ is $L^{\vee}\xrightarrow{0} A$. By definition, $A$ be a compact generator of $\perf{S}$. Also by definition, $L^{\vee}\xrightarrow{0}A$ is a compact generator of $\perf{S_0}$. It is easy to see that $A$, with the trivial $(L^{\vee}\xrightarrow{0}A)$-dg module structure, is a generator of $\cohb{S_0}$: if $M$ is a nonzero object in $\cohb{S_0}$, then there exists $0\neq m \in \textup{H}^i(M)$ for some $i$. We can moreover assume that $i=0$. Then $m$ induces a non-zero morphism $A\rightarrow M$.

In particular, we get the equivalence 
\begin{equation*}
\perf{\rhom_{(L^{\vee}\xrightarrow{0}A)}(A,A)}\simeq \cohb{S_0}.
\end{equation*}
The endomorphism dg algebra $\rhom_{(L^{\vee}\xrightarrow{0}A)}(A,A)$ can be computed explicitly by means of the following $L^{\vee}\xrightarrow{0}A$-resolution of $A$:
\begin{equation}
    \dots \underbracket{L^{\vee \otimes 2}}_{-3}\xrightarrow{0}\underbracket{L^{\vee}}_{-2}\xrightarrow{1}\underbracket{L^{\vee}}_{-1}\xrightarrow{0}\underbracket{A}_{0}.
\end{equation}
Applying $\hm_{A}(-,A)$ we obtain
\begin{equation}
\underbracket{A}_{0}\xrightarrow{0}\underbracket{\hm_{A}(L^{\vee},A)}_{1}\xrightarrow{1}\underbracket{\hm_{A}(L^{\vee},A)}_{2}\xrightarrow{0}\underbracket{\hm_{A}(L^{\vee \otimes 2},A)}_{3}\dots
\end{equation}
\begin{equation*}
\simeq
    \underbracket{A}_{0}\xrightarrow{0}\underbracket{L}_{1}\xrightarrow{1}\underbracket{L}_{2}\xrightarrow{0}\underbracket{L^{\otimes2}}_{3}\dots
\end{equation*}
which is quasi isomorphic to $A$. However, when we ask for $(L^{\vee}\xrightarrow{0}A)$-linearity, the copies of $L^{\otimes n}$ in odd degree disappear, as the local generators $\varepsilon$ in degree $-1$ of $(L^{\vee}\xrightarrow{0}A)$ act via the identity on $(L^{\vee}\xrightarrow{0}A)$. Therefore we find that $\rhom_{(L^{\vee}\xrightarrow{0} A)}(A,A)$ is quasi-isomorphic to
\begin{equation}
    \underbracket{A}_{0}\rightarrow \underbracket{0}_{1}\rightarrow \underbracket{L}_{2}\rightarrow \dots \simeq Sym_{A}(L[-2]).
\end{equation}
This shows that there exists an equivalence
\begin{equation}\label{equivalence modules}
    \rhom_{(L^{\vee}\xrightarrow{0}A)}(A,A)\simeq Sym_{A}(L[-2]) 
\end{equation}
as $(L^{\vee}\xrightarrow{0}A)$-dg modules. Notice that both these objects carry a canonical algebra structure. 
In order to conclude that the two commutative algebra structures coincide, we consider the dg functor
\begin{equation}
    \rhom_{(L^{\vee}\xrightarrow{0}A)}(A,-): \cohs{S,0}\rightarrow Sym_{A}(L[-2])-dgmod.
\end{equation}

Notice that $Sym_{A}(L[-2])$ is a strict cdga, seen as a commutative algebra object in $\qcoh{A}$.
Similarly to \cite[Lemma 2.39]{brtv}, for $(E,d,h)$ an object in $\cohs{A,0}$, the $Sym_A(L[-2])$-dg module $\rhom_{(L^{\vee}\xrightarrow {0}A)}(A,(E,d,h))$ can be computed (in degrees $i$ and $i+1$) as
\begin{equation}
    \bigoplus_{n\geq 0}E_{i-2n}\otimes_A L^{\otimes n} \xrightarrow{d+h} \bigoplus_{n\geq 0}E_{i+1-2n}\otimes_A L^{\otimes n}.
\end{equation}
The same arguments given in \textit{loc.cit.} hold mutatis mutandis in our situation and therefore we obtain a symmetric monoidal functor
\begin{equation}
      \cohs{S,0}^{\otimes}\rightarrow Sym_{A}(L[-2])-dgmod^{\otimes}
\end{equation}
which preserves quasi-isomorphisms. If we localize both the l.h.s. and the r.h.s. we thus obtain a symmetric monoidal $\oo$-functor
\begin{equation}
    \cohb{S_0}^{\otimes}\rightarrow \qcoh{Sym_A(L[-2])}^{\otimes},
\end{equation}
from which one recovers the equivalence of symmetric monoidal dg categories
\begin{equation}
    \cohb{S_0}^{\otimes}\simeq \perf{Sym_A(L[-2])}^{\otimes}.
\end{equation}
\end{proof}
\begin{cor}
Let $S=Spec(A)$ be a regular noetherian affine scheme and let $L$ be a line bundle over $S$. The lax monoidal $\oo$-functors $(\ref{cohbp monoidal})$ factors as
\begin{equation}\label{cohbp monoidal modules 2}
    \cohbp{\bullet}{\bullet}^{\otimes}:\textup{LG}_{(S,L)}^{\textup{op},\boxplus}\rightarrow \md_{\perf{Sym_{A}(L[-2])}}(\dgcatm)^{\otimes}.
\end{equation}
\end{cor}
Then, for a noetherian regular scheme $S$ with a line bundle $\Lc_S$, we obtain a lax monoidal $\oo$-functor
\begin{equation}\label{cohbp non affine}
    \cohbp{\bullet}{\bullet}^{\otimes}:\tlgm{S}^{\textup{op},\otimes}\rightarrow \dgcatmo
\end{equation}
as the limit
\begin{equation}
    \varprojlim_{(Spec(A),\Lc_{S|Spec(A)})}\bigl ( \cohbp{\bullet}{\bullet}^{\otimes}:\textup{LG}_{Spec(A),\Lc_{S|Spec(A)}}^{\textup{op},\boxplus}\rightarrow \textup{\textbf{dgCat}}_A^{\textup{idm},\otimes}\bigr )
\end{equation}
where $Spec(A)\rightarrow S$ is a Zariski open subscheme. We have used Lemma \ref{zariski descent tlgmm} and the definition
\begin{equation}
    \varprojlim_{Spec(A)\rightarrow S}\textup{\textbf{dgCat}}_A^{\textup{idm},\otimes}= \dgcatmo.
\end{equation}
\begin{rmk}\label{actions}
The monoidal structure on $\cohbp{\bullet}{\bullet}^{\otimes}$ implies that each dg category $\cohbp{X_0}{X}$ is endowed with an action of $\cohb{S_0}$ (recall that we are assuming that $S$ is regular). Similarly to \cite[Remark 2.38]{brtv} we can describe this action. Consider the diagram
\begin{equation}
    \begindc{\commdiag}[20]
    \obj(20,30)[1]{$X_0\times^h_S S_0$}
    \obj(0,15)[2]{$X_0$}
    \obj(70,30)[3]{$X_0$}
    \obj(50,15)[4]{$X$}
    \obj(20,0)[5]{$S_0$}
    \obj(0,-15)[6]{$S$}
    \obj(70,0)[7]{$S$}
    \obj(50,-15)[8]{$\V(\Lc_S)$}
    \mor{1}{2}{$p_{X_0}$}[\atright,\solidarrow]
    \mor{1}{3}{$a$}
    \mor{2}{4}{$\mathfrak{i}$}
    \mor{3}{4}{$\mathfrak{i}$}
    \mor{5}{6}{$$}
    \mor{5}{7}{$$}
    \mor{6}{8}{$\textup{zero}$}
    \mor{7}{8}{$\textup{zero}$}
    \mor{1}{5}{$$}
    \obj(14,9)[a]{$p_{S_0}$}
    \mor{2}{6}{$$}
    \mor{3}{7}{$$}
    \mor{4}{8}{$$}
    \enddc
\end{equation}
Notice that $X_0\times^{h}_S S_0\simeq X_0\times^h_X X_0$. Given $\Ec \in \cohbp{X_0}{X}$ and $\Fc \in \cohb{S_0}$, then $\Ec \boxtimes \Fc= a_*(p_{X_0}^* \Ec \otimes p_{S_0}^*\Fc)$. In particular, when $\Fc=\Lc_S^{\vee}\xrightarrow{0}\Oc_S=\Oc_{S_0}$, then
\begin{equation}
    \Ec \boxtimes (\Lc_S^{\vee}\xrightarrow{0}\Oc_S)=a_*(p_{X_0}^* \Ec\otimes p_{S_0}^*\Oc_{S_0})\simeq a_*(p_{X_0}^* \Ec\otimes a^*\Oc_{X_0})
\end{equation}
\begin{equation*}
    \underbracket{\simeq}_{\textup{proj. form.}}a_*p_{X_0}^*\Ec \underbracket{\simeq}_{\textup{der. prop. base change}}\mathfrak{i}^*\mathfrak{i}_*\Ec.
\end{equation*}
When $\Fc=\Oc_S=t_*\Oc_S$, where $t:S=\pi_0(S_0)\rightarrow S_0$ is the truncation morphism, the (homotopy) cartesian square
\begin{equation}
    \begindc{\commdiag}[20]
    \obj(0,15)[1]{$X_0\times_{S}S$}
    \obj(50,15)[2]{$X_0\times_S^hS_0$}
    \obj(0,-15)[3]{$S$}
    \obj(50,-15)[4]{$S_0$}
    \mor{1}{2}{$id\times t $}
    \mor{1}{3}{$q$}
    \mor{2}{4}{$p_{S_0}$}
    \mor{3}{4}{$t$}
    \enddc
\end{equation}
implies that
\begin{equation}
    \Ec\boxtimes \Oc_S=a_*(p_{X_0}^*\Ec\otimes p_{S_0}^*t_*\Oc_S)\underbracket{\simeq}_{\textup{der. prop. base change}} a_*(p_{X_0}^*\Ec\otimes (id\times t)_*\Oc_{X_0})
\end{equation}
\begin{equation*}
    \underbracket{\simeq}_{\textup{proj. form.}}a_*(id\times t)_*((id\times t)^*p_{X_0}^*\Ec)\simeq \Ec.
\end{equation*}
Finally, we can consider $\Lc^{\vee}_S$ endowed with the trivial $\Lc^{\vee}_S\xrightarrow{0}\Oc_S$ dg module structure, namely $t_*\Lc^{\vee}_S$. Then, for any $\Ec \in \cohbp{X_0}{X}$, we have
\begin{equation}
    \Ec \boxtimes t_*\Lc^{\vee}_S=a_*(p_{X_0}^*\Ec\otimes p_{S_0}^*t_*\Lc_S^{\vee})\underbracket{\simeq}_{\text{der. prop. base change}} a_*(p_{X_0}^*\Ec\otimes (id\times t)_*q^*\Lc_S^{\vee}).
\end{equation}
Notice that $q^*\Lc^{\vee}_S\simeq \Lc^{\vee}_X\otimes_{\Oc_X}\Oc_{X_0}$, thus we can continue
\begin{equation}
    \simeq  a_*(p_{X_0}^*\Ec\otimes (id\times t)_*(\Lc^{\vee}_X\otimes_{\Oc_X}\Oc_{X_0}))\simeq \Ec\otimes_{\Oc_{X_0}}\Oc_{X_0}\otimes_{\Oc_X}\Lc^{\vee}_X\simeq \Ec \otimes_{\Oc_X}\Lc_X^{\vee}.
\end{equation}
\end{rmk}
We will construct, using $\cohbp{\bullet}{\bullet}^{\otimes}$, a lax monoidal $\oo$-functor \begin{equation}
    \sing{\bullet,\bullet}^{\otimes}:\textup{LG}_{(S,\Lc_S)}^{\textup{op},\boxplus}\rightarrow \dgcatmo,
\end{equation} 
following \cite{pr11} and \cite{brtv}. 
\begin{construction}
As in Lemma \ref{RHom(A,A)=Sym_A(L[-2])}, consider the strict commutative dg algebra $Sym_{\Oc_S}(\Lc_S[-2])$ as a commutative algebra object in $\qcoh{S}$. Consider the $Sym_{\Oc_S}(\Lc_S[-2])$-algebra
\begin{equation}
    \Rc:=Sym_{\Oc_S}(\Lc_S[-2])[\nu^{-1}]
\end{equation}
where $\nu$ is the generator in degree $2$. More explicitly, 
\begin{equation}
    \Rc= \dots \rightarrow \underbracket{\Lc_S^{\vee,\otimes 2}}_{-4} \rightarrow \underbracket{\Lc_S^{\vee}}_{-2}\rightarrow 0\rightarrow \underbracket{\Oc_S}_{0}\rightarrow 0 \rightarrow \underbracket{\Lc_S}_{2}\rightarrow 0 \rightarrow \underbracket{\Lc_S^{\otimes 2}}_{4} \rightarrow \dots
\end{equation}
Then we have a symmetric monoidal $\oo$-functor $-\otimes_{\perf{Sym_{\Oc_S}(\Lc_S[-2])}}\perf{\Rc}$ 
\begin{equation}\label{invert u}
\md_{\perf{Sym_{\Oc_S}(\Lc_S[-2])}}(\dgcatm)^{\otimes}
\rightarrow \md_{\perf{\Rc}}(\dgcatm)^{\otimes}.
\end{equation}
Composing it with $(\ref{cohbp monoidal modules 2})$ we obtain a lax monoidal $\oo$-functor
\begin{equation}\label{cohbp otimes perf R[u,u^{-1}]}
    \tlgm{S}^{\textup{op},\boxplus}\rightarrow \md_{\perf{\Rc}}(\dgcatm)^{\otimes}
\end{equation}
which, at the level of objects, is defined by
\begin{equation*}
    (X,s_X)\mapsto \cohbp{X_0}{X}\otimes_{\perf{Sym_{\Oc_S}(\Lc[-2])}}\perf{\Rc}.
\end{equation*}
\end{construction}
\begin{rmk}
Let $\mathcal{U}=\{ U_i= Spec(A_i) \}$ be a Zariski affine covering of $S$, such that $\Lc_{S|U_i}$ is equivalent to $A_i$. Then the restriction of $\Rc$ to $U_i$ is equivalent to the dg algebra $A_i[u,u^{-1}]$, where $u$ sits in cohomological degree $2$.
\end{rmk}
\begin{lmm} Let $\mathcal{U}=\{U_i=Spec(B_i)\}$ be an affine open covering which trivializes $\Lc_S$ and $\Lc_S^{\vee}$.

Under the equivalence $(\ref{equivalence cohb(S_0) perf(R[u])})$, the full sub-category $\perf{S_0}$ of $\cohb{S_0}$ corresponds to the full sub-category of $\perf{Sym_{\Oc_S}(\Lc_S[-2])}$ spanned by perfect $Sym_{\Oc_S}(\Lc_S[-2])$-modules whose restriction to each $U_i$ is a $u$-torsion perfect $B_i[u]$-dg module, that we will denote $\perf{Sym_{\Oc_S}(\Lc_S[-2])}^{\textup{tors}}$.
\end{lmm}
\begin{proof}
Let $\Mcor\in \perf{S_0}$. By definitions this means that, for every Zariski open $f:Spec(A)\rightarrow S_0$, the pullback $f^*\Mcor\in \perf{A}$. Consider our covering $\mathcal{U}$. Then, for every $i\in I$, $\Mcor_{|U_i}$ is a perfect $B_i$-module. By the same argument given above, we get equivalences $\cohb{U_{i,0}}^{\otimes}\simeq \perf{B_i[u]}^{\otimes}$ and the argument provided in \cite[Proposition 2.43]{brtv} guarantees that $\perf{U_{i,0}}$ corresponds to $\perf{B_i[u]}^{u-\textup{tors}}$ under this equivalence. This shows that every perfect complex is locally of $u$-torsion. The converse follows by the local characterization of perfect complexes and by the characterization given in \cite[Proposition 2.43]{brtv}.
\end{proof}
\begin{lmm}
The dg quotient $\cohb{S_0}\rightarrow \sing{S_0}$ corresponds to the base-change 
\begin{equation}
  -\otimes_{Sym_{\Oc_S}(\Lc_S[-2])}\Rc : \perf{Sym_{\Oc_S}(\Lc_S[-2])}\rightarrow \perf{\Rc}
\end{equation}
under the equivalence $(\ref{equivalence cohb(S_0) perf(R[u])})$
\end{lmm}
\begin{proof}
This follows by the fact that the categories involved satisfy Zariski descent, by the previous lemmas and by \cite[Proposition 2.43]{brtv}.
\end{proof}
\begin{lmm}
There is an equivalence of dg categories
\begin{equation}
    \cohbp{X_0}{X}\otimes_{\perf{Sym_{\Oc_S}(\Lc_S[-2])}}\perf{Sym_{\Oc_S}(\Lc_S[-2])}^{\textup{tors}}\simeq \perf{X_0}.
\end{equation}
\end{lmm}
\begin{proof}
By definition, 
\begin{equation}
    \cohbp{X_0}{X}\otimes_{\perf{Sym_{\Oc_S}(\Lc_S[-2])}}\perf{Sym_{\Oc_S}(\Lc_S[-2])}^{\textup{tors}}
\end{equation} 
is the subcategory spanned by locally $u$-torsion modules in $\cohbp{X_0}{X}$. We will show that they coincide with the subcategory of perfect complexes, using \cite[Proposition 2.45]{brtv}. Suppose that $\Mcor$ is in $\perf{X_0}$. Then, for every affine open covering $j:Spec(A)\rightarrow X_0$ which trivializes $\Lc_{X_0}$, $j^*\Mcor$ is perfect, and thus it is a $u$-torsion perfect $A[u]$ module by \textit{loc. cit}. Conversely, if $\Mcor$ is locally $u$-torsion, each restriction $j^*\Mcor$ is $u$-torsion, i.e. perfect. This proves the lemma.
\end{proof}
\begin{cor}
Let $(X,s_X)$ be a twisted LG model over $(S,\Lc_S)$. Then the exact sequence in $\dgcatm$
\begin{equation}
    \perf{X_0}\rightarrow \cohbp{X_0}{X}\rightarrow \sing{X,s_X}
\end{equation}
is equivalent to
\begin{equation}
\begindc{\commdiag}[18]
    \obj(0,20)[1]{$\cohbp{X_0}{X}\otimes_{\perf{Sym_{\Oc_S}(\Lc_S[-2])}}\perf{Sym_{\Oc_S}(\Lc_S[-2])}^{\textup{tors}}$} 
    \obj(0,0)[2]{$\cohbp{X_0}{X}$}
    \obj(0,-20)[3]{$\cohbp{X_0}{X}\otimes_{\perf{Sym_{\Oc_S}(\Lc_S[-2])}}\perf{\Rc}.$}
    \mor{1}{2}{$$}
    \mor{2}{3}{$$}
\enddc
\end{equation}
\end{cor}
\begin{proof}
This is an obvious consequence of the previous lemmas and of the fact that the $\oo$-functor
\[
T\mapsto \cohbp{X_0}{X}\otimes_{\perf{Sym_{\Oc_S}(\Lc_S[-2])}}T
\]preserves exact sequences in $\md_{\perf{Sym_{\Oc_S}(\Lc_S[-2])}}(\dgcatm)$.
\end{proof}
\begin{construction}
Consider the following lax monoidal $\oo$-functor
\begin{equation}
\begindc{\commdiag}[18]
\obj(-60,10)[1]{$\tlgm{S}^{\textup{op},\boxplus}$}
\obj(60,10)[2]{$\md_{\perf{Sym_{\Oc_S}(\Lc_S[-2])}}(\dgcatm)^{\otimes}$}
\obj(60,-10)[3]{$\md_{\perf{\Rc}}(\dgcatm)^{\otimes}$.}

\mor{1}{2}{$\cohbp{\bullet}{\bullet}^{\otimes}$}
\mor{2}{3}{$-\otimes_{Sym_{\Oc_S}(\Lc_S[-2])}\Rc$}
\enddc    
\end{equation}
The previous corollary means that its underlying $\oo$-functor, composed with the forgetful functor
\begin{equation}
    \md_{\perf{\Rc}}(\dgcatm)\rightarrow \dgcatm,
\end{equation}
is defined on objects by the assignment
\begin{equation*}
    (X,s)\mapsto \sing{X,s}.
\end{equation*}
\end{construction}
\subsection{The motivic realization of \texorpdfstring{$\sing{X,s_X}$}{Sing(X,s)}}
Recall from Section \ref{the bridge between motives and non commutative motives} that there is a motivic realization lax monoidal $\oo$-functor 
\begin{equation}
    \Mv_S:\dgcatmo\rightarrow \md_{\bu_S}(\sh)^{\otimes}
\end{equation}
with the following properties
\begin{itemize}
    \item $\Mv_S$ commutes with filtered colimits.
    \item For every dg category $T$, $\Mv_S(T)$ is the spectrum
    \begin{equation*}
        Y\in \sm \mapsto \textup{HK}(\perf{Y}\otimes_S T).
    \end{equation*}
    \item $\Mv_S$ sends exact sequences of dg categories to fiber-cofiber sequences in $\md_{\bu_S}(\shm)$. 
\end{itemize}
Our main scope in this section is to study the motivic realization of the dg category $\sing{X,s_X}$ associated to $(X,s_X)\in \tlgm{S}$, under the assumption that $X$ is regular.
The first important fact is the following one:
\begin{prop}{\textup{\cite[Proposition 3.13]{brtv}}}
Let $p:X\rightarrow S$ be in $\sch$. Then $\Mcor^{\vee}_X(\perf{X})\simeq \bu_X$ (see section \ref{stable homotopy category of schemes} for notation).
\end{prop}
\begin{proof}
Consider the construction given in Section \ref{the bridge between motives and non commutative motives} with $S$ replaced by $X$:
\begin{equation}
    \textup{\textbf{Sm}}^{\times}_X \xrightarrow{\perf{\bullet}} \textup{\textbf{dgCat}}^{\textup{idm,op},\otimes}\xrightarrow{\iota} \textup{\textbf{SH}}^{\textup{nc,op},\otimes}_X \xrightarrow{\rhom(-,1_X^{nc})} \textup{\textbf{SH}}^{\textup{nc},\otimes}_X \xrightarrow{\Mcor_X} \textup{\textbf{SH}}_X^{\otimes}.
\end{equation}
We wish to show that $\Mcor_X (\rhom(\iota \circ \perf{X},1^{nc}_X))\simeq \bu_X$. Notice that since the $\oo$-functor above is right lax monoidal and $\perf{X}$ is a commutative algebra in $\textup{\textbf{dgCat}}^{\textup{idm,op},\otimes}$, there is a canonical morphism of commutative algebras 
\begin{equation}
    \bu_X\rightarrow \Mcor_X (\rhom(\iota \circ \perf{X},1^{nc}_X)).
\end{equation} 
It suffices to show that they represent the same $\oo$-functor. Let $\mathcal{Y}$ be an object in $\textup{\textbf{SH}}_X^{\otimes}$. Then
\begin{equation}
    \Map_{\textup{\textbf{SH}}_X}(\mathcal{Y},\Mcor_X (\rhom(\iota \circ \perf{X},1^{nc}_X)))
    \end{equation}
\begin{equation*}
\simeq \Map_{\textup{\textbf{SH}}^{\textup{nc}}_X}(\textup{R}_{\textup{Perf}}(\mathcal{Y}),\rhom(\iota \circ \perf{X},1^{nc}_X))
\end{equation*}
\begin{equation*}
    \simeq \Map_{\textup{\textbf{SH}}^{\textup{nc}}_X}(\textup{R}_{\textup{Perf}}(\mathcal{Y})\otimes \textup{R}_{\textup{Perf}}(\Sigma^{\oo}_+X),1^{nc}_X)\simeq \Map_{\textup{\textbf{SH}}_X^{nc}}(\textup{R}_{\textup{Perf}}(\mathcal{Y}\otimes \Sigma^{\oo}_+X),1^{nc}_X)
\end{equation*}
\begin{equation*}
    \simeq \Map_{\textup{\textbf{SH}}_X}(\textup{R}_{\textup{Perf}}(\mathcal{Y}),1_X^{nc})\simeq \Map_{\textup{\textbf{SH}}_X}(\mathcal{Y},\Mcor_X(1_X^{nc}))\simeq \Map_{\textup{\textbf{SH}}_X}(\mathcal{Y},\bu_X)
\end{equation*}
where we have used that $\iota \circ \perf{\bullet}\simeq \textup{R}_{\textup{Perf}}\circ \Sigma^{\oo}_+$, which are all symmetric monoidal functors, $\Sigma^{\oo}_+(X)$ is the unit in $\textup{\textbf{SH}}_X^{\otimes}$, that $\textup{R}_{\textup{Perf}}$ is left adjoint to $\Mcor_X$ and Robalo's result $\Mcor_X(1^{nc}_X)\simeq \bu_X$ (\cite[Theorem 1.8]{ro15}). The fact that the equivalence holds in $\textup{CAlg}(\textup{\textbf{SH}}_X^{\otimes})$ follows immediately from the conservativity of the forgetful functor
\begin{equation}
   \textup{CAlg}(\textup{\textbf{SH}}_X^{\otimes})\rightarrow \textup{\textbf{SH}}_X^{\otimes}.
\end{equation}
\end{proof}
\begin{rmk}\label{multiplication by perfect complex in HK}
Notice that since the equivalence $\bu_X\rightarrow \Mcor^{\vee}_X(\perf{X})$ holds in $\textup{CAlg}(\textup{\textbf{SH}}_X)$, if we consider a perfect complex $\Ec$ scheme $X$, then the image of
\begin{equation}
    \perf{X}\xrightarrow{-\otimes \Ec}\perf{X}
\end{equation}
along $\Mcor^{\vee}_X$ coincides with multiplication by the class $[\Ec]\in \textup{HK}_0(X)$ in the commutative algebra $\bu_X$, which we denote $m_{\Ec}$. In a formula : $\Mcor^{\vee}_X(-\otimes \Ec)\simeq m_{\Ec}\in \Map_{\bu_X}(\bu_X,\bu_X)$.
Moreover, for any exact triangle $\Ec'\rightarrow \Ec \rightarrow \Ec''$ we get $[\Ec]=[\Ec']+[\Ec'']\in \textup{HK}_0(X)$. Then $m_{\Ec}=m_{\Ec'}+m_{\Ec''}$. In particular, $m_{\Ec[1]}=-m_{\Ec}$.
\end{rmk}
The next step will be understanding the motivic realization of the category of coherent bounded complexes $\cohb{Z}$ of an $S$-scheme $Z$, at least when it is possible to regard it as a closed subscheme of a regular $S$-scheme $X$.
\begin{prop}{\textup{\cite[Proposition 3.17]{brtv}}}\label{motivic realization cohb}
Let $p:X\rightarrow S$ be a regular $S$-scheme of finite type. Consider a closed subscheme $i: Z\rightarrow X$ and label $j:U\rightarrow X$ the embedding of the complementary open subscheme. Then there is an equivalence 
\begin{equation}
    \Mcor_X(\cohb{Z})\simeq i_*i^!\bu_X.
\end{equation}
\end{prop}
\begin{proof}
Consider the exact sequence of dg categories 
\begin{equation}
    \cohb{X}_Z \rightarrow \cohb{X}\xrightarrow{j^*} \cohb{U}
\end{equation}
where $\cohb{X}_Z$ denotes the subcategory of objects in $\cohb{X}$ whose support is in $Z$.

Recall that, if $X$ is a scheme, an $\Oc_X$-linear dg category is a sheaf of dg categories on the Zariski site of $X$. Then the above sequence of dg categories is exact by definition and if we apply $\Mcor^{\vee}_X$ to it we obtain a fiber cofiber sequence in $\textup{\textbf{SH}}_X$ (see \cite[Remark 2.2.2]{tv19b}).

The regularity hypothesis imposed on $X$ implies that $\cohb{X}\simeq \perf{X}$ and $\cohb{U}\simeq \perf{U}$. If we apply $\Mcor_X^{\vee}$ and use the previous proposition, we obtain
\begin{equation}
    \Mcor^{\vee}_X(\cohb{X}_Z)\rightarrow \bu_X \xrightarrow{\Mcor^{\vee}_X(j^*)} \Mcor^{\vee}_X(\perf{U}),
\end{equation}
which is a fiber-cofiber sequence in $\textup{\textbf{SH}}_X$ (since $\Mv_X$ sends exact triangles of dg categories to fiber-cofiber sequences). Moreover, $\Mcor_X^{\vee}$ is compatible with pushforwards and $\Mcor^{\vee}_X(j^*)\sim j^*$. As the spectrum $\bu$ of non-connective homotopy-invariant K-theory is compatible with pullbacks (see \cite{cd19}), the previous fiber-cofiber sequence is nothing but
\begin{equation}
    \Mcor^{\vee}_X(\cohb{X}_Z)\rightarrow \bu_X \rightarrow j_*j^*\bu_X,
\end{equation}
where the map on the right is induced by the unit of the adjunction $(j^*,j_*)$. In particular, we get a canonical equivalence
\begin{equation}
    \Mcor^{\vee}_X(\cohb{X}_Z)\simeq i_*i^!\bu_X
\end{equation}
and therefore we are left to show that  we have an equivalence \begin{equation}
    \Mcor_X^{\vee}(\cohb{Z})\simeq \Mcor_X^{\vee}(\cohb{X}_Z).
\end{equation} 

Here $\cohb{Z}$ is the sheaf of dg categories
\begin{equation}
    V=Spec(R)\mapsto \cohb{Z\times_XV},
\end{equation}
where $V$ is an affine open subscheme of $X$.

Notice that there is a canonical morphism \begin{equation*}
    \Mcor^{\vee}_X(i_*):\Mcor^{\vee}_X(\cohb{Z}\rightarrow \Mcor^{\vee}_X(\cohb{X}_Z)).
\end{equation*}
The collection of objects $\Sigma^{\oo}_+Y\otimes \bu_X$, where $Y\in \textup{\textbf{Sm}}_X$, forms a family of compact generators of $\md_{\bu_X}(\textup{\textbf{SH}}_X)$. As $\Mcor^{\vee}_X$ commutes with filtered colimits, it suffices to show that the morphism
\begin{equation}
    \begindc{\commdiag}[18]
    \obj(0,10)[1]{$\Map_{\bu_X}(\Sigma^{\oo}_+Y\otimes \bu_X,\Mcor^{\vee}_X(\cohb{Z}))$} 
    \obj(0,-10)[2]{$\Map_{\bu_X}(\Sigma^{\oo}_+Y\otimes \bu_X,\Mcor^{\vee}_X(\cohb{X}_Z))$}
    \mor{1}{2}{$$}
    \enddc
\end{equation}

is an equivalence of spectra. Notice that
\begin{equation*}
     \Map_{\bu_X}(\Sigma^{\oo}_+Y\otimes \bu_X,\Mcor^{\vee}_X(\cohb{Z}))\simeq  \Map_{\textup{\textbf{SH}}_X}(\Sigma^{\oo}_+Y,\Mcor^{\vee}_X(\cohb{Z}))
\end{equation*}
\begin{equation}
    \simeq \Map_{\textup{\textbf{SH}}^{nc}_X}(\textup{R}_{\textup{Perf}}\circ \Sigma^{\oo}_+Y,\rhom(\cohb{Z},1^{nc}_X))
\end{equation}
\begin{equation*}
\simeq \Map_{\textup{\textbf{SH}}^{nc}_X}(\perf{Y}\otimes \cohb{Z},1^{nc}_X).
\end{equation*}
By \cite{pr11}, $\perf{Y}\otimes \cohb{Z}\simeq \cohb{Y}\otimes \cohb{Z}\simeq \cohb{Y\times_X Z}$, and therefore the spectrum above coincides with $\textup{HK}(\cohb{Y\times_X Z})$, which by the homotopy-invariance of G-theory and by the theorem of the Heart (\cite[Corollary 6.4.1]{ba15}) coincides with the $G$-theory of $Y\times_X Z$. In the same way, we obtain that $\Map_{\textup{\textbf{SH}}_X}(\Sigma^{\oo}_+Y,\Mcor^{\vee}_X(\cohb{X}_Z))$ coincides with the $G$-theory spectrum of the abelian category $\textup{\textbf{Coh}}(Y)_{Y\times_XZ}$ (the heart of the dg category $\cohb{Y}_{Y\times_XZ}$). The claim now follows from Quillen's d\'evissage.
\end{proof}

\begin{rmk}
An anonymous referee has pointed out that in order to deal with $\Oc_X$-linear dg categories one can use the theory of $1$-affiness of D.~Gaitsgory (\cite{g15}). One can use this to give an alternative proof of the proposition above.
\end{rmk}

As a final observation, we remark that the assignemnt $Z\mapsto \Mcor^{\vee}_S(\cohb{Z})$ is insensible to (derived) thickenings.
\begin{lmm}{\textup{\cite[Proposition 3.24]{brtv}}}
Let $Z$ be a derived scheme of finite type over $S$ and let $t:\pi_0(Z)\rightarrow Z$ be the canonical closed embedding of the underlying scheme of $Z$. Then
\begin{equation}
    \Mcor^{\vee}_S(\cohb{Z})\simeq \Mcor^{\vee}_S(\cohb{\pi_0(Z)}).
\end{equation}
\end{lmm}
\begin{proof}
By the proof of the previous proposition, $\Mcor^{\vee}_S(\cohb{Z})$ represents the spectra valued sheaf $Y\mapsto \textup{HK}(\cohb{Y\times_S Z})$. Similarly, $\Mcor^{\vee}_S(\cohb{\pi_0(Z)})$ represents the spectra-valued sheaf $Y\mapsto \textup{HK}(\cohb{Y\times_S \pi_0(Z)})$. Notice that $\pi_0(Y\times_S Z)=\pi_0(Y)\times_{\pi_0(S)}\pi_0(Z)=Y\times_S \pi_0(Z)$ and that the heart of $\cohb{Y\times_S Z}$ is equivalent to the heart of $\cohb{Y\times_S \pi_0(Z)}$. Then the theorem of the Heart and the computation above allows us to conclude.
\end{proof}
Now consider a twisted LG model $(X,s_X)$ over $(S,\Lc_S)$ and assume that $X$ is a regular scheme. As above, let $X_0\xrightarrow{\mathfrak{i}}X$ be the derived zero section of $s_X$ in $X$ and let $j:X_{\mathcal{U}}=X-X_0\rightarrow X$ be the corresponding open embedding. In this case, we have that $\sing{X,s_X}\simeq \sing{X_0}$. Consider the diagram in $\textup{\textbf{dgCat}}_X^{\textup{idm}}$ and its image in $\textup{\textbf{SH}}_X$
\begin{equation}
    \begindc{\commdiag}[13]
    \obj(-50,30)[1]{$\perf{X_0}$}
    \obj(0,30)[2]{$\cohb{X_0}$}
    \obj(50,30)[3]{$\sing{X_0}$}
    \obj(0,0)[4]{$\perf{X}$}
    \obj(0,-30)[5]{$\perf{X_{\mathcal{U}}}$}
    \mor{1}{2}{$$}[\atright,\injectionarrow]
    \mor{2}{3}{$$}
    \mor{2}{4}{$\mathfrak{i}_*$}
    \mor{4}{5}{$j^*$}
    \mor{1}{4}{$\mathfrak{i}_*$}[\atright,\solidarrow]
\enddc
\end{equation}
\begin{equation}
\begindc{\commdiag}[13]
    \obj(-70,30)[1a]{$\Mcor^{\vee}_X(\perf{X_0})$}
    \obj(0,30)[2a]{$\Mcor^{\vee}_X(\cohb{X_0})$}
    \obj(70,30)[3a]{$\Mcor^{\vee}_X(\sing{X_0})$}
    \obj(0,0)[4a]{$\Mcor^{\vee}_X(\perf{X})$}
    \obj(0,-30)[5a]{$\Mcor^{\vee}_X(\perf{X_{\mathcal{U}}}).$}
    \mor{1a}{2a}{$$}
    \mor{2a}{3a}{$$}
    \mor{2a}{4a}{$\Mcor^{\vee}_X(\mathfrak{i}_*)$}
    \mor{4a}{5a}{$j^*$}
    \mor{1a}{4a}{$\Mcor^{\vee}(\mathfrak{i}_*)$}[\atright,\solidarrow]
    \obj(-80,0)[6]{$\rightsquigarrow$}
    \obj(-80,-10)[7]{$\Mv_X$}
    \enddc
\end{equation}
By the previous results, the second diagram can be rewritten and completed as follows
\begin{equation}\label{fundamental diagram motivic realization sing}
    \begindc{\commdiag}[18]
    \obj(-60,30)[1a]{$\Mcor^{\vee}_X(\perf{X_0})$}
    \obj(0,30)[2a]{$i_*i^!\bu_X$}
    \obj(60,30)[3a]{$\Mcor^{\vee}_X(\sing{X_0})$}
    \obj(0,0)[4a]{$\bu_X$}
    \obj(0,-30)[5a]{$j_*j^*\bu_X$}
    \obj(-60,0)[6a]{$j_!j^*\bu_X$}
    \obj(60,0)[7a]{$i_*i^*\bu_X$}
    \mor{1a}{2a}{$$}
    \mor{2a}{3a}{$$}
    \mor{2a}{4a}{$\Mcor^{\vee}_X(\mathfrak{i}_*)$}
    \mor{4a}{5a}{$\text{counit }(j^*,j_*)$}
    \mor{1a}{4a}{$\Mcor^{\vee}(\mathfrak{i}_*)$}
    \mor{6a}{4a}{}
    \mor{4a}{7a}{}
    \mor{6a}{5a}{$\alpha$}
    \enddc
\end{equation}
where $i=\mathfrak{i}\circ t: \pi_0(X_0)\rightarrow X$ is the closed embedding of the underlying scheme of $X_0$ in $X$.
\begin{rmk}
The morphism $\bu_X\rightarrow i_*i^*\bu_X$ can be factored as
\begin{equation}
    \underbracket{\Mcor^{\vee}_X(\perf{X})}_{\bu_X\simeq }\xrightarrow{\Mcor^{\vee}_X(\mathfrak{i}^*)} \Mcor^{\vee}_X(\perf{X_0}) \xrightarrow{\Mcor^{\vee}_X(t^*)} \underbracket{\Mcor^{\vee}_X(\perf{\pi_0(X_0)})}_{\simeq i_*i^*\bu_X}.
\end{equation}
\end{rmk}
\begin{lmm}
The endomorphism 
\begin{equation}
\Mcor^{\vee}_X(\mathfrak{i}^*)\circ \Mcor^{\vee}_X(\mathfrak{i}_*):\Mcor^{\vee}_X(\perf{X_0})\rightarrow \Mcor^{\vee}_X(\perf{X_0})
\end{equation}
is homotopic to $1-m_{\Lc^{\vee}_{X_0}}$, where $m_{\Lc^{\vee}_{X_0}}$ is the autoequivalence induced by $-\otimes \Lc^{\vee}_{X_0}:\perf{X_0}\rightarrow \perf{X_0}$.
\end{lmm}
\begin{proof}
Indeed, by Remark \ref{actions}, the endomorphism $\mathfrak{i}^*\mathfrak{i}_*$ is equivalent to $-\boxtimes (\Lc^{\vee}_S\xrightarrow{0}\Oc_S)$, the identity is equivalent to $-\boxtimes \Oc_S$ and $-\boxtimes \Lc^{\vee}_S$ corresponds to $-\otimes \Lc^{\vee}_{X_0}:\perf{X_0}\rightarrow \perf{X_0}$. Then, considering the cofiber sequence of dg functors
\begin{equation}
    \begindc{\commdiag}[15]
    \obj(-30,15)[1]{$-\otimes_{\Oc_S}\Lc_S^{\vee}$}
    \obj(30,15)[2]{$-\otimes_{\Oc_S}\Oc_S \simeq id$}
    \obj(-30,-15)[3]{$0$}
    \obj(30,-15)[4]{$\mathfrak{i}^*\mathfrak{i}_*$}
    \mor{1}{2}{$0$}
    \mor{1}{3}{$$}
    \mor{2}{4}{$$}
    \mor{3}{4}{$$}
    \enddc
\end{equation}

we see that $\mathfrak{i}^*\mathfrak{i}_*$ is equivalent to the dg functor
\begin{equation}
    \perf{X_0}\rightarrow \perf{X_0}
\end{equation}
\begin{equation*}
    \Ec\mapsto \Ec \oplus \Ec\otimes_{\Oc_{X_0}}\Lc^{\vee}_{X_0}[1].
\end{equation*}
By Remark \ref{multiplication by perfect complex in HK} we conclude that
\begin{equation}
    \Mv_X(\mathfrak{i}^*\mathfrak{i}_*)\simeq m_{\Oc_{X_0}\oplus \Lc^{\vee}_{X_0}[1]}\simeq 1-m_{\Lc^{\vee}_{X_0}}.
\end{equation}
\end{proof}
\begin{cor}\label{fundamental fiber cofiber sing}
Let $(X,s_X)\in \tlgm{S}$ and assume that $X$ is regular. Then there is a fiber-cofiber sequence in $\md_{\bu_X}(\textup{\textbf{SH}}_X)$
\begin{equation}
\begindc{\commdiag}[15]
   \obj(0,20)[1]{$\Mcor^{\vee}_X(\sing{X,s_X})$}
   \obj(0,0)[2]{$cofiber(\Mcor^{\vee}_X(t^*)\circ (1-m_{\Lc^{\vee}_{X_0}}):\Mcor^{\vee}_X(\perf{X_0})\rightarrow i_*i^*\bu_X)$}
   \obj(0,-20)[3]{$cofiber(\alpha:j_!j^*\bu_X\rightarrow j_*j^*\bu_X)$.}
   \mor{1}{2}{$$}
   \mor{2}{3}{$$}
\enddc
\end{equation}
In particular, if we apply the $\oo$-functor $i^*$, we get the following fiber-cofiber sequence:
\begin{equation}
\begindc{\commdiag}[15]
   \obj(0,20)[1]{$i^*\Mcor^{\vee}_X(\sing{X,s_X})$}
   \obj(0,0)[2]{$i^*cofiber(\Mcor^{\vee}_X(\perf{X_0})\xrightarrow{\Mcor^{\vee}_X(t^*)\circ (1-m_{\Lc^{\vee}_{X_0}})} i_*i^*\bu_X)$}
   \obj(0,-20)[3]{$i^*j_*j^*\bu_X.$}
   \mor{1}{2}{$$}
   \mor{2}{3}{$$}
   
\enddc
\end{equation}
\end{cor}
\begin{proof}
The second statement is an immediate consequence of the first and of the equivalence $i^*j_!\simeq 0$. The first fiber-cofiber sequence can be obtained by applying the octahedreon axiom to the triangle
\begin{equation*}
    \begindc{\commdiag}[15]
    \obj(-30,15)[1]{$\Mv_X(\perf{X_0})$}
    \obj(30,15)[2]{$i_*i^!\bu_X$}
    \obj(30,-15)[3]{$i_*i^*\bu_X$}
    \mor{1}{2}{$$}
    \mor{2}{3}{$$}
    \mor{1}{3}{$\Mv_X(i^*)\circ \Mv_X(\mathfrak{i}_*)$}[\atright,\solidarrow]
    \enddc
\end{equation*}
which appears in diagram (\ref{fundamental diagram motivic realization sing}) and by the fact that
\begin{equation*}
    cofib(\alpha)\simeq cofib(i_*i^!\bu_X\rightarrow i_*i^*\bu_X)
\end{equation*}
(see \cite[Lemma 4.4.2.8]{pi20}).
\end{proof}
Let $\Mv_{S,\Q}$ denote the composition
\begin{equation}
    \dgcatmo \rightarrow \md_{\bu_{S}}(\sh)^{\otimes}\xrightarrow{-\otimes \textup{H}\Q} \md_{\bu_{S,\Q}}(\sh)^{\otimes}.
\end{equation}
We shall now study the commutative algebra object 
\begin{equation}
\Mcor^{\vee}_{S,\Q}(\sing{S,0})\underbracket{\simeq}_{S \text{ regular}}\Mcor^{\vee}_{S,\Q}(\sing{S_0}).
\end{equation}
This is a particularly important object for our purposes as, for every $(X,s_X)\in \tlgm{S}$, the $\Q$-linear motivic realization of $\sing{X,s_X}$ lies in the category of $\Mcor^{\vee}_{S,\Q}(\sing{S_0})$-modules $\md_{\Mcor^{\vee}_{S,\Q}(\sing{S_0})}(\sh)$.
\begin{prop}\label{motivic realization sing(S,0) 1}
There is an equivalence in $\textup{CAlg}(\sh)$:
\begin{equation}
    \Mcor^{\vee}_{S,\Q}(\sing{S,0})\simeq cofiber(\bu_{S,\Q}\xrightarrow{1-m_{\Lc^{\vee}_S}}\bu_{S,\Q})
\end{equation}
Here 
\begin{equation}
    cofiber(\bu_{S,\Q}\xrightarrow{1-m_{\Lc^{\vee}_S}}\bu_{S,\Q}):=\bu_{S,\Q}\otimes_{\textup{H}\Q[t]}\textup{H}\Q \in \textup{CAlg}(\sh),
\end{equation}
where we consider the morphisms $\textup{H}\Q[t]\rightarrow \textup{H}\Q$ induced by $t\mapsto 0$ and $\textup{H}\Q[t]\rightarrow \bu_{S,\Q}$ induced by $t\mapsto 1-m_{\Lc_S^{\vee}}$.
\end{prop}
\begin{proof}
The underlying object of the commutative algebra
$\bu_{S,\Q}\otimes_{\textup{H}\Q[t]}\textup{H}\Q$ is $cofiber(\bu_{S,\Q}\xrightarrow{1-m_{\Lc^{\vee}_S}}\bu_{S,\Q})$. This follows easily from the fact that 
\begin{equation}
    \textup{H}\Q[t]\xrightarrow{t}\textup{H}\Q[t]
\end{equation}
is a free resolution of $\textup{H}\Q$.
Consider the diagram
\begin{equation}
    \begindc{\commdiag}[20]
    \obj(-40,10)[1]{$S_0$}
    \obj(0,10)[2]{$S$}
    \obj(50,10)[3]{$\emptyset$}
    \obj(-40,-10)[4]{$S$}
    \obj(0,-10)[5]{$\V:=\V(\Lc_S)$}
    \obj(50,-10)[6]{$\mathcal{U}=\V-S$,}
    \mor{1}{2}{$\mathfrak{i}$}
    \mor{3}{2}{$j$}
    \mor{4}{5}{$i_0$}
    \mor{6}{5}{$j_0$}
    \mor{1}{4}{$$}
    \mor{2}{5}{$i_0$}
    \mor{3}{6}{$$}
    \enddc
\end{equation}
where both squares are (homotopy) cartesian. Notice that in this case $\pi_0(S_0)=S$ and therefore $i=\mathfrak{i}\circ t=id$. Then, the fiber-cofiber sequence of the previous corollary gives us an equivalence
\begin{equation}
    \Mcor^{\vee}_{S,\Q}(\sing{S_0})\simeq cofiber(\Mcor^{\vee}_{S,\Q}(\perf{S_0}\xrightarrow{\Mcor^{\vee}_{S,\Q}(t^*)\circ (1-m_{\Lc^{\vee}_S})}\bu_{S,\Q})).
\end{equation}
Since the square
\begin{equation}
    \begindc{\commdiag}[18]
    \obj(-40,15)[1]{$\Mcor^{\vee}_{S,\Q}(\perf{S_0})$}
    \obj(40,15)[2]{$\Mcor^{\vee}_{S,\Q}(\perf{S_0})$}
    \obj(-40,-15)[3]{$\Mcor^{\vee}_{S,\Q}(\perf{S})$}
    \obj(40,-15)[4]{$\Mcor^{\vee}_{S,\Q}(\perf{S})$}
    \mor{1}{2}{$1-m_{\Lc^{\vee}_{S_0}}$}
    \mor{1}{3}{$\Mcor^{\vee}_S(t^*)$}
    \mor{2}{4}{$\Mcor^{\vee}_S(t^*)$}
    \mor{3}{4}{$1-m_{\Lc^{\vee}_{S}}$}
    \enddc
\end{equation}
is commutative up to coherent homotopy, to get the first desired equivalence
\begin{equation}
    \Mv_{S,\Q}(\sing{S,0})\simeq \Mv_{S,\Q}(\sing{S_0})
\end{equation}
\begin{equation*}
\simeq cofiber(\Mcor^{\vee}_{S,\Q}(\perf{S_0}\xrightarrow{\Mcor^{\vee}_{S,\Q}(t^*)\circ (1-m_{\Lc^{\vee}_{S}})}\bu_{S,\Q}))
\end{equation*}
it suffices to show that $\Mcor^{\vee}_{S,\Q}(t^*)$ is an equivalence. This is true as $\mathfrak{i}\circ t=id$ and $t\circ \mathfrak{i}$ is homotopic to the identity (see \cite[Remark 3.31]{brtv}).
\end{proof}
\begin{prop}\label{motivic realization sing(S,0) 2}
There is an equivalence in $\textup{CAlg}(\sh)$
\begin{equation}
     \Mcor^{\vee}_{S,\Q}(\sing{S,0}) \simeq i_0^*j_{0*}\bu_{\mathcal{U},\Q}.
\end{equation}
Here $i_0:S\rightarrow \V(\Lc_S)$ is the zero section, $j_0:\mathcal{U}=\V(\Lc_S)-S\rightarrow \V(\Lc_S)$ the open complementary.

\end{prop}
\begin{proof}
Let $\V:=\V(\Lc_S)$. Consider the equivalence that we get from the localization sequence of $(i_0,j_0)$
\begin{equation}
    i_0^*j_{0*}\bu_{\mathcal{U},\Q}\simeq cofiber(i^*_0i_{0*}i^!_0\bu_{\V,\Q}\simeq i_0^!\bu_{\V,\Q}\xrightarrow{c} \bu_{S,\Q}\simeq i^*_0\bu_{\V,\Q}),
\end{equation}
where the map $c$ is the one induced by the counit $i_{0*}i_0^!\bu_{\V,\Q}\rightarrow \bu_{\V,\Q}$. Notice that $i_0:S\rightarrow V$ is a closed embedding between regular schemes. In particular, absolute purity holds. It follows from \cite[Remark 13.5.5]{cd19} that the composition
\begin{equation}
    \bu_{\V,\Q}\rightarrow i_{0*}i^*_0\bu_{\V,\Q}\xrightarrow{\eta_{i_0}}\bu_{\V,\Q}
\end{equation}
corresponds to $1-m_{\Lc^{\vee}_S}$, as the conormal sheaf of $i_0:S\rightarrow \V$ is $\Lc_{S}^{\vee}$. If we apply $i^!$ we obtain
\begin{equation}
   i^!(1-m_{\Lc^{\vee}_S}):  i^!\bu_{\V,\Q}\xrightarrow{c}\bu_{S,\Q}\underbracket{\xrightarrow{i^!\eta_{i_0}'}}_{\simeq \text{ abs. pur.}} i^!\bu_{\V,\Q}, 
\end{equation}
which under the equivalence $\bu_{S,\Q}\simeq i^!\bu_{\V,\Q}$ corresponds to $1-m_{\Lc^{\vee}_S}$.
\end{proof}
\begin{construction}
Let $(X,s_X)\in \tlgm{S}$. Consider the morphism
\begin{equation}
    \bu_X\xrightarrow{1-m_{\Lc^{\vee}_X}} \bu_X
\end{equation} 
Since $j:X_{\mathcal{U}}\rightarrow X$ is the complementary open subscheme to the zero locus of the section $s_X$, it follows that $j^*\Lc^{\vee}_X\simeq \Oc_{X_{\mathcal{U}}}$. In particular, 
\begin{equation}
j^*\circ (1-m_{\Lc^{\vee}_X})\simeq 1-m_{\Lc^{\vee}_{X_{\mathcal{U}}}}\sim 0.
\end{equation}
Then we obtain a morphism
\begin{equation}
   sp^{\textup{mot}}_{X}: cofiber(\bu_X\xrightarrow{1-m_{\Lc^{\vee}_X}}\bu_X)\rightarrow j_*j^*\bu_X.
\end{equation}
\end{construction}
\begin{prop}
Let $(X,s_X)\in \tlgm{S}$ and assume that $X$ is regular and that $X\xrightarrow{s_X}\V(\Lc_S)\xleftarrow{\text{zero}}S$ is Tor-independent (i.e. $ X\times_{\V(\Lc_S)}S\simeq X\times^h_{\V(\Lc_S)}S$). Then
\begin{equation}
    i^*\Mcor^{\vee}_X(\sing{X,s_X})\simeq fiber(i^*sp^{\textup{mot}}_X).
\end{equation}
\begin{proof}
This follows immediately from the octahedron property applied to the triangle in the following diagram
\begin{equation}
    \begindc{\commdiag}[20]
    \obj(-30,15)[1]{$\bu_{X_0}$}
    \obj(0,15)[2]{$i^!\bu_X$}
    \obj(40,15)[3]{$i^*\Mcor^{\vee}_X(\sing{X_0})$}
    \obj(0,0)[4]{$\bu_{X_0}$}
    \obj(0,-15)[5]{$i^*j_*j^*\bu_X$}
    \mor{1}{2}{$$}
    \mor{2}{3}{$$}
    \mor{2}{4}{$$}
    \mor{4}{5}{$$}
    \mor{1}{4}{$1-m_{\Lc^{\vee}_{X_0}}$}[\atright,\solidarrow]
    \enddc
\end{equation}
and from the compatibility of $1-m_{\Lc^{\vee}_X}$ with pullbacks.
\end{proof}
\end{prop}
\section{Monodromy-invariant vanishing cycles}\label{Monodromy-invariant vanishing cycles}
\subsection{The formalism of vanishing cycles}\label{the formalism of vanishing cycles}
We shall begin with a quick review of the formalism of vanishing cycles.
We hope that this will be useful to understand the analogy which led to the definition of monodromy-invariant vanishing cycles.
\begin{notation}\label{notation trait}
Throughout this section, $A$ will be an excellent henselian trait and $S$ will denote the associated affine scheme. Label $k$ its residue field and $K$ its fraction field. Let $\sigma$ be $Spec(k)$ and $\eta$ be $Spec(K)$. Fix algebraic closures $k^{alg}$ and $K^{alg}$ of $k$ and $K$ respectively. We will consider:
\begin{itemize}
\item The maximal separable extension $k^{sep}$ of $k$ inside $k^{alg}$. We will use the notation $\bar{\sigma}=Spec(k^{sep})$.
\item The maximal unramified extension $K^{unr}$ of $K$ in $K^{alg}$. We will use the notation $\eta^{unr}=Spec(K^{unr})$.
\item The maximal tamely ramified extension $K^t$ of $K$ inside $K^{alg}$. We will use the notation $\eta^t=Spec(K^t)$.
\item The maximal separable extension $K^{sep}$ of $K$ inside $K^{alg}$. We will use the notation $\bar{\eta}=Spec(K^{sep})$.
\end{itemize}
Moreover, we will fix an uniformizer $\pi$. 
\end{notation}
\begin{rmk}
It is well known that there is an equivalence between the category of separable extensions of $k$ and that of unramified extensions of $K$. In particular, $\textup{Gal}(k^{sep}/k)\simeq \textup{Gal}(K^{unr}/K)$. This, together with the fundamental theorem of Galois theory, implies that there is an exact sequence of groups
\begin{equation}
    1\rightarrow \textup{Gal}(K^{sep}/K^{unr})\rightarrow \textup{Gal}(K^{sep}/K)\rightarrow \textup{Gal}(k^{sep}/k)\rightarrow 1.
\end{equation}
The Galois group on the left is called the inertia group and it is usually denoted by $I$. The chain of extensions $K^{unr}\subseteq K^{t}\subseteq K^{sep}$ gives us the following decomposition of $I$:
\begin{equation}
    1\rightarrow \underbracket{\textup{Gal}(K^{sep}/K^{t})}_{=:I_w} \rightarrow I \rightarrow \underbracket{\textup{Gal}(K^{sep}/K)}_{=:I_t}\rightarrow 1.
\end{equation}
The Galois group $I_w$ is called wild inertia group, while $I_t$ is called tame inertia group. 
See \cite{se62} for more details.
\end{rmk}
\begin{rmk}
Let $A^{sh}$ be a strict henselianisation of $A$ and let $\bar{S}=Spec(A^{sh})$. We then have the following picture
\begin{equation}
    \begindc{\commdiag}[20]
    \obj(0,15)[1]{$\bar{\sigma}$}
    \obj(30,15)[2]{$\bar{S}$}
    \obj(60,15)[3]{$\eta^{unr}$}
    \obj(0,-5)[4]{$\sigma$}
    \obj(30,-5)[5]{$S$}
    \obj(60,-5)[6]{$\eta$}
    \obj(90,15)[7]{$\eta^t$}
    \obj(120,15)[8]{$\bar{\eta}$}
    \mor{1}{2}{$\bar{i}$}[\atright,\injectionarrow]
    \mor{3}{2}{$j_{unr}$}[\atleft,\injectionarrow]
    \mor{4}{5}{$i$}[\atright,\injectionarrow]
    \mor{6}{5}{$j$}[\atright,\injectionarrow]
    \mor{1}{4}{$v_{\sigma}$}
    \mor{2}{5}{$v$}
    \mor{3}{6}{$v_{\eta}$}
    \mor{7}{3}{$$}
    \mor{8}{7}{$$}
    \cmor((90,18)(90,21)(87,24)(60,24)(34,24)(31,21)(31,19)) \pdown(60,22){$j_t$}
    \cmor((120,18)(120,24)(117,27)(74,27)(32,27)(29,24)(29,19)) \pdown(75,32){$\bar{j}$}
    \enddc
\end{equation}
where both squares are cartesian. In particular, notice that $(\sigma,\eta)$ and $(\bar{\sigma},\eta^{unr})$ form closed-open coverings of $S$ and $\bar{S}$ respectively.
\end{rmk}
Let $p:X\rightarrow S$ be of finite type and let $\bar{X}=X\times_S \bar{S}$. Consider the following diagram, cartesian over the base (this is also called Grothendieck's trick in the literature)
\begin{equation}
    \begindc{\commdiag}[20]
    \obj(0,10)[1]{$X_{\bar{\sigma}}$}
    \obj(30,10)[2]{$\bar{X}$}
    \obj(60,10)[3]{$X_{\eta^{unr}}$}
    \obj(90,10)[4]{$X_{\eta^t}$}
    \obj(120,10)[5]{$X_{\bar{\eta}}$}
    \obj(0,-10)[6]{$\bar{\sigma}$}
    \obj(30,-10)[7]{$\bar{S}$}
    \obj(60,-10)[8]{$\eta^{unr}$}    \obj(90,-10)[9]{$\eta^t$}
    \obj(120,-10)[10]{$\bar{\eta}$.}
    \mor{1}{2}{$$}[\atright,\injectionarrow]
    \mor{3}{2}{$$}[\atright,\injectionarrow]
    \mor{4}{3}{$$}
    \mor{5}{4}{$$}
    \mor{6}{7}{$$}[\atright,\injectionarrow]
    \mor{8}{7}{$$}[\atright,\injectionarrow]
    \mor{9}{8}{$$}
    \mor{10}{9}{$$}
    \mor{1}{6}{$$}
    \mor{2}{7}{$$}
    \mor{3}{8}{$$}
    \mor{4}{9}{$$}
    \mor{5}{10}{$$}
    \enddc
\end{equation}

\begin{rmk}
The $\oo$-category $\shvl(\bar{X})$ is the recollement of $\shvl(X_{\bar{\sigma}})$ and $\shvl(X_{\eta^{unr}})$ in the sense of \cite[Appendix A.8]{ha}\footnote{Strictly speaking, one should consider the full subcategories of $\shvl(\bar{X})$ spanned by the images of $\bar{i}_*$ and $j_{\eta^{unr}}$ which are closed under equivalences.}. Indeed, $\shvl(\bar{X})$ is a stable $\oo$-category and therefore it admits finite limits. Moreover, both $\bar{i}_*:\shvl(X_{\bar{\sigma}})\rightarrow \shvl(\bar{X})$ and $j^{unr}_*: \shvl(X_{\eta^{unr}})\rightarrow \shvl(\bar{X})$ are fully faithful $\oo$-functors: these are immediate consequences of the proper and smooth base change theorems applied to the cartesian squares
\begin{equation}
    \begindc{\commdiag}[15]
    \obj(-15,15)[1]{$X_{\bar{\sigma}}$}
    \obj(15,15)[2]{$X_{\bar{\sigma}}$}
    \obj(-15,-15)[3]{$X_{\bar{\sigma}}$}
    \obj(15,-15)[4]{$\bar{X}$}
    \mor{1}{2}{$id$}
    \mor{1}{3}{$id$}
    \mor{2}{4}{$\bar{i}$}
    \mor{3}{4}{$\bar{i}$}
    \obj(45,15)[5]{$X_{\bar{\sigma}}$}
    \obj(75,15)[6]{$X_{\bar{\sigma}}$}
    \obj(45,-15)[7]{$X_{\bar{\sigma}}$}
    \obj(75,-15)[8]{$\bar{X}$.}
    \mor{5}{6}{$id$}
    \mor{5}{7}{$id$}
    \mor{6}{8}{$j_{\eta^{unr}}$}
    \mor{7}{8}{$j_{\eta^{unr}}$}
    \enddc
\end{equation}
It is known that the conditions of \cite[Definition A.8.1]{ha} are verified. In particular, using the equivalence of $\shvl(X_{\eta^{unr}})$ with the $\oo$-category of $\ell$-adic sheaves on the generic geometric fiber endowed with a continuous action of $I$  $\shvl(X_{\bar{\eta}})^I$, we have that $\shvl(\bar{X})$ is the recollement of $\shvl(X_{\bar{\sigma}})$ and $\shvl(X_{\bar{\eta}})^{I}$.
\end{rmk}
\begin{rmk}
Notice that we have the following diagram where each arrow is an equivalence:
\begin{equation}
    \begindc{\commdiag}[15]
    \obj(-30,20)[1]{$\shvl(X_{\eta^{unr}})$}
    \obj(0,-20)[2]{$\shvl(X_{\eta^t})^{I_t}$.}
    \obj(30,20)[3]{$\shvl(X_{\bar{\eta}})^{I}$}
    \mor{1}{2}{$j_{t}^*$}
    \mor{1}{3}{$\bar{j}^*$}
    \mor{2}{3}{$j_w^*$}
    \cmor((30,25)(30,30)(28,32)(0,32)(-28,32)(-30,30)(-30,25))
    \pdown(0,39){$\bar{j}_*(-)^{\textup{h}I}$}
    \mor(40,17)(13,-17){$j_{w*}(-)^{\textup{h}I_w}$}
    \mor(-13,-17)(-40,17){$j_{t*}(-)^{\textup{h}I_t}$}
    \enddc
\end{equation}
\end{rmk}
\begin{construction}
We will now sketch the construction of Deligne's topos in the $\oo$-categorical world. Notice that, by the recollement technique, the \'etale $\oo$-topos of $\bar{S}$ is the recollement of the \'etale $\oo$-topos of $\bar{\sigma}$ and that of $\eta^{unr}$. Moreover, notice that the \'etale $\oo$-topos of $\eta^{unr}$ is equivalent to the $\oo$-category of spaces with a continuous action of $I$.  For a scheme $Y$ over $\bar{\sigma}$, we consider the lax $\oo$-limit $\tilde{Y}_{et}\times_{\tilde{\bar{\sigma}}_{et}}\tilde{\bar{S}}_{et}$, which exists by \cite[pag. 614]{htt}. The decomposition of $\tilde{\bar{S}}_{et}$ gives us a digram
\begin{equation}
    \begindc{\commdiag}[20]
    \obj(-30,15)[1]{$\tilde{Y}_{et}$}
    \obj(0,15)[2]{$\tilde{Y}_{et}\times_{\tilde{\bar{\sigma}}_{et}}\tilde{\bar{S}}_{et}$}
    \obj(45,15)[3]{$\tilde{Y}_{et}\times_{\tilde{\bar{\sigma}}_{et}}\tilde{\eta^{unr}}_{et}$}
    \obj(-30,-15)[4]{$\tilde{\bar{\sigma}}_{et}$}
    \obj(0,-15)[5]{$\tilde{\bar{S}}_{et}$}
    \obj(45,-15)[6]{$\tilde{\eta^{unr}}_{et}$}
    \mor{1}{2}{$$}
    \mor{3}{2}{$$}
    \mor{1}{4}{$$}
    \mor{2}{5}{$$}
    \mor{3}{6}{$$}
    \mor{4}{5}{$$}
    \mor{6}{5}{$$}
    \enddc
\end{equation}
We will label the $\oo$-category of $\ell$-adic sheaves on $\tilde{Y}_{et}\times_{\tilde{\bar{\sigma}}_{et}}\tilde{\eta^{unr}}_{et}$ by $\shvl(Y)^I$, as it is the $\oo$-category of $\ell$-adic sheaves on $Y$ endowed with a continuous action of $I$. 
The $\oo$-category $\shvl(Y)^{I_t}$ of $\ell$-adic sheaves on $Y$ endowed with a continuous action of $I_t$ identifies with the full subcategory of $\shvl(Y)^I$ such that the induced action of $I_w$ is trivial.
\end{construction}
\begin{defn} Let $p:\bar{X}\rightarrow \bar{S}$ be an $\bar{S}$-scheme. 
Let $\Ec\in \shvl(\bar{X}_{\eta^{unr}})$. The $\ell$-adic sheaf of tame nearby cycles of $\Ec$ is defined as
\begin{equation}
\Psi_t(\Ec):=\bar{i}^*j_{t*}\Ec_{|X_{\eta_t}} \in \shvl(X_{\bar{\sigma}})^{I_t}.
\end{equation}
Analogously, the $\ell$-adic sheaf of nearby cycles of $\Ec$ is defined as
\begin{equation}
    \Psi(\Ec):=\bar{i}^*\bar{j}_*\Ec_{|X_{\bar{\eta}}}\in \shvl(X_{\bar{\sigma}})^I.
\end{equation}
\end{defn}
\begin{rmk}
With the same notation as above, the following equivalence holds in $\shvl(\bar{X}_{\bar{\sigma}})^{I_t}$:
\begin{equation}
    \Psi_t(\Ec)\simeq \Psi(\Ec)^{\textup{h}I_w}.
\end{equation}
\end{rmk}
\begin{defn}
Let $p:\bar{X}\rightarrow \bar{S}$ be an $\bar{S}$-scheme. Let $\Fc \in \shvl(\bar{X})$. The unit of the adjuction $(j_t^*,j_{t*})$ induces a morphism in $\shvl(\bar{X}_{\bar{\sigma}})^{I_t}$
\begin{equation}
    \bar{i}^*\Fc \rightarrow \Psi_t(\Fc_{\bar{X}_{\eta^{unr}}}),
\end{equation}
where we regard the object on the left endowed with the trivial $I_t$-action. The cofiber of this morphism is by definition the $\ell$-adic sheaf of tame vanishing cycles of $\Fc$, which we will denote by $\Phi_t(\Fc)$. In a similar way, we define the $\ell$-adic sheaf of vanishing cycles $\Phi(\Fc)$ as the cofiber of the morphism (called the specialization morphism in literature)
\begin{equation}
    \bar{i}^*\Fc \rightarrow \Psi(\Fc_{\bar{X}_{\bar{\eta}}})
\end{equation}
induced by the unit of the adjunction $(\bar{j}^*,\bar{j}_*)$.
\end{defn}
\subsection{Monodromy-invariant vanishing cycles}
\begin{context}
Assume that $S$ is strictly henselian, i.e. that $k$ is a separably closed field and that $p:X\rightarrow S$ is a proper morphism.
\end{context}
For our purposes, we are interested in the image of $\Phi(\Ql{,X})$ via the $\oo$-functor 
\begin{equation*}
    (-)^{\textup{h}I}:\shvl(X_{\sigma})^{I}\rightarrow \shvl(X_{\sigma}).
\end{equation*}
It is then important to remark that it is possible to determine $p_{\sigma*}\Phi(\Ql{,X})^{\textup{h}I}$ without ever mentioning this $\oo$-functor. Notice that $\Phi(\Ql{,X})^{\textup{h}I}$ is the cofiber of the image of the specialization morphism via $(-)^{\textup{h}I}$:
\begin{equation}\label{fiber cofiber sphI}
    \Ql{,X_0}^{\textup{h}I}\xrightarrow{(sp)^{\textup{h}I}}\Psi(\Ql{,X_{\bar{\eta}}})^{\textup{h}I}\rightarrow \Phi(\Ql{,X})^{\textup{h}I}
\end{equation}
is a fiber-cofiber sequence in $\shvl(X_{\sigma})$.
\begin{rmk}
There are equivalences in $\textup{CAlg}(\shvl(\sigma))$
\begin{itemize}
    \item $\bigl( p_{\sigma*}\Ql{,X_{\sigma}} \bigr)^{\textup{h}I}\simeq p_{\sigma*}\Ql{,X_{\sigma}}\otimes_{\Ql{,\sigma}}\Ql{,\sigma}^{\textup{h}I}$ (see \cite[Proposition 4.32]{brtv}),
    \item $\bigl( p_{\sigma*}i^*\bar{j}_*\Ql{,X_{\bar{\eta}}} \bigr)^{\textup{h}I}\simeq p_{\sigma*}i^*j_*\Ql{,X_{\eta}}$ (see \cite[Proposition 4.31]{brtv}),
    \item $\Ql{,\sigma}^{\textup{h}I}\simeq i_0^*j_{0*}\Ql{,\eta}$ (see \cite[Lemma 4.34]{brtv}).
\end{itemize}
\end{rmk}
Consider the object $p_{\sigma*}p_{\sigma}^*i_0^*j_{0*}\Ql{,\eta}$ in $\shvl(\sigma)$. The following chain of equivalences, together with the previous remark, identify it with $\bigl(p_{\sigma*}\Ql{,X_{\sigma}} \bigr)^{\textup{h}I}$:
\begin{equation}
    p_{\sigma*}p_{\sigma}^*i_0^*j_{0*}\Ql{,\eta}\simeq p_{\sigma*}\bigl(\Ql{,X_{\sigma}}\otimes_{\Ql{,X_{\sigma}}}p_{\sigma}^*i_0^*j_{0*}\Ql{,\eta}\bigr)\underbracket{\simeq}_{\text{proj.form.}}p_{\sigma*}\Ql{,X_{\sigma}}\otimes_{\Ql{,\sigma}}i_0^*j_{0*}\Ql{,\eta}.
\end{equation}
\begin{defn}\label{definition monodromy-invarian t vanishing cycles}
Consider the canonical morphism in $\shvl(X_{\sigma})$
\begin{equation}
    sp_p^{\textup{hm}}:p_{\sigma}^*i_0^*j_{0*}\Ql{,\eta}\simeq i^*p^*j_{0*}\Ql{,\eta}\rightarrow i^*j_*p_{\eta}^*\Ql{,\eta}\simeq i^*j_*\Ql{,X_{\eta}}
\end{equation}
induced by the base change natural transformation $p^*j_{0*}\rightarrow j_*p_{\eta}^*$. We will refer to it as the monodromy-invariant specialization morphism. We will refer to the cofiber of this morphism in $\shvl(X_{\sigma})$ as monodromy-invariant vanishing cycles, which we will denote $\Phi_{p}^{\textup{hm}}(\Ql{})$.
\end{defn}
In order to justify the choice of the name of this morphism, we shall prove the following:
\begin{lmm}\label{comparison monodromy invariant inertia invariant vanishing cycles}
The two arrows $p_{\sigma*}(sp)^{\textup{h}I}$ and $p_{\sigma*}sp_p^{\textup{hm}}$ are homotopic under the equivalences 
\begin{equation*}\bigl( p_{\sigma*}i^*\bar{j}_*\Ql{,X_{\bar{\eta}}} \bigr)^{\textup{h}I}\simeq p_{\sigma*}i^*j_*\Ql{,X_{\eta}},
\end{equation*}
\begin{equation*}\bigl( p_{\sigma*}\Ql{,X_{\sigma}} \bigr)^{\textup{h}I}\simeq p_{\sigma*}\Ql{,X_{\sigma}}\otimes_{\Ql{,\sigma}}\Ql{,\sigma}^{\textup{h}I}\simeq p_{\sigma*}p_{\sigma}^*i_0^*j_{0*}\Ql{,\eta}.
\end{equation*}
\end{lmm}
\begin{proof}
We shall view both arrows as maps $p_{\sigma*}\Ql{,X_{\sigma}}\otimes_{\Ql{,\sigma}}\Ql{,\sigma}^{\textup{h}I}\rightarrow p_{\sigma*}i^*j_*\Ql{,X_{\eta}}$.

Since the tensor product defines a cocartesian symmetric monoidal structure on $\textup{CAlg}(\shvl(\sigma))$ (see \cite{ha}), it suffices to show that the maps of commutative algebras obtained from $p_{\sigma*}(sp)^{\textup{h}I}$ and $p_{\sigma*}sp_p^{\textup{hm}}$ by precomposition with $p_{\sigma*}\Ql{,X_{\sigma}}\rightarrow p_{\sigma*}\Ql{,X_{\sigma}}\otimes_{\Ql{,\sigma}}\Ql{,\sigma}^{\textup{h}I}$ and $\Ql{,\sigma}^{\textup{h}I}\rightarrow p_{\sigma*}\Ql{,X_{\sigma}}\otimes_{\Ql{,\sigma}}\Ql{,\sigma}^{\textup{h}I}$ respectively are equivalent. It is essentially the proof of the main theorem in \cite{brtv} that the precompositions of $p_{\sigma*}(sp)^{\textup{h}I}$ with these two maps are induced by the canonical morphisms $1\rightarrow j_*j^*$ and $1\rightarrow p_{\eta*}p_{\eta}^*$. We are left to show that the same holds true if we consider $sp_p^{\textup{hm}}$ instead of $(sp)^{\textup{h}I}$. Recall that $sp_p^{\textup{hm}}$ is induced by the base change morphism $p^*j_{0*}\rightarrow j_*p_{\eta}^*$, i.e. the morphism which corresponds under adjunction to $j_{0*}\rightarrow j_{0*}p_{\eta*}p_{\eta}^*$, induced by the unit of the adjunction $(p_{\eta}^*,p_{\eta*})$. From this we obtain the following diagram
\begin{equation}
    \begindc{\commdiag}[14]
    \obj(-60,20)[1]{$i_0^*j_{0*}\Ql{,\eta}$}
    \obj(40,20)[2]{$i_0^*j_{0*}p_{\eta*}p_{\eta}^*\Ql{,\eta}\simeq p_{\sigma*}i^*j_*p_{\eta}^*\Ql{,\eta}$}
    \obj(-60,0)[3]{$p_{\sigma*}p_{\sigma}^*i_0^*j_{0*}\Ql{,\eta}$}
    \obj(40,0)[4]{$p_{\sigma*}p_{\sigma}^*p_{\sigma*}i^*j_*p_{\eta}^*\Ql{,\eta}$}
    \obj(40,-20)[6]{$p_{\sigma*}i^*j_*p_{\eta}^*\Ql{,\eta}$.}
    \mor{1}{2}{$1\rightarrow p_{\eta*}p_{\eta}^*$}
    \mor{1}{3}{$1\rightarrow p_{\sigma*}p_{\sigma}^*$}
    \mor{2}{4}{$$}
    \mor{3}{4}{$$}
    \obj(-120,0)[5]{$p_{\sigma*}p_{\sigma}^*\Ql{,\sigma}$}
    \mor{5}{3}{$1\rightarrow j_{0}^*j_{0*}$}
    \mor{3}{6}{$p_{\sigma*}sp_p^{\textup{hm}}$}[\atright,\solidarrow]
    \mor{4}{6}{$$}
    \cmor((65,16)(68,15)(70,14)(75,0)(70,-14)(68,-15)(65,-16)) \pleft(80,0){$id$}
    \cmor((-120,-5)(-70,-16)(-50,-19)(-35,-20)(-20,-20)(0,-20)(15,-20)) \pright(0,-20){$$}
    \enddc
\end{equation}
In particular the composition $i_0^*j_{0*}\Ql{,\eta}\rightarrow p_{\sigma*}p_{\sigma}^*i_0^*j_{0*}\Ql{,\eta} \xrightarrow{p_{\sigma*}sp_p^{\textup{hm}}} p_{\sigma*}i^*j_*\Ql{,X_{\eta}}$ is homotopic to the map induced by $1\rightarrow p_{\eta*}p_{\eta}^*$. 

The composition $p_{\sigma*}p_{\sigma}^*\Ql{,\sigma}\rightarrow p_{\sigma*}p_{\sigma}^*i_0^*j_{0*}\Ql{,\eta} \rightarrow p_{\sigma*}i^*j_*\Ql{,X_{\eta}}$ is homotopic to the map $p_{\sigma*}\Ql{,X_{\sigma}}\rightarrow p_{\sigma*}i^*j_{*}\Ql{,X_{\eta}}$ induced by $1\rightarrow j_*j^*$ by the following commutative triangle
\begin{equation}
    \begindc{\commdiag}[18]
    \obj(-25,0)[1]{$p^*\Ql{,S}$}
    \obj(25,15)[2]{$p^*j_{0*}j_0^*\Ql{,S}$}
    \obj(25,-15)[3]{$j_*p_{\eta}^*j_0^*\Ql{,S}\simeq j_*j^*p^*\Ql{,S}$.}
    \mor{1}{2}{$1\rightarrow j_{0*}j_0^*$}
    \mor{2}{3}{$$}
    \mor{1}{3}{$1\rightarrow j_*j^*$}[\atright,\solidarrow]
    \enddc
\end{equation}
\end{proof}

\begin{rmk}
Notice that what we said above is actually true for any $\ell$-adic sheaf in the image of $p^*:\shvl(S)\rightarrow \shvl(X)$.
\end{rmk}
The advantage of this reformulation is that it allows to define inertia invariant vanishing cycles without ever mention the inertia group. We will adopt it for the situation we are interested in.

Assume that $S$ is a noetherian regular scheme.
\begin{defn}\label{mododromy-invariant vanishing cycles def}
Let $(X,s_X)$ be a twisted LG model over $(S,\Lc_S)$. Consider the diagram obtained from the zero section of the vector bundle associated to $\Lc_X$:
\begin{equation}
  \begindc{\commdiag}[18]
    \obj(0,10)[1]{$\pi_0(X_0)$}
    \obj(30,10)[2]{$X$}
    \obj(0,-10)[3]{$X$}
    \obj(30,-10)[4]{$\V(\Lc_X)$}
    \obj(80,10)[5]{$X_{\mathcal{U}}$}
    \obj(80,-10)[6]{$\mathcal{U}_X:=\V(\Lc_X)-X.$}
    \mor{1}{2}{$i$}
    \mor{1}{3}{$s_0$}
    \mor{2}{4}{$s_X$}
    \mor{3}{4}{$i_0$}[\atright,\solidarrow]
    \mor{5}{2}{$j$}[\atright,\solidarrow]
    \mor{5}{6}{$s_{\mathcal{U}}$}
    \mor{6}{4}{$j_0$}
    \enddc
\end{equation}
Define the monodromy-invariant specialization morphism associated to $(X,s_X)$ as the map in $\shvl(\pi_0(X_0))$
\begin{equation}\label{monodromy-invariant specialization morphism}
   sp_{(X,s_X)}^{\textup{mi}}: s_0^*i_0^*j_{0*}\Ql{,\mathcal{U}_X}\simeq i^*s_{X}^*j_{0*}\Ql{,\mathcal{U}_X}\rightarrow i^*j_*\Ql{,X_{\mathcal{U}}}\simeq i^*j_*s_{\mathcal{U}}^*\Ql{,\mathcal{U}_X}
\end{equation}
induced by the natural transformation $s_X^*j_{0*}\rightarrow j_*s_{\mathcal{U}}^*$. We will refer to the cofiber of $sp_{(X,s_X)}^{\textup{mi}}$ as monodromy invariant vanishing cycles of $(X,s_X)$, that we will denote $\Phi^{\textup{mi}}_{(X,s_X)}(\Ql{})$.
\end{defn}
\begin{prop}\label{cofiber first chern class}
Let $(X,s_X)$ be as above and assume $X$ regular. There is an equivalence
\begin{equation}
    i_0^*j_{0*}\Ql{,\mathcal{U}_X}\simeq cofib \bigl ( c_1(\Lc_X):\Ql{,X}(-1)[-2]\rightarrow \Ql{,X} \bigr ).
\end{equation}
\end{prop}
\begin{proof}
Consider the diagram
\begin{equation*}
    X\xrightarrow{i_0}\V_X:=\V(\Lc_X)\xleftarrow{j_0}\mathcal{U}.
\end{equation*}
The $\ell$-adic sheaf $i_0^*j_{0*}\Ql{,\mathcal{U}}$ is the cofiber of the morphism $i_0^*\bigl( i_{0*}i_0^{!}\Ql{,\V_X}\xrightarrow{\text{counit}} \Ql{,\V_X} \bigr)$ in $\shvl(X)$. Recall that $i_0^*i_{0*}\simeq id$. Therefore, we can consider the following triangle
\begin{equation}
    \begindc{\commdiag}[18]
    \obj(-20,15)[1]{$i_0^!\Ql{,\V_X}$}
    \obj(20,15)[2]{$\Ql{,X}$}
    \obj(0,-15)[3]{$\Ql{,X}(-1)[-2]$.}
    \mor{1}{2}{$$}
    \mor{3}{1}{$\text{abs. pur.}$}
    \obj(-7,0)[4]{$\simeq$}
    \mor{3}{2}{$$}
    \enddc
\end{equation}
The absolute purity isomorphism is given by the class $cl(X)\in \textup{H}^2_X(\V_X,\Ql{}(1))$, whose image in $\textup{H}^2(X,\Ql{}(1))$ is $c_1(\mathcal{N}_{X|\V_X}^{\vee})$, the first Chern class of the conormal bundle (see \cite{fuj02}). The map $\Ql{,X}\rightarrow \Ql{,X}(1)[2]$ corresponding to this class is the image of the right vertical arrow in the diagram above via the $\oo$-functor $-(1)[2]$. It suffices to notice that $\mathcal{N}_{X|\V_X}^{\vee}\simeq \Lc_X$ to conclude.
\end{proof}
\section{The comparison theorem}
\subsection{\texorpdfstring{$2$}{2}-periodic \texorpdfstring{$\ell$}{l}-adic sheaves}
We shall now approach the comparison between monodromy-invariant vanishing cycles and the $\ell$-adic realization of the dg category of singularity $\sing{X_0}$. In order to do so, we need to work with the category of $\Ql{,S}(\beta)$-modules in $\shvl(S)$. Recall that there is an adjucntion of $\oo$-functors
\begin{equation}
    \shvl(S)\rightleftarrows \md_{\Ql{,S}(\beta)}(\shvl(S))
\end{equation}
given by $-\otimes_{\Ql{,S}}\Ql{,S}(\beta)$ and by the forgetful functor.

Notice that $\md_{\Ql{,\bullet}}(\shvl(\bullet))$ defines a fibered category over the category of schemes, which satisfies Grothendieck's six functors formalism.
\begin{defn}\label{Phimi}
With the same notation as above let 
\begin{equation}\label{definition spmi beta}
    s_0^*i_0^*j_{0*}\Ql{,\mathcal{U}_X}(\beta)\simeq i^*s_X^*j_{0*}\Ql{,\mathcal{U}_X}(\beta)\rightarrow i^*j_*s_{\mathcal{U}}^*\Ql{,\mathcal{U}_X}(\beta)
\end{equation}
be the arrow in $\md_{\Ql{,X}(\beta)}(\shvl(X))$ induced by the base change morphism $s_X^*j_{0*}\rightarrow j_*s_{\mathcal{U}}^*$. Denote by $\Phi^{\textup{mi}}_{(X,s_X)}(\Ql{}(\beta))$ its cofiber.
\end{defn}

\begin{prop}\label{compatibility Phimi with beta}
The maps $(\ref{definition spmi beta})$ and $sp^{\textup{mi}}_{(X,s_X)}\otimes_{\Ql{,X}}\Ql{,X}(\beta)$ (see Definition \ref{monodromy-invariant specialization morphism}) are homotopic. In particular,
there is an equivalence 
\begin{equation}
    \Phi^{\textup{mi}}_{(X,s_X)}(\Ql{})\otimes_{\Ql{,X}}\Ql{,X}(\beta)\simeq \Phi_{(X,s_X)}^{\textup{mi}}(\Ql{,X}(\beta)).
\end{equation}
in $\md_{\Ql{,X}(\beta)}\bigl( \shvl(X) \bigr)$.
\end{prop}
\begin{proof}This is analogous to \cite[Proposition 4.28]{brtv}. As a first step, notice that both $*$-pullbacks and $*$-pushforwards commute with Tate and usual shifts. Since $-\otimes_{\Ql{}}\Ql{}(\beta)$ commutes with $*$-pullbacks, we have
\begin{equation}
    i^*s_X^*j_{0*}\Ql{,\mathcal{U}_X}\otimes_{\Ql{,X}}\Ql{,X}(\beta)\simeq i^*s_X^*(j_{0*}\Ql{,\mathcal{U}_X}\otimes_{\Ql{,\V_X}}\Ql{,\V_X}(\beta)),
\end{equation}
\begin{equation}
     i^*j_{*}\Ql{,X_{\mathcal{U}}}\otimes_{\Ql{,\pi_0(X)}}\Ql{,\pi_0(X)}(\beta)\simeq i^*(j_{*}\Ql{,X_{\mathcal{U}}}\otimes_{\Ql{,X}}\Ql{,X}(\beta)).
\end{equation}
Using the equivalence $\Ql{,\V_X}(\beta)\simeq \bigoplus_{i\in \Z}\Ql{,\V_X}(i)[2i]$, we see that
\begin{equation}
    j_{0*}\Ql{,\mathcal{U}_X}\otimes_{\Ql{,\V_X}}\Ql{,\V_X}(\beta)\simeq \bigoplus_{i\in \Z}(j_{0*}\Ql{,\mathcal{U}_X}(i)[2i]).
\end{equation}
We shall show that the canonical map 
\begin{equation}
    \bigoplus_{i\in \Z}(j_{0*}\Ql{,\mathcal{U}_X}(i)[2i])\rightarrow j_{0*}(\bigoplus_{i\in \Z}\Ql{,\mathcal{U_X}}(i)[2i])
\end{equation}
is an equivalence. This follows immediately from the fact that the $*$-pushforward commutes with filtered colimits and from the equivalence $\bigoplus_{i\in \Z} \Ql{,\mathcal{U}_X}(i)[2i]\simeq colim_{i\geq 0}\bigoplus_{k=-i}^{i}\Ql{,\mathcal{U}_X}(k)[2k]$. This shows that
\begin{equation}
i^*s_X^*j_{0*}\Ql{,\mathcal{U}_X}\otimes_{\Ql{,X}}\Ql{,X}(\beta)\simeq i^*s_X^*j_{0*}\Ql{,\mathcal{U}_X}(\beta).
\end{equation}
The same argument applies to show that 
\begin{equation}
i^*j_{*}\Ql{,X_{\mathcal{U}}}\otimes_{\Ql{,\pi_0(X)}}\Ql{,\pi_0(X)}(\beta)\simeq i^*j_{*}\Ql{,X_{\mathcal{U}}}(\beta).
\end{equation}
To show that the two maps are homotopic, it suffices to show that they are so before applying $i^*$. Notice that $(\Ql{,\mathcal{U}_X}\rightarrow s_{\mathcal{U}*}s_{\mathcal{U}}^*\Ql{,\mathcal{U}_X})\otimes_{\Ql{,\mathcal{U}_X}}\Ql{,\mathcal{U}_X}(\beta)$ is homotopic to $\Ql{,\mathcal{U}_X}(\beta)\rightarrow s_{\mathcal{U}*}s_{\mathcal{U}}^*\Ql{,\mathcal{U}_X}(\beta)$, as $-\otimes_{\Ql{,\mathcal{U}_X}}\Ql{,\mathcal{U}_X}(\beta)$ is compatible with the $*$-pullback and with the $!$-pushforward, which coincides with the $*$-pushforward for $s_{\mathcal{U}}$ as it is a closed morphism. For what we have said above,
\begin{equation}
    j_{0*}(\Ql{,\mathcal{U}_X}\rightarrow s_{\mathcal{U}*}s_{\mathcal{U}}^*\Ql{,\mathcal{U}_X})\otimes_{\Ql{,\mathcal{U}_X}}\Ql{,\mathcal{U}_X}(\beta)\simeq j_{0*}\Ql{,\mathcal{U}_X}(\beta)\rightarrow j_{0*}s_{\mathcal{U}*}s_{\mathcal{U}}^*\Ql{,\mathcal{U}_X}(\beta)
\end{equation}
and $-\otimes_{\Ql{}}\Ql{}(\beta)$ is compatible with $*$-pullbacks. The last assertion follows as $-\otimes_{\Ql{}}\Ql{}(\beta)$ is exact.
\end{proof}
\begin{cor}
Let $(X,s_X)\in \tlgm{S}$ and assume that $X$ is a regular scheme. The following equivalence holds in $\md_{\Ql{,X}(\beta)}(\shvl(X))$
\begin{equation}
    i_{0}^*j_{0*}\Ql{,\mathcal{U}_X}(\beta)\simeq cofiber \bigl(c_1(\Lc_X)\otimes_{\Ql{,X}}\Ql{,X}(\beta):\Ql{,X}(\beta)\rightarrow \Ql{,X}(\beta) \bigr).
\end{equation}
\end{cor}
\begin{proof}
This follows immediately from the previous Proposition, Proposition \ref{cofiber first chern class} and from algebraic Bott periodicity $\Ql{,X}(\beta)(-1)[-2]\simeq \Ql{,X}(\beta)$.
\end{proof}
\begin{rmk}
Notice that $i_0^*j_{0*}$ and $i^*j_*$ are lax monoidal functors. In particular, $s_0^*i_0^*j_{0*}\Ql{,\mathcal{U}_X}(\beta)$ and $i^*j_*\Ql{,X_{\mathcal{U}}}(\beta)$ have commutative algebra structures. The map $sp_{(X,s_X)}^{\textup{mi}}\otimes_{\Ql{,X_0}}\Ql{,X_0}(\beta)$ is a map of commutative algebras as $s_X^*j_{0*}\Ql{,X_{\mathcal{U}_X}}(\beta)\rightarrow j_*s_{\mathcal{U}}^*\Ql{,X_{\mathcal{U}_X}}(\beta)$ is so. In particular, the $\ell$-adic sheaf $\Phi^{\textup{mi}}_{(X,s_X)}(\Ql{}(\beta))$ lives in $\md_{i^*s_X^*j_{0*}\Ql{,\mathcal{U}_X}}(\beta)(\shvl(X_0))$.
\end{rmk}

\subsection{The main theorem}
\begin{prop}\label{compatibility chern characters}
Let $X$ be a regular scheme and let $\Lc \in Pic(X)$. Then
\begin{equation}
    cofib(1-\Rl_X(m_{\Lc}):\Rl_X(\bu_X)\rightarrow \Rl_X(\bu_X))
\end{equation}
\begin{equation*}
\simeq cofib(c_1(\Lc)\otimes_{\Ql{,X}}\Ql{,X}(\beta):\Ql{,X}(\beta)\rightarrow \Ql{,X}(\beta)),
\end{equation*}
where $c_1(\Lc)$ is the first Chern class of $\Lc$.

Moreover, 
\begin{equation}
    cofib(c_1(\Lc)\otimes_{\Ql{,X}}\Ql{,X}(\beta))\simeq  cofib(c_1(\Lc^{\vee})\otimes_{\Ql{,X}}\Ql{,X}(\beta)).
\end{equation}
\end{prop}
\begin{proof}
We start by noticing that the composition
\begin{equation}
    \textup{HK}(X)\otimes_{\Z} \Q \simeq H^{2*,*}_{\mathscr{M}}(X,\Z)\otimes_{\Z}\Q \simeq H^{2*,*}_{\mathscr{M}}(X,\Q),
\end{equation}
where $H^{2*,*}_{\mathscr{M}}$ denotes motivic cohomology, coincides with the Chern character (see \cite[\S 11.3.6]{cd19}) which is compatible with the $\ell$-adic Chern character (see \cite[\S 3.6 - \S 3.7]{brtv} and \cite{cd19}). Now, according to \cite[Proposition 12.2.9]{cd19} and \cite[Proposition 3.8]{de08}, $c_1(\Lc)$ is nilpotent and therefore the Chern character of $\Lc$ can be written as $1+c_1(\Lc)+\frac{1}{2}c_1(\Lc)^2+\dots+\frac{1}{m!}c_1(\Lc)^{m}$ for some $m\geq 1$. Then 
\begin{equation*}
    1-\Rl_X(m_{\Lc})=-c_1(\Lc)(1+\frac{1}{2}c_1(\Lc)+\dots+\frac{1}{m!}c_1(\Lc)^{m-1}).
\end{equation*}
As $c_1(\Lc)$ is nilpotent, so is $\frac{1}{2}c_1(\Lc)+\dots+\frac{1}{m!}c_1(\Lc)^{m-1}$, whence $(1+\frac{1}{2}c_1(\Lc)+\dots+\frac{1}{m!}c_1(\Lc)^{m-1})$ is invertible. This shows that
\begin{equation*}
    cofiber(1-\Rl_X(m_{\Lc}))\simeq cofiber(-c_1(\Lc)\otimes_{\Ql{,X}}\Ql{,X}(\beta))
\end{equation*}
\begin{equation*}
\simeq cofiber(c_1(\Lc)\otimes_{\Ql{,X}}\Ql{,X}(\beta)).
\end{equation*}
The last equivalence follows from the fact that, by \cite[Proposition 12.2.9 and Remark 13.2.2]{cd19}, we have 
\begin{equation}
    0=c_1(\Lc \otimes \Lc^{\vee})\simeq F(c_1(\Lc),c_1(\Lc^{\vee}))=c_1(\Lc)+c_1(\Lc^{\vee})+\beta^{-1}\cdot c_1(\Lc)\cdot c_1(\Lc^{\vee}).
\end{equation}
In particular, 
\begin{equation}
    c_1(\Lc^{\vee})\simeq -c_1(\Lc)(1+\beta^{-1}\cdot c_1(\Lc^{\vee}))
\end{equation}
and, as we have remarked above, $-(1+\beta^{-1}\cdot c_1(\Lc^{\vee}))$ is invertible as $c_1(\Lc^{\vee})$ is nilpotent.
\end{proof}
We are finally ready to state our main theorem:
\begin{thm}\label{main theorem}
Let $(X,s_X)$ be a twisted LG model over $(S,\Lc_S)$. Assume that $X$ is regular and that $0,s_X:X\rightarrow \V(\Lc_X)$ are Tor-independent. The following equivalence holds in $\md_{i^*s_X^*j_{0*}\Ql{,\mathcal{U}_X}(\beta)}(\shvl(X_0))$:
\begin{equation}
    i^*\Rlv_{X}(\sing{X,s_X})\simeq \Phi^{\textup{mi}}_{(X,s_X)}(\Ql{}(\beta))[-1],
\end{equation}
where the $i^*s_X^*j_{0*}\Ql{,\mathcal{U}_X}(\beta)$-module structure on the l.h.s. is the one induced by the equivalence of commutative algebras
\begin{equation}
   i^*\Rlv_X(\sing{X,0})\simeq   i^*s_X^*j_{0*}\Ql{,\mathcal{U}_X}(\beta).
\end{equation}
\end{thm}
\begin{proof}
We start by noticing that $i=s_0:X_0\rightarrow X$. Indeed, $i_0 \circ s_0=s_X\circ i$ and both $i_0$ and $s_X$ are sections of the canonical morphism $\V_X=\V(\Lc_X)\rightarrow X$. In particular
\begin{equation}
     i^*\Rlv_X(\sing{X,0})\simeq  s_0^*\Rlv_X(\sing{X,0})\underbracket{\simeq}_{\textup{Proposition }\ref{motivic realization sing(S,0) 2}}
      i^*s_X^*j_{0*}\Ql{,\mathcal{U}_X}(\beta).
\end{equation}
Since $X$ is supposed to be regular, $\sing{X,s_X}\simeq \sing{X_0}$. By Propositions \ref{fundamental fiber cofiber sing} and \ref{compatibility chern characters} we have a fiber-cofiber sequence
\begin{equation}
\begindc{\commdiag}[18]
    \obj(0,15)[1]{$i^*\Rl_X(\sing{X_0})$}
    \obj(0,0)[2]{$cofiber(c_1(\Lc_{X_0})\otimes_{\Ql{,X_0}}\Ql{,X_0}(\beta):\Ql{,X_0}(\beta)\rightarrow \Ql{,X_0}(\beta))$}
    \obj(0,-15)[3]{$ i^*j_*\Ql{,X_{\mathcal{U}}}(\beta).$}
    \mor{1}{2}{$$}
    \mor{2}{3}{$f$}
\enddc
\end{equation}
in $\md_{i^*s_X^*j_{0*}\Ql{,\mathcal{U}_X}(\beta)}(\shvl(X_0))$.
On the other hand, by Proposition \ref{compatibility Phimi with beta}, we have another fiber-cofiber sequence 
\begin{equation}
\begindc{\commdiag}[18]
    \obj(0,15)[1]{$\Phi^{\textup{mi}}_{(X,s_X)}(\Ql{}(\beta))[-1]$}
    \obj(0,0)[2]{$cofiber(c_1(\Lc_{X_0})\otimes_{\Ql{,X_0}}\Ql{,X_0}(\beta):\Ql{,X_0}(\beta)\rightarrow \Ql{,X_0}(\beta))$}
    \obj(0,-15)[3]{$i^*j_*\Ql{,X_{\mathcal{U}}}(\beta).$}
    \mor{1}{2}{$$}
    \mor{2}{3}{$g$}
\enddc
\end{equation}
Both $f$ and $g$ in these two fiber-cofiber sequences are defined by the universal property of the object in the middle. Therefore, if we consider the diagram
\begin{equation}
    \begindc{\commdiag}[18]
    \obj(-80,20)[1]{$\Ql{,X_0}(\beta)$}
    \obj(-5,20)[2]{$\Ql{,X_0}(\beta)$}
    \obj(55,20)[3]{$cofiber(c_1(\Lc_{X_0}^{\vee})\otimes_{\Ql{,X_0}}\Ql{,X_0}(\beta))$}
    \obj(-5,-20)[4]{$i^*j_*\Ql{,X_{\mathcal{U}}}(\beta)$,}
    \mor{1}{2}{$c_1(\Lc_{X_0}^{\vee})\otimes_{\Ql{,X_0}}\Ql{,X_0}(\beta)$}
    \mor{2}{3}{$$}
    \mor{1}{4}{$\sim 0$}
    \mor(-3,17)(-3,-17){$\tilde{f}$}[\atright,\solidarrow]
    \mor(3,17)(3,-17){$\tilde{g}$}
    \mor(76,17)(4,-16){$f$}[\atright,\solidarrow]
    \mor(84,17)(12,-16){$g$}
    \enddc
\end{equation}

where the middle vertical arrows induce the morphism $f$ and $g$ respectively, it suffices to show $\tilde{f}\sim \tilde{g}$. By definition, $\tilde{f}$ is induced by the counit $\Ql{,X}(\beta)\rightarrow j_*j^*\Ql{,X}(\beta)$, while $\tilde{g}$ is induced by the base change morphism $b.c.:s_X^*j_{0*}\Ql{,\mathcal{U}_X}(\beta)\rightarrow j_*s_{\mathcal{U}}^*\Ql{,\mathcal{U}_X}(\beta)$:
\begin{equation}
    \tilde{g}:i^*s_X^*\Ql{,\V_X}(\beta)\rightarrow i^*s_X^*j_{0*}j_0^*\Ql{,\V_X}(\beta)\rightarrow i^*j_*s_{\mathcal{U}}^*j_{0}^*\Ql{,\V_X}(\beta).
\end{equation}
It also suffices to show that the two arrows are homotopic before applying $i^*$. Consider the diagram
\begin{equation}
    \begindc{\commdiag}[18]
    \obj(-60,10)[1]{$s_X^*\Ql{,\V_X}(\beta)$}
    \obj(0,10)[2]{$s_X^*j_{0*}j_0^*\Ql{,\V_X}(\beta)$}
    \obj(70,10)[3]{$j_*s_{\mathcal{U}}^*j_0^*\Ql{,\V_X}(\beta)\simeq j_*j^*s_X^*\Ql{,\V_X}(\beta)$}
    \obj(0,-10)[4]{$s_X^*s_{X*}j_*s_{\mathcal{U}}^*j_0^*\Ql{,\V_X}(\beta)$}
    \obj(0,-26)[5]{$s_X^*s_{X*}j_*j^*s_X^*\Ql{,\V_X}(\beta),$}
    \obj(0,-18)[]{$\simeq$}
    \mor{1}{2}{$$}
    \mor{2}{3}{$b.c.$}
    \mor{2}{4}{$$}
    \mor(25,-18)(80,7){$s_X^*s_{X*}\rightarrow 1$}[\atright,\solidarrow]
    \mor(-60,7)(-25,-18){$$}
    \enddc
\end{equation}

If we analyse the commutative triangle on the left, we see that the objlique arrow corresponds to the unit of the adjunction $((s_X\circ j)^*,(s_X\circ j)_*)$. Indeed, by definition, the vertical arrow is 
\begin{equation}
    s_X^*j_{0*}(j_0^*\Ql{,\V_X}(\beta)\xrightarrow{\text{unit }(s_{\mathcal{U}}^*,s_{\mathcal{U}*})} s_{\mathcal{U}*}s_{\mathcal{U}}^*j_0^*\Ql{,\V_X}(\beta))
\end{equation}
and the horizontal arrow is the one induced by the unit of $(j_0^*,j_{0*})$. The composition is exactly the unit of the adjunction $((j_0\circ s_{\mathcal{U}})^*,(j_0\circ s_{\mathcal{U}})_*)$. Hence, the claim is proved as $j_0\circ s_{\mathcal{U}}\simeq s_X\circ j$. Now notice that $s_{X*}$ is conservative ($s_X$ is a closed morphism), i.e. the counit $s_X^*s_{X*}\rightarrow 1$ is a natural equivalence. If we compose the oblique arrow with it, we obtain the unit of $(j^*,j_*)$ evaluated in $s_X^*\Ql{,\V_X}(\beta)$. The statement about the module structures is clear.
\end{proof}
\begin{rmk}
Notice that our main theorem provides a generalization of the formula proved in \cite{brtv}: assume that we are given a proper flat morphism $p:X\rightarrow S$ from a regular scheme to an excellent strictly henselian trait\footnote{We shall recycle the notation that we have introduced in $\S$ \ref{the formalism of vanishing cycles}}. Let $s_X:X\rightarrow \aff{1}{X}$ be the pullback of the section $S\rightarrow \aff{1}{S}$ given by the uniformizer $\pi$. Then we can consider the diagram
\begin{equation}
    \begindc{\commdiag}[18]
    \obj(-40,20)[1]{$\sigma$}
    \obj(0,20)[2]{$S$}
    \obj(40,20)[3]{$\eta$}
    \obj(-40,0)[4]{$X_0$}
    \obj(0,0)[5]{$X$}
    \obj(40,0)[6]{$X_{\eta}$}
    \obj(-40,-20)[7]{$X$}
    \obj(0,-20)[8]{$\aff{1}{X}$}
    \obj(40,-20)[9]{$\gm{X}$}
    \mor{1}{2}{$i_{\sigma}$}
    \mor{3}{2}{$j_{\eta}$}[\atright,\solidarrow]
    \mor{4}{1}{$p_{\sigma}$}
    \mor{5}{2}{$p$}
    \mor{6}{3}{$p_{\eta}$}[\atright,\solidarrow]
    \mor{4}{5}{$i$}
    \mor{6}{5}{$j$}[\atright,\solidarrow]
    \mor{4}{7}{$s_0$}[\atright,\solidarrow]
    \mor{5}{8}{$s_X$}
    \mor{6}{9}{$s_{\gm{X}}$}
    \mor{7}{8}{$i_0$}[\atright,\solidarrow]
    \mor{9}{8}{$j_0$}
    \enddc
\end{equation}
where all squares are cartesian. The theorem we have just proved tells us that
\begin{equation*}
    i^*\Rlv_X(\sing{X,s_X})\simeq \Phi_{(X,s_X)}^{\textup{mi}}(\Ql{}(\beta))[-1].
\end{equation*}
By the regularity assumption on $X$, $\sing{X,s_X}\simeq \sing{X_0}$. If we apply $p_{\sigma*}$ to the formula above we find:
\begin{equation}
    p_{\sigma*}i^*\Rlv_X(\sing{X_0})\simeq i_{\sigma}^*\Rlv_S(\sing{X_0})\simeq
\end{equation}
\begin{equation*}
    p_{\sigma*}\Phi_{(X,s_X)}^{\textup{mi}}(\Ql{}(\beta))[-1]\underbracket{\simeq}_{\text{Lemma }\ref{comparison monodromy invariant inertia invariant vanishing cycles}} \bigl( p_{\sigma*}\Phi_p(\Ql{,X}(\beta))[-1] \bigr)^{\textup{h}I},
\end{equation*}
which is exactly the content of \cite[Theorem 4.39]{brtv}. Also notice that in this case 
\begin{equation}
i^*s_X^*j_{0*}\Ql{,\gm{X}}(\beta)\simeq \Ql{,X_{0}}(\beta)\oplus \Ql{,X_{0}}(\beta)[1],
\end{equation}
as the first Chern class of the trivial line bundle is zero. This recovers the equivalence 
\begin{equation}
    (\Ql{,\sigma}(\beta))^{\textup{h}I}\simeq \Ql{,\sigma}(\beta)\oplus \Ql{,\sigma}(\beta)[1]
\end{equation}
proved in \textit{loc.cit.}
\end{rmk}
\section{The \texorpdfstring{$\ell$}{l}-adic realization of the dg category of singularities of a twisted LG model of rank \texorpdfstring{$r$}{r}}\label{the l-adic realization of the dg category of singularities of a twisted LG model of rank r}

In this section we will explain how, by means of a theorem due to D.~Orlov and J.~Burke - M.~Walker, the main theorem we proved in the previous section allows us to compute the $\ell$-adic realization of the dg category of singularities of the zero locus of a global section of any vector bundle on a regular scheme. This is the reason that led us to consider twisted LG model, as defined in $\S 3.1$, as the author was initially interested in computing the $\ell$-adic realization of the zero locus of a multifunction $X\rightarrow \aff{n}{S}$ (with $X$ regular).
\subsection{Reduction of codimension}
In this section we will provide an $\oo$-functorial lax monoidal enhancement of the "reduction of codimension" equivalence proved by D.~Orlov (\cite[Theorem 2.1]{orl06}) and by J.~Burke - M.~Walker (\cite[Theorem A.4]{bw15}).
\begin{context}
Let $S$ be a regular noetherian scheme of finite Krull dimension.
\end{context}
\begin{defn}
Notice that all we said in section \S \ref{the category of twisted LG models} can be generalised \ital{mutatis mutandis} to the situation where line bundles $\Lc_S$ are replaced by vector bundles of a fixed rank $r$ $\Ec_S$. In particular, for a fixed vector bundle $\Ec_S$ on $S$, we can define a symmetric monoidal (ordinary) category $\tlgmr{S}^{\boxplus}$, analogous to the one we defined in section \S \ref{the category of twisted LG models}.
\end{defn}
\begin{notation}
Let $\Ec_S$ be a vector bundle of rank $r$ over $S$. We will denote $\mathbb{P}(\Ec_S)=Proj_S(Sym_{\Oc_S}(\Ec_S^{\vee}))$ and  $\pi_S:\mathbb{P}(\Ec_S)\rightarrow S$ the associate projective bundle and projection. Moreover, we will denote by $\Oc(1)$ the twisting sheaf on $\mathbb{P}(\Ec_S)$.
\end{notation}

\begin{construction}
We will construct an (ordinary) symmetric monoidal functor
\begin{equation}\label{from twisted lg models of rank r to twisted lg models}
\Xi^{\boxplus}:\tlgmr{S}^{\boxplus}\rightarrow \tlgmproj{\Ec_S^{\vee}}^{\boxplus}
\end{equation}
following the lead of \cite{bw15} and \cite{orl06}.
Let $(X,s)\in \tlgmr{S}$. Consider $W_s \in \Gamma(\mathbb{P}(\Ec_X^{\vee}),\Oc(1))$ defined as the morphism
\begin{equation}
    W_s:\Oc_{\mathbb{P}(\Ec_X^{\vee})}\simeq \widetilde{Sym_{\Oc_X}(\Ec_X)}\rightarrow \widetilde{Sym_{\Oc_X}(\Ec_X)(1)}\simeq \Oc(1)
\end{equation}
induced by the morphism of modules
\begin{equation}
    Sym_{\Oc_X}(\Ec_X)\rightarrow Sym_{\Oc_X}(\Ec_X)(1)
\end{equation}
induced by $s:\Oc_X\rightarrow \Ec_X^{\vee}$. Here $\widetilde{(-)}$ is the functor that associates a quasi-coherent module on $\mathbb{P}(\Ec_X^{\vee})$ to a graded $Sym_{\Oc_X}(\Ec_X)$-module.
The assignment 
\begin{equation*}
    (X,s)\mapsto (\mathbb{P}(\Ec_X^{\vee}),W_{s}),
\end{equation*}
together with the obvious law for morphisms defines a functor
\begin{equation}
\Xi:\tlgmr{S}\rightarrow \tlgmproj{\Ec_S^{\vee}}.
\end{equation}
It is immediate to observe that $(S,\underline{0})\mapsto (\mathbb{P}(\Ec_S^{\vee}),0)$, i.e. the functor is compatible with the units of the two symmetric monoidal structures. It remains to show that
\begin{equation*}
    \Xi((X,s)\boxplus (Y,t))\simeq \Xi((X,s))\boxplus \Xi((Y,t)).
\end{equation*}
On the left hand side we have
\begin{equation*}
    (\mathbb{P}(\Ec_{X\times_S Y}^{\vee}),W_{s\boxplus t})),
\end{equation*}
while on the right hand side we have
\begin{equation*}
    (\mathbb{P}(\Ec_X^{\vee})\times_{\mathbb{P}(\Ec_S^{\vee})}\mathbb{P}(\Ec_Y^{\vee}),W_{s}\boxplus W_{t}).
\end{equation*}

Since $\mathbb{P}(\Ec_X^{\vee})\times_{\mathbb{P}(\Ec_S^{\vee})}\mathbb{P}(\Ec_Y^{\vee})\simeq (X\times_S \mathbb{P}(\Ec_S^{\vee}))\times_{\mathbb{P}(\Ec_S^{\vee})}(Y\times_S \mathbb{P}(\Ec_S^{\vee}))\simeq (X\times_S Y)\times_S \mathbb{P}(\Ec_S^{\vee})\simeq \mathbb{P}(\Ec_{X\times_S Y}^{\vee})$, it suffices to show that $W_{s\boxplus t}=W_{s}\boxplus W_{t}$. It is enough to do it Zariski-locally ($\Oc(1)$ is a sheaf). Locally, $\mathbb{P}(\Ec_{X\times_S Y}^{\vee})$ is isomorphic to a projective space. We may consider the covering of $\mathbb{P}(\Ec_{X\times_S Y}^{\vee})$ consisting of open affine subschemes $Spec\bigl (B\otimes_A C[\frac{T_1}{T_j},\dots,\frac{T_n}{T_j}] \bigr )$, where $Spec(A)$ (resp. $Spec(B)$, resp. $Spec(C)$) is an open affine subscheme of $S$ (resp. $X$, resp. $Y$). The restriction to this open subset of $W_{s\boxplus t}$ is of the form
\begin{equation*}
    (f_1\otimes1+1\otimes g_1)\cdot \frac{T_1}{T_j}+\dots +(f_n\otimes 1+1\otimes g_n)\cdot \frac{T_n}{T_j},
\end{equation*}
while that of $W_{s}\boxplus W_{t}$ is
\begin{equation*}
    \bigl(f_1\cdot \frac{T_1}{T_j}\otimes 1 + 1\otimes g_1\cdot \frac{T_1}{T_j}\bigr)+\dots+\bigl(f_n\cdot \frac{T_n}{T_j}\otimes1+1\otimes g_n\cdot \frac{T_n}{T_j}\bigr).
\end{equation*}
The claim follows and thus we obtain the desired symmetric monoidal functor.
\end{construction}

If we compose the symmetric monoidal functor (\ref{from twisted lg models of rank r to twisted lg models}) with the lax monoidal $\oo$-functor defined in Section \ref{The dg category of singularities of a twisted LG model} we get
\begin{equation}\label{sing projectivization}
    \sing{\mathbb{P}(\Ec_{\bullet}^{\vee}),W_{\bullet}}^{\otimes} : \tlgmr{S}^{\textup{op},\boxplus}
    \rightarrow \md_{\sing{\mathbb{P}(\Ec_S^{\vee}),0}}\bigl(\dgcatm\bigr)^{\otimes}
\end{equation}
which, at the level of objects, corresponds to the assignment 
\begin{equation*}
    (X,s)\mapsto \sing{\mathbb{P}(\Ec_X^{\vee}),W_s}.
\end{equation*}

We define the dg category of singularities of a twisted LG model $(X,s)$ of rank $r$ over $(S,\Ec_S)$ in the following way: consider the derived zero locus of $s$, defined as the homotopy pullback of $s$ along the zero section 
\begin{equation}
    \begindc{\commdiag}[12]
    \obj(-20,20)[1]{$X_0$}
    \obj(20,20)[2]{$X$}
    \obj(-20,-20)[3]{$S$}
    \obj(20,-20)[4]{$\V(\Ec_S).$}
    \mor{1}{2}{$i$}
    \mor{1}{3}{$$}
    \mor{2}{4}{$s$}
    \mor{3}{4}{$0$}
    \enddc
\end{equation}

As $0:S\rightarrow \V(\Ec_S)$ is a closed lci morphism (locally, is of the form $U\rightarrow \aff{r}{S}$), so is $i:X_0\rightarrow X$. Thus the pushforward induces a dg functor
\begin{equation}
    i_*:\sing{X_0}\rightarrow \sing{X}.
\end{equation}
\begin{defn}
We define
\begin{equation}
    \sing{X,s}:=Ker\bigl(i_*:\sing{X_0}\rightarrow \sing{X} \bigr).
\end{equation}
\end{defn}
Similarly to what we did in Section \ref{The dg category of singularities of a twisted LG model}, we can define a lax monoidal $\oo$-functor 
\begin{equation}\label{lax monoidal sing twisted lg models rank r}
    \sing{\bullet,\bullet}^{\otimes}:\tlgmr{S}^{\textup{op},\boxplus}\rightarrow \md_{\sing{S,0}}\bigl( \dgcatm \bigr)^{\otimes}
\end{equation}
At the level of objects, it is defined by $(X,s)\mapsto \sing{X,s}$.
\begin{rmk}
Let $Spec(B)$ be an affine scheme and let $E_B$ be a projective B-module of constant rank $r$ and $s\in E_B$. The derived zero locus of $s$ is the spectrum of $B\otimes^{\mathbb{L}}_{Sym_B(E_B^{\vee})}B$, where the two augmentations $Sym_B(E_B^{\vee})\rightarrow B$ are determined by $s$ and $0$. The dg algebra associated to the simplicial ring $B\otimes^{\mathbb{L}}_{Sym_B(E_B^{\vee})}B$ via the Dold-Kan correspondence is the Koszul dg algebra associated to $(B,E_B^{\vee},s)$ 
\begin{equation}
K(B,E_B^{\vee},s):= 0\rightarrow \bigwedge^{r}E_B^{\vee}\rightarrow \bigwedge^{r-1}E_B^{\vee}\rightarrow \dots \bigwedge^{2}E_B^{\vee}\rightarrow E_B^{\vee}\xrightarrow{s} B\rightarrow 0 
\end{equation}
concentrated in degrees $[-r,0]$. This is locally true (see for example \cite[Remark 1.22]{pi19} or \cite{kr19}) and the global statement follows from the existence of a morphism of dg algebras
\begin{equation}
    K(B,E_B^{\vee},s)\rightarrow N(B\otimes^{\mathbb{L}}_{Sym_B(E_B^{\vee})}B),
\end{equation}
where $N$ is the normalized Moore complex functor.
\end{rmk}
\begin{construction}
Similarly to what we said in Construction \ref{cohs section line bundle}, there is an equivalence between $\qcoh{B\otimes^{\mathbb{L}}_{Sym_B(E_B^{\vee})}B}$ and the category of cofibrant $K(B,E_B^{\vee},s)$-dg modules, denoted by $\widehat{K(B,E_B^{\vee},s)}$. Under this equivalence, $\cohb{B\otimes^{\mathbb{L}}_{Sym_B(E_B^{\vee})}B}$ corresponds to the full subcategory of $\widehat{K(B,E_B^{\vee},s)}$ spanned by $K(B,E_B^{\vee},s)$-dg modules that are cohomologically bounded and whose cohomology is coherent over $coker(E_B^{\vee}\xrightarrow{s}B)$. 

Similarly, $\perf{B\otimes^{\mathbb{L}}_{Sym_B(E_B^{\vee})}B}$ corresponds to the full subcategory of $\widehat{K(B,E_B^{\vee},s)}$ spanned by homotopically finitely presented dg modules. We consider the full subcategory $\cohs{B,E_B^{\vee},s}$ spanned by $K(B,E_B^{\vee},s)$-dg modules that are cohomologically bounded, whose cohomology is coherent over $coker(E_B^{\vee}\xrightarrow{s}B)$ and whose underlying complex of $B$-modules is strictly perfect.
\end{construction}
\begin{lmm}\textup{\cite[Lemma 2.33]{brtv}}\\
Let $\cohs{B,E_B^{\vee},s}^{\textup{acy}}$ be the full sub-category of $\cohs{B,E_B^{\vee},s}$ spanned by acyclic dg modules. Then the cofibrant replacement induces an equivalence of dg categories
\begin{equation}
    \cohs{B,E_B^{\vee},s}[\textup{q.iso}^{-1}]\simeq \cohs{B,E_B^{\vee},s}/\cohs{B,E_B^{\vee},s}^{\textup{acy}}
    \end{equation}
\begin{equation*}
\simeq \cohbp{X_0}{X}.
\end{equation*}
If we label $\textup{Perf}^s(B,E_B^{\vee},s)$ the full subcategory of $\cohs{B,E_B^{\vee},s}$ spanned by perfect $K(B,E_B^{\vee},s)$-dg modules, there are equivalances of dg categories
\begin{equation}
   \cohs{B,E_B^{\vee},s}/\textup{Perf}^s(B,E_B^{\vee},s)\simeq \cohbp{X_0}{X}/\perf{X_0}\simeq \sing{X,s}.
   \end{equation}
\end{lmm}
\begin{proof}
The proof of \cite[Lemma 2.33]{brtv} applies, \ital{mutatis mutandis}.
\end{proof}
In the same way as in Construction \ref{functorial properties cohs line bundle} and in the following discussion, we can construct a lax monoidal $\oo$-functor
\begin{equation}
    \cohbp{\bullet}{\bullet}^{\otimes}:\tlgmr{S}^{op,\boxplus}\rightarrow \dgcatmo
\end{equation}
Moreover, using the same arguments as in the discussion following Remark 1.27 in \cite{pi19}, we get the lax monoidal $\oo$-functor (\ref{lax monoidal sing twisted lg models rank r}).

At this point, the reader might complain that there are too many $\textup{\textbf{Sing}}$'s involved, and this may cause confusion. As a partial justification to the choice of notation we made, let us show that these different $\oo$-functors are closely related.
\begin{defn}
Let $(X,s)\in \tlgmr{S}$. We define the projective bundle over $X_0$ associated to $\Ec_{X_0}=i^*\Ec_X$ as the derived pullback
\begin{equation}
\begindc{\commdiag}[18]
    \obj(-20,10)[1]{$\mathbb{P}(\Ec_{X_0}^{\vee})$}
    \obj(20,10)[2]{$\mathbb{P}(\Ec_X^{\vee})$}
    \obj(-20,-10)[3]{$X_0$}
    \obj(20,-10)[4]{$X.$}
    \mor{1}{2}{$$}
    \mor{1}{3}{$p_0$}[\atright,\solidarrow]
    \mor{2}{4}{$p$}
    \mor{3}{4}{$i$}
    \enddc
\end{equation}
\end{defn}

Recall that, given a global section of a line bundle on a scheme, we define its (derived) zero locus as in Definition \ref{derived zero locus global section line bundle}. In particular, for $(X,s)\in \tlgmr{S}$ we have a global section $W_s$ of the line bundle $\Oc(1)$ on $\mathbb{P}(\Ec_X^{\vee})$ and thus we have a (derived) pullback square
\begin{equation}
    \begindc{\commdiag}[18]
    \obj(-20,10)[1]{$V(W_{s})$}
    \obj(20,10)[2]{$\mathbb{P}(\Ec_X^{\vee})$}
    \obj(-20,-10)[3]{$\mathbb{P}(\Ec_X^{\vee})$}
    \obj(20,-10)[4]{$\V(\Oc(1)).$}
    \mor{1}{2}{$k$}
    \mor{1}{3}{$k$}
    \mor{2}{4}{$W_s$}
    \mor{3}{4}{$0$}
    \enddc
\end{equation}

Let $(X,s)\in \tlgmr{S}$ and consider the following diagram, where both squares are homotopy cartesian:
\begin{equation}\label{diagram for defining upsilon}
    \begindc{\commdiag}[18]
    \obj(0,20)[1]{$\mathbb{P}(\Ec_X^{\vee})$}
    \obj(30,20)[2]{$\V(\Oc(1))$}
    \obj(-30,0)[3]{$\mathbb{P}(\Ec_{X_0}^{\vee})$}
    \obj(0,0)[4]{$V(W_{s})$}
    \obj(30,0)[5]{$\mathbb{P}(\Ec_X^{\vee})$}
    \obj(-30,-20)[6]{$X_0$}
    \obj(30,-20)[7]{$X.$}
    \mor{1}{2}{$W_{s}$}
    \mor{4}{1}{$k$}
    \mor{5}{2}{$0$}[\atright,\solidarrow]
    \mor{4}{5}{$k$}
    \mor{3}{4}{$j$}
    \mor{3}{6}{$p_0$}
    \mor{5}{7}{$p$}
    \mor{6}{7}{$i$}
    \enddc
\end{equation}
By \cite[Chapter 4, Lemma 3.1.3, Lemma 5.1.4]{gr17}, $j_*p_0^*$ preserves complexes with coherent bounded cohomology. The derived proper base change equivalence $k_*j_*p_0^*\simeq p^*i_*$ implies that we have a dg functor
\begin{equation}
    \Upsilon_{(X,s)}=j_*p_0^*:\cohbp{X_0}{X}\rightarrow \cohbp{V(W_s)}{\mathbb{P}(\Ec_X^{\vee})}.
\end{equation}
and it is functorial in $(X,s)$. We will need to enhance the assignment $(X,s)\mapsto \Upsilon_{(X,s)}$ with a lax monoidal structure. We will use strict models.

Assume that $S=Spec(A)$ is affine and let $E$ be a projective $A$ module of rank $r$. Since $E$ is an $A$-module of finite type, there exists a surjection
\begin{equation}
  A^{n}\rightarrow E.
\end{equation}
Denote by $t_i$ ($i=1,\dots, n$) the images in $E$ of the elements $(1,0,\dots,0)$, \dots, $(0,\dots,0,1)\in A^n$.

We will define a lax monoidal $\oo$-natural transformation
\begin{equation}
  \Upsilon^{\otimes}: \cohbp{V(\bullet)}{\bullet}^{\otimes}\rightarrow \cohbp{V(W_{\bullet})}{\mathbb{P}(E_{\bullet}^{\vee})}^{\otimes},
\end{equation}
where
\begin{equation}
\cohbp{V(\bullet)}{\bullet}^{\otimes}, \cohbp{V(W_{\bullet})}{\mathbb{P}(E_{\bullet}^{\vee})}^{\otimes}: \textup{LG}_{(S,E)}^{\boxplus,\textup{aff,op}}\rightarrow \dgcatmo.
\end{equation}

For every $(Spec(B),s)\in \textup{LG}_{(S,E)}^{\boxplus,\textup{aff,op}}$,
\begin{equation*}
    \mathbb{P}(E_B^{\vee})=Proj(Sym_B(E_B))
\end{equation*} 
has a Zariski affine covering $\{ D_+(t_i) \}_{i=1}^n$, where
\begin{equation}
   D_+(t_i)=Spec(Sym_B(E_B^{\vee})_{(t_i)}).
\end{equation}
Here $Sym_B(E_B^{\vee})_{(t_i)}$ denotes the ring of degree $0$ elements of the graded ring $Sym_B(E_B^{\vee})[t_i^{-1}]$. Since
\begin{equation}
\cohbp{V(W_{s})}{\mathbb{P}(E_{B}^{\vee})}^{\otimes}\simeq \varprojlim \cohbp{V(W_{s|\mathcal{U}^{\bullet}})}{\mathcal{U}^{\bullet}}^{\otimes},
\end{equation}
where $\mathcal{U}^{\bullet}$ is the Cech hypercover of $\{ D_+(t_i) \}_{i=1}^n$, it will suffice to define an homotopy coherent diagram of lax monoidal $\oo$-natural transformations
\begin{equation}\label{homotopy coherent diagram lax monoidal oo natural transformations upsilon 1}
\cohbp{V(\bullet)}{\bullet}^{\otimes}\rightarrow \cohbp{V(W_{\bullet|\mathcal{U}^{\bullet}}}{\mathcal{U}^{\bullet}}^{\otimes}.
\end{equation}

Recall that $\cohs{B,E_{B}^{\vee},s_B}$ is the $A$-dg category of $K(B,E_B^{\vee},s_B)$-dg modules whose underlying $B$-dg module is strictly perfect. In other words, it is the dg category of $B$-dg modules $(M,d)$ endowed with a $B$-linear map
\begin{equation}
  h: M\rightarrow M\otimes_B E_B[-1]
\end{equation}
such that
\begin{equation}
  h^2=0, \hspace{0.5cm} [d,h]=id_M\otimes s_B:M\rightarrow M\otimes_B E_B.
\end{equation}
\begin{construction}
For every $(B,s_B) \in  \textup{LG}_{(S,E)}^{\boxplus,\textup{aff,op}}$ and every $i=1,\dots,n$, define a pseudo-functor
\begin{equation}
  \Upsilon_i:\cohs{B,E_B^{\vee},s_B}\rightarrow \cohs{Sym_B(E_B)_{(t_i)},t_i^{-1}\cdot Sym_B(E_B)_{(t_i)},W_{s_{B|D_+(t_i)}}}
\end{equation}
as follows. For $(M,d,h)\in \cohs{B,E_B^{\vee},s_B}$, let
\begin{equation}
  \Upsilon_i(M,d,h)=(M\otimes_B Sym_B(E_B),d\otimes id, \chi_{i,h}).
\end{equation}
The map $\chi_{i,h}:M\otimes_B Sym_B(E_B)\rightarrow M\otimes_B t_i\cdot Sym_B(E_B)$ is defined as the restriction to $D_+(t_i)$ of the composition
\begin{equation}
\begindc{\commdiag}[18]
  \obj(-40,10)[1]{$M\otimes_B Sym_B(E_B)$}
  \obj(40,10)[2]{$M\otimes_B E_B \otimes_B Sym_B(E_B)[-1]$}
  \mor{1}{2}{$h\otimes id$}
  \obj(0,-10)[3]{$M\otimes_B Sym_B(E_B)(1)[-1],$}
  \cmor((40,8)(40,2)(38,0)(19,0)(2,0)(0,-2)(0,-8)) \pdown(0,0){$$}
\enddc
\end{equation}

where $Sym_B(E_B)(1)$ is the $Sym_B(E_B)$ module $Sym_B(E_B)$ with the grading shifted by $1$ and the second morphism is induced by multiplication. Notice that $ Sym_B(E_B)(1)\otimes_{Sym_B(E_B)}\otimes Sym_B(E_B)_{(t_i)}=t_i\cdot Sym_B(E_B)_{(t_i)}$. 
The desired properties that 
\begin{equation}
  \chi_{i,h}^2=0,\hspace{0.5cm} [d\otimes id, \chi_{i,h}]=W_{s_{B|D_+(t_i)}}\otimes id
\end{equation}
are local (on $Spec(B)$) and if $E_B=B^r$, $s_B=(s_{B,1},\dots,s_{B,n})$, then $\Upsilon_i$ admits the explicit description
\begin{equation}
  \Upsilon_i(M,d,h=\{h_j\}_{j=1,\dots,r})=(M\otimes_B B[x_1,\dots,x_r],d\otimes id, x_1\cdot h_1+\dots+x_r\cdot h_r)
\end{equation}
and $W_{s_{B|D_+(t_i)}}=x_1\cdot s_{B,1}+\dots +x_r\cdot s_{B,r}$.

It follows immediately from the definitions of the pseudo functorial structures of $(B,s_B)\mapsto \cohs{B,E_B^{\vee},s_B}$ and $(B,s_B)\mapsto \cohs{Sym_B(E_B)_{(t_i)},t_i^{-1}\cdot Sym_B(E_B)_{(t_i)},W_{s_{B|D_+(t_i)}}}$ that $  \Upsilon_i$ is a pseudo natural transformation. 
\end{construction}

\begin{rmk}
 Let $X=Spec(B)$ and $X_0=Spec(K(B,E_B^{\vee},s_B))$. Notice that $\Upsilon_i$ models the dg functor
 \begin{equation}
    j_*p_0^*:\cohbp{X_0}{X}\rightarrow \cohbp{V(W_{s_B})\cap D_+(t_i)}{D_+(t_i)},
 \end{equation}
 where 
 \begin{equation*}
     p_0: D_+(t_i)\times_{\mathbb{P}(E_X^{\vee})}\mathbb{P}(E_{X_0}^{\vee})\rightarrow X_0
 \end{equation*}
and
\begin{equation*}
   j: D_+(t_i)\times_{\mathbb{P}(E_X^{\vee})}\mathbb{P}(E_{X_0}^{\vee})\rightarrow V(W_{s_B})\cap D_+(t_i).
\end{equation*}
\end{rmk}

\begin{construction}
 We will now endow $\Upsilon_i$ with a pseudo lax monoidal structure. If $(B,s_B)$ and $(C,s_C)$ are affine twisted LG models of rank $r$ over $(A,E)$, then
 \begin{equation}
   \mu: \cohs{B,E_B^{\vee},s_B}\otimes \cohs{C,E_C^{\vee},sC}\rightarrow \cohs{B\otimes_AC,E_{B\otimes_AC}^{\vee},s_{B}\boxplus s_C}
 \end{equation}
 is defined on objects by
 \begin{equation}
   ((M,d_M,h_M), (N,d_N,h_N))\mapsto (M\otimes_AN, d_{M\otimes_AN}, h_M\boxplus h_N),
 \end{equation}
where $h_M\boxplus h_N$ is the map
\begin{equation}
\begindc{\commdiag}[18]
  \obj(-40,10)[1]{$M\otimes_AN$}
  \obj(40,10)[2]{$(M\otimes_AN) \otimes_{B\otimes_AC} E_{B\otimes_AC}^{\oplus 2}[-1]$}
  \obj(0,-10)[3]{$(M\otimes_AN) \otimes_{B\otimes_AC} E_{B\otimes_AC}[-1].$}
  \mor{1}{2}{$\begin{bmatrix} h_M\otimes C \\ B\otimes h_N \end{bmatrix}$}
  \cmor((40,8)(40,2)(38,0)(19,0)(2,0)(0,-2)(0,-8)) \pdown(50,0){$\begin{bmatrix} 1 & 1 \end{bmatrix}$}
\enddc
\end{equation}
It is defined on morphisms in the obvious way.

The dg functor
\begin{equation*}
  \cohs{Sym_B(E_B)_{(t_i)},t_i^{-1}\cdot Sym_B(E_B)_{(t_i)},W_{s_{B|D_+(t_i)}}}
\end{equation*} 
\begin{equation*}
    \otimes
 \end{equation*}
 \begin{equation*}
    \cohs{Sym_C(E_C)_{(t_i)},t_i^{-1}\cdot Sym_C(E_C)_{(t_i)},W_{s_{C|D_+(t_i)}}}
 \end{equation*}
 \begin{equation}
 \mu \downarrow
 \end{equation}
\begin{equation*}
 \cohs{Sym_{B\otimes_AC}(E_{B\otimes_AC})_{(t_i)},t_i^{-1}\cdot Sym_{B\otimes_AC}(E_{B\otimes_AC})_{(t_i)},W_{s_B \boxplus s_{C|D_+(t_i)}}}
\end{equation*}
is defined similarly.

We need to verify that the two compositions
\begin{equation}
  \begindc{\commdiag}[18]
    \obj(0,20)[1]{$ \cohs{B,E_B^{\vee},s_B}\otimes \cohs{C,E_C^{\vee},sC}$}
    \obj(0,-20)[4]{$\cohs{Sym_{B\otimes_AC}(E_{B\otimes_AC})_{(t_i)},t_i^{-1}\cdot Sym_{B\otimes_AC}(E_{B\otimes_AC})_{(t_i)},W_{(s_B\boxplus s_C)_{|D_+(t_i)}}}$}
    \mor(-5,15)(-5,-15){$\mu \circ (\Upsilon_{(B,s_B)}\otimes \Upsilon_{(C,s_C)})$}[\atright,\solidarrow]
    \mor(5,15)(5,-15){$\Upsilon_{(B\otimes_AC,s_B\boxplus s_C)}\circ \mu$}
  \enddc
\end{equation}
are isomorphic up to natural isomorphism. The composition $\Upsilon_{(B\otimes_AC,s_B\boxplus s_C)}\circ \mu$ is 
\begin{equation}
  ((M,d_M,h_M),(N,d_N,h_N))\mapsto 
\end{equation}
\begin{equation*}
(M\otimes_A N\otimes_{B\otimes_AC} Sym_{B\otimes_AC}(E_{B\otimes_AC})_{(t_i)}, d \otimes id, \chi_{i,h_M\boxplus h_N}),
\end{equation*}
while the composition $\mu \circ (\Upsilon_{(B,s_B)}\otimes \Upsilon_{(C,s_C)})$ is
\begin{equation}
  ((M,d_M,h_M),(N,d_N,h_N))\mapsto 
\end{equation}
\begin{equation*}
((M\otimes_B Sym_{B}(E_{B})_{(t_i)})\otimes_{Sym_A(E)_{(t_i)}} (N\otimes_C Sym_{C}(E_{C})_{(t_i)}), d\otimes id, \chi_{i,h_M}\boxplus \chi_{i,h_N}).
\end{equation*}
It is clear that there is a canonical isomorphism of $B\otimes_AC$-dg modules
\begin{equation}
   (M\otimes_A N\otimes_{B\otimes_AC} Sym_{B\otimes_AC}(E_{B\otimes_AC})_{(t_i)}, d\otimes id)\simeq
\end{equation}
\begin{equation*}
 ((M\otimes_B Sym_{B}(E_{B})_{(t_i)})\otimes_{Sym_A(E)_{(t_i)}} (N\otimes_C Sym_{C}(E_{C})_{(t_i)}),d\otimes id).
\end{equation*}
Then we only need to show that
\begin{equation}
\chi_{i,h_M\boxplus h_N}=\chi_{i,h_M}\boxplus \chi_{i,h_N}
\end{equation}
under this natural isomorphism. This follows from the definitions.

The units of the lax monoidal pseudo functors $(B,s_B)\mapsto \cohs{B,E_B^{\vee},s_B}$ and $(B,s_B)\mapsto \cohs{Sym_B(E_B)_{(t_i)}, t_i^{-1}\cdot Sym_B(E_B)_{(t_i)}, W_{s_B}}$ are compatible: the diagram
\begin{equation}
  \begindc{\commdiag}[13]
    \obj(-40,0)[1]{$\underline{A}$}
    \obj(40,20)[2]{$\cohs{A,E^{\vee},0}$}
    \obj(40,-20)[3]{$\cohs{Sym_A(E)_{(t_i)},t_i^{-1}Sym_A(E)_{(t_i)},0}$}
    \mor(-38,2)(35,15){$$}
    \mor{2}{3}{$$}
    \mor(-38,-2)(35,-15){$$} 
  \enddc
\end{equation}

commutes. Here $\underline{A}$ is the dg category with one object $*$ and $End(*)=A$. The diagonal arrows are determined by
\begin{equation}
  *\mapsto A, \hspace{0.5cm} *\mapsto Sym_A(E)_{(t_i)}
\end{equation}
where $A$ (resp. $Sym_A(E)_{(t_i)}$) is concentrated in degree $0$ and the map $A\rightarrow  E[-1]$ (resp. $Sym_A(E)_{(t_i)}\rightarrow t_iSym_A(E)_{t_i}$) is zero.
\end{construction}
The previous constructions provide us with pseudo lax monoidal natural transformations
\begin{equation}
  \Upsilon_i^{\otimes} : \cohs{\bullet, E_{\bullet}^{\vee},s_{\bullet}}^{\otimes}\rightarrow \cohs{Sym_{\bullet}(E_{\bullet})_{(t_i)},t_i^{-1}Sym_{\bullet}(E_{\bullet}),W_{s_{\bullet}}}^{\otimes}.
\end{equation}
It is obvious that the $\Upsilon_i^{\otimes}$ preserve quasi-isomorphisms and that they are compatible with one another. In other words, they can be used to define the homotopy coherent diagram of lax monoidal $\oo$-natural transformations $(\ref{homotopy coherent diagram lax monoidal oo natural transformations upsilon 1})$.

Therefore, we get a lax monoidal $\oo$-natural transformation 
\begin{equation}
\Upsilon^{\otimes}: \cohbp{V(\bullet)}{\bullet}^{\otimes}\rightarrow \cohbp{V(W_{\bullet})}{\mathbb{P}(E_{\bullet}^{\vee})}^{\otimes}
\end{equation}
between the $\oo$-functors 
\begin{equation}
  \cohbp{V(\bullet)}{\bullet}^{\otimes}, \cohbp{V(W_{\bullet})}{\mathbb{P}(E_{\bullet}^{\vee})}^{\otimes} : \textup{LG}_{(A,E)}^{\textup{aff}, \boxtimes, \textup{op}} \rightarrow \dgcatmo.
\end{equation}

By Kan extension and descent, we extend $\Upsilon^{\otimes}$ to a lax monoidal $\oo$-natural transformation between 
\begin{equation}
  \cohbp{V(\bullet)}{\bullet}^{\otimes}, \cohbp{V(W_{\bullet})}{\mathbb{P}(E_{\bullet}^{\vee})}^{\otimes} : \textup{LG}_{(A,E)}^{\boxtimes, \textup{op}} \rightarrow \dgcatmo,
\end{equation}
i.e. we extend to all twisted LG models of rank $r$ over $(A,E)$.

Finally, if $S$ is not affine, we define the lax monoidal $\oo$-natural tranformation
\begin{equation}\label{jp0}
  \Upsilon^{\otimes}:\cohbp{V(\bullet)}{\bullet}\rightarrow \cohbp{V(W_{\bullet})}{\mathbb{P}(\Ec_{\bullet}^{\vee})}
\end{equation}
of lax monoidal $\oo$-functors $\textup{LG}_{(S,\Ec)}^{\boxplus,\textup{op}}\rightarrow \dgcatmo$ as
\begin{equation}  
 \varprojlim_{Spec(A)\rightarrow S} \bigl ( \Upsilon^{\otimes}_{(A,\Ec_A)}:\cohbp{V(\bullet)}{\bullet}^{\otimes}\rightarrow \cohbp{V(W_{\bullet})}{\mathbb{P}(E_{\bullet}^{\vee})}^{\otimes}\bigr ),
\end{equation}
where the limit is taken over Zariski open subschemes $Spec(A)\rightarrow S$.
Here $ \Upsilon^{\otimes}_{(A,\Ec_A)}$ is the lax monoidal $\oo$-natural transformation we defined above. Once again, we have used that
\begin{equation}
  \dgcatmo= \varprojlim_{Spec(A)\rightarrow S}\textup{\textbf{dgCat}}^{\textup{idm},\otimes}_A
\end{equation}
and that
\begin{equation}
  \textup{LG}_{(S,\Ec)}^{\boxplus}\simeq \varprojlim_{Spec(A)\rightarrow S} \textup{LG}_{(A,\Ec_A)}^{\boxplus}
\end{equation}
(this is an analogue of Lemma \ref{zariski descent tlgmm} for twisted LG models of rank $r$, which can be proved \textit{mutatis mutandis}).

\begin{rmk}
  Let $(X,s_X)$ be a twisted LG model of rank $r$ over $(S,\Ec)$. Then
  \begin{equation}
    \Upsilon_{(X,s_X)}=j_*p_0^*:\cohbp{X_0}{X}\rightarrow \cohbp{V(W_{s_X})}{\mathbb{P}(\Ec_X^{\vee})},
  \end{equation}
  where we have used the notation of diagram (\ref{diagram for defining upsilon})
\end{rmk}
\begin{lmm}
$j:\mathbb{P}(\Ec_{X_0}^{\vee})\rightarrow V(W_{s})$ is an lci morphism of derived schemes.
\end{lmm}
\begin{proof}
Since the property of being lci is local, we can assume that $\Ec_X\simeq \Oc_X^r$ and $s_X=\underline{f}\in \Oc_X^r(X)$.
Since $j$ is clearly of finite presentation, we will only need to show that the relative cotangent complex is of Tor amplitude $[-1,0]$. Recall that we have the fundamental fiber-cofiber sequence
\begin{equation}
    j^*\Lb_{V(W_{\underline{f}})/\proj{n-1}{X}}\rightarrow \Lb_{\proj{r-1}{X_0}/\proj{r-1}{X}}\rightarrow \Lb_{\proj{r-1}{X_0}/V(W_{\underline{f}})}
\end{equation}
at our disposal. Since $\proj{r-1}{X_0}\simeq S\times^h_{\aff{r}{S}}\proj{r-1}{X}$, $\Lb_{\proj{r-1}{X_0}/\proj{r-1}{X}}$ is equivalent to the pullback of $\Lb_{S/\aff{r}{S}}\simeq \Oc_S^{r}[1]$ along $\proj{r-1}{X_{0}}\rightarrow S$, i.e. $\Lb_{\proj{r-1}{X_0}/\proj{r-1}{X}}\simeq \Oc_{\proj{r-1}{X_0}}^r[1]$. On the other hand, $\Lb_{V(W_{\underline{f}})/\proj{r-1}{X}}\simeq k^*\Lb_{\proj{r-1}{X}/V(\Oc(-1))}\simeq k^*(\Oc(-1)[1])$.
Therefore, $j^*\Lb_{V(W_{\underline{f}})/\proj{r-1}{X}}\simeq \Oc_{\proj{r-1}{X_0}}\otimes_{\Oc_{\proj{r-1}{X}}}\Oc(-1)$. We just need to identify the morphism $j^*\Lb_{V(W_{\underline{f}})/\proj{r-1}{X}}\simeq \Oc_{\proj{r-1}{X_0}}\otimes_{\Oc_{\proj{r-1}{X}}}\Oc(-1)[1]\rightarrow \Oc_{\proj{r-1}{X_0}}^{r}[1]\simeq \Lb_{\proj{r-1}{X_0}}$. It coincides with the morphism $(T_1,\dots,T_r)$. In particular, the cofiber of this morphism, i.e. $\Lb_{\proj{r-1}{X_0}V(W_{\underline{f}})}$, is equivalent to $\Oc_{\proj{r-1}{X_0}}^{r-1}[1]$.
\end{proof}
\begin{construction}
Since $j$ is an lci morphism, the dg functor $(\ref{jp0})$ preserves perfect complexes. Therefore, for every $(X,s_X)\in \tlgmr{S}$, we have an induced dg functor
\begin{equation}
    \Upsilon_{(X,s_X)}:=j_*p_0^*:\sing{X,s_X}\rightarrow \sing{\mathbb{P}(\Ec_X^{\vee}),W_{s_X}}
\end{equation}
Starting from $(\ref{jp0})$, by the usual standard arguments we thus obtain a lax monoidal $\oo$-natural transformation
\begin{equation}\label{reduction of codimension natural transformation}
    \Upsilon^{\otimes}:\sing{\bullet,\bullet}^{\otimes}\rightarrow \sing{\mathbb{P}(\Ec_{\bullet}^{\vee}),W_{\bullet}}^{\otimes}
\end{equation}
of functors
\begin{equation*}
\tlgmr{S}^{\textup{op},\boxplus}\rightarrow \md_{\sing{S,0}}(\dgcatm)^{\otimes},
\end{equation*}
where $\sing{\mathbb{P}(\Ec_{\bullet}^{\vee}),W_{\bullet}}^{\otimes}$ denotes the composition of ($\ref{sing projectivization}$) with
\begin{equation*}
    \md_{\sing{\mathbb{P}(\Ec_S),0}}(\dgcatm)^{\otimes}\rightarrow \md_{\sing{S,0}}(\dgcatm)^{\otimes}.
\end{equation*}
\end{construction}
\begin{thm}{\textup{(\cite{orl06}, \cite{bw15})}}\label{reduction of codimension theorem}
If $s_X$ is a regular section, the lax monoidal $\oo$-natural transformation $(\ref{reduction of codimension natural transformation})$ induces an equivalence of dg categories
\begin{equation}
   \Upsilon_{(X,s_X)}: \sing{X,s_X}\simeq \sing{\mathbb{P}(\Ec_X^{\vee}),W_{s_X}}.
\end{equation}
\end{thm}
\begin{proof}
This is an immediate consequence of \cite[Theorem A.4]{bw15}. Indeed, as the dg categories are triangulated, it suffices to show that the statement is true on the induced functor of triangulated categories. This coincides by construction with that of \textit{loc. cit.}
\end{proof}
\begin{rmk}
We actually believe that the statement above remains true even if $s_X$ is not assumed to be regular, as long as one considers the derived zero loci instead of the classical ones.
\end{rmk}

\subsection{The \texorpdfstring{$\ell$}{l}-adic realization of the dg category of singularities of a twisted LG model of rank \texorpdfstring{$r$}{r}}

It is now easy to obtain the following computation:
\begin{thm}\label{formula for (X,s)}
Let $(X,s_X)$ be a twisted LG model of rank $r$ over $(S,\Ec_S)$. Assume that $X$ is a regular scheme and that $s_X$ is a regular global section of $\Ec_X$. The following equivalence holds in $\md_{\Rlv_X(\sing{X,0})}(\shvl(X))$
\begin{equation}
    \Rlv_X(\sing{X,s_X})\simeq p_*i_*\Phi^{\textup{mi}}_{(\mathbb{P}(\Ec_X^{\vee}),W_{s_X})}(\Ql{}(\beta))[-1],
\end{equation}
where $i:V(W_{s_X})\rightarrow \mathbb{P}(\Ec_X^{\vee})$ is the closed embedding of the zero locus of $W_{s_X}$ and $p:\mathbb{P}(\Ec_X^{\vee})\rightarrow X$ is the canonical projection.
\end{thm}
\begin{proof}
As $X$ is regular and $s_X$ is a regular section, we have an equivalence $\sing{X,s_X}\simeq \sing{V(s_X)}$. Notice that in this situation $V(s_X)$ coincides with the underived zero locus of $s_X$. As $\mathbb{P}(\Ec_X^{\vee})$ is regular, by Theorem \ref{reduction of codimension theorem}, we have that 
\begin{equation*}
    \sing{X,s_X}\simeq \sing{V(W_{s_X})}\simeq \sing{\mathbb{P}(\Ec_X^{\vee}),W_{s_X}}.
\end{equation*} 
Then
\begin{equation}
    \Rlv_X(\sing{X,s_X})\simeq p_*\Rlv_{\mathbb{P}(\Ec_X^{\vee})}(\sing{\mathbb{P}(\Ec_X^{\vee}),W_{s_X}})
\end{equation}
\begin{equation*}
\underbracket{\simeq}_{\textup{Theorem } \ref{main theorem}}p_*i_*\Phi^{\textup{mi}}_{(\mathbb{P}(\Ec_X^{\vee}),W_{s_X})}(\Ql{}(\beta))[-1].
\end{equation*}

Notice that since $s_X$ is a regular global section of $\Ec_X$, $W_{s_X}$ is a regular global section of $\Oc_{\mathbb{P}(\Ec_X^{\vee})}$ and the hypothesis of Theorem \ref{main theorem} apply to the pair $(\mathbb{P}(\Ec_X^{\vee}), W_{s_X})$.

The fact that this equivalence holds in  $\md_{\Rlv_X(\sing{X,0})}(\shvl(X))$ follows immediately from the fact that \begin{equation*}
\Rlv_{\mathbb{P}(\Ec_X^{\vee})}(\sing{\mathbb{P}(\Ec_X^{\vee}),W_{s_X}})\simeq i_*\Phi^{\textup{mi}}_{(\mathbb{P}(\Ec_X^{\vee}),W_{s_X})}(\Ql{}(\beta))[-1]
\end{equation*}
is an equivalence of $\Rlv_{\mathbb{P}(\Ec_X^{\vee})}(\sing{\mathbb{P}(\Ec_X^{\vee}),0})$-modules and from the fact that 
\begin{equation}
\Rlv_X(\sing{X,0})\rightarrow p_*\Rlv_{\mathbb{P}(\Ec_X^{\vee})}(\sing{\mathbb{P}(\Ec_X^{\vee}),0})
\end{equation}
is a morphism of commutative algebras. 
\end{proof}
\begin{cor}\label{answer intial question}
Assume that $S=Spec(A)$ is a noetherian regular local ring of dimension $n$ and let $\pi_1,\dots,\pi_n$ be generators of the maximal ideal. Let $p:X\rightarrow S=Spec(A)$ be a regular, flat $S$-scheme of finite type. Let $\underline{\pi}:S\rightarrow \aff{n}{S}$ be the closed embedding associated to $\pi_1,\dots,\pi_n$. Then $\underline{\pi}\circ p$ is a regular global section of $\Oc_X^{n}$. Then the equivalence
\begin{equation*}
    \Rlv_X(\sing{X,\underline{\pi}\circ p})\simeq q_*i_*\Phi^{\textup{mi}}_{(\proj{n-1}{X},W_{\underline{\pi}\circ p})}(\Ql{}(\beta))[-1]
\end{equation*}
holds in $\md_{\Rlv_X(\sing{X,\underline{0}})}(\shvl(X))$. 

Here $q:\proj{n-1}{X}=Proj_X(\Oc_X[T_1,\dots,T_n])\rightarrow X$ is the canonical projection and $i:V(W_{\underline{\pi}\circ p})\rightarrow \proj{n-1}{X}$ is the closed embedding determined by the equation 
\begin{equation*}
    W_{\underline{\pi}\circ p}=p^*(\pi_1)\cdot T_1+\dots + p^*(\pi_n)\cdot T_n=0.
\end{equation*}
\end{cor}

\section{Towards a vanishing cycles formalism over \texorpdfstring{$\ag{S}$}{ag}}
The discussion in Section \ref{Monodromy-invariant vanishing cycles} suggests that it should be possible to construct a vanishing cycles formalism where the role of the base henselian trait is played by some more general geometric object. Moreover, the monodromy-invariant part of this construction should recover the sheaf of Definition \ref{Phimi}.

Notice that similar ideas, i.e. to develop a formalism of vanishing cycles over a base scheme $B$ with a closed-open decomposition $(\sigma,\eta)$, already appeared in J.~Ayoub's works \cite{ay08i,ay08ii,ay14}.

We will give a complete account on the formalism of tame vanishing cycles over $\ag{S}$ in a forthcoming paper in collaboration with D.-C.~Cisinski.

\subsection{Tame vanishing cycles over \texorpdfstring{$\aff{1}{S}$}{A1}}
It is well known that even if tame vanishing cycles where first defined for schemes over an henselian trait, what one actually uses is the geometry of $\aff{1}{S}$. Indeed, one can consider the following diagram
\begin{equation}
      \begindc{\commdiag}[18]
      \obj(-40,0)[1]{$S$}
      \obj(0,0)[2]{$\aff{1}{S}=Spec_S(\Oc_S[t])$}
      \obj(40,0)[3]{$\gm{S}$}
      \obj(100,0)[4]{$\varprojlim_{n\in \gm{S}(S)} \gm{S}[x]/(x^{n}-t)$}
      \mor{1}{2}{$i_0$}[\atleft,\injectionarrow]
      \mor{3}{2}{$j_0$}[\atright,\injectionarrow]
      \mor{4}{3}{$$}
      \enddc
\end{equation}

and use it to replace the one considered in \cite{sga7i,sga7ii} when $S=Spec(A)$ is a strictly henselian trait
\begin{equation}
    \begindc{\commdiag}[15]
    \obj(0,0)[1]{$\sigma$}
    \obj(30,0)[2]{$S$}
    \obj(60,0)[3]{$\eta$}
    \obj(90,0)[7]{$\eta^t.$}
    \mor{1}{2}{$$}[\atright,\injectionarrow]
    \mor{3}{2}{$$}[\atleft,\injectionarrow]
    \mor{7}{3}{$$}
    \cmor((90,4)(89,6)(87,7)(60,8)(33,7)(31,6)(30,4)) \pdown(60,14){$$}
    \enddc
\end{equation}

If we pullback the first diagram along the morphism $S\rightarrow \aff{1}{S}$ given by an uniformizer, we recover the second one. It is not surprising that in this way we are only able to recover the so called tame vanishing cycles. Indeed, the wild inertia group is of arithmetic nature.
One can define nearby cycles for schemes over $\aff{1}{S}$ in the usual way: for $f:X\rightarrow \aff{1}{S}$ consider the diagram
\begin{equation}
    \begindc{\commdiag}[18]
    \obj(-30,10)[1]{$X_0$}
    \obj(0,10)[2]{$X$}
    \obj(60,10)[3]{$X_{\mathcal{U}_n}$}
    \obj(130,10)[4]{$X_{\mathcal{U}_{\oo}}$}
    \obj(-30,-10)[5]{$S$}
    \obj(0,-10)[6]{$\aff{1}{S}$}
    \obj(60,-10)[7]{$\gm{S}[x]/(x^n-t)=\mathcal{U}_n$}
    \obj(130,-10)[8]{$\mathcal{U}_{\oo}=\varprojlim_{n\in \Oc_S^{\times}} \mathcal{U}_n.$}
    \mor{1}{2}{$i$}
    \mor{3}{2}{$j_n$}[\atright,\solidarrow]
    \mor{4}{3}{$\phi_n$}[\atright,\solidarrow]
    \mor{1}{5}{$f_0$}[\atright,\solidarrow]
    \mor{2}{6}{$f$}[\atright,\solidarrow]
    \mor{3}{7}{$f_n$}
    \mor{4}{8}{$f_{\oo}$}
    \mor{5}{6}{$i_0$}[\atright,\solidarrow]
    \mor{7}{6}{$j_{0,n}$}
    \mor{8}{7}{$\phi_{0,n}$}
    \cmor((130,13)(130,17)(127,20)(75,20)(3,20)(0,17)(0,13)) \pdown(75,25){$j_{\oo}$}
    \cmor((130,-13)(130,-17)(127,-20)(75,-20)(3,-20)(0,-17)(0,-13)) \pup(75,-25){$j_{0,\oo}$}
    \enddc
\end{equation}
\begin{defn}
Let $\Fc \in \shvl(X)$. 
\begin{itemize}
    \item The $\ell$-adic sheaf of nearby cycles of $\Fc$ is defined as
\begin{equation}
    \Psi_f(\Fc):=i^*j_{\oo*}j_{\oo}^*\Fc \simeq \varinjlim_{n\in \Oc_S^{\times}} i^*j_{n*}j_n^*\Fc.
\end{equation}
Notice that, since $\mu_{n,S}:=\Oc_S[x]/(x^n-1)$ naturally acts on $X_{\mathcal{U}_n}$, then $i^*j_{n*}j_n^*\Fc$ has a natural induced action. Then $\Psi_f(\Fc)$ has a natural action of $\varprojlim_{n\in \Oc_S^{\times}}\mu_{n,S}=:\mu_{\oo,S}$.
\item There is a natural morphism $i^*\Fc \rightarrow \Psi_f(\Fc)$, that we can see as a $\mu_{\oo,S}$-equivariant morphism if we endow $i^*\Fc$ with the trivial action. The sheaf of vanishing cycles $\Phi_f(\Fc)$ of $\Fc$ is then defined as the cofiber of this morphism, with the induced $\mu_{\oo,S}$-action.
\end{itemize}
\end{defn}
\subsection{Tame vanishing cycles over \texorpdfstring{$\ag{S}$}{ag}}
We can reproduce the situation described above for schemes over $\ag{S}$. In this case, the role of the zero section is played by $\bgm{S} \rightarrow \ag{S}$, and that of the open complementary by $S\simeq \gm{S}/\gm{S}\rightarrow \ag{S}$. We can consider elevation to the $n^{th}$ power in this case too:
\begin{equation}
    \Theta_n: \ag{S} \rightarrow \ag{S}
\end{equation}
which can be described as the functor
\begin{equation}
    (p:X\rightarrow S, \Lc_X, s_X)\mapsto (p:X\rightarrow S, \Lc_X^{\otimes n},s_X^{\otimes n}).
\end{equation}
See \cite{c07} and \cite{agv08}.
\begin{rmk}
Notice that the $\Theta_n$ are compatible, i.e.
\begin{equation}\label{compatibility Theta_n}
    \Theta_n\circ \Theta_m\simeq \Theta_{mn} \simeq \Theta_m\circ \Theta_n
\end{equation}
for any $n,m\in \N$.
\end{rmk}
We can therefore consider the diagrams (for each $n\in \Oc_S^{\times}$)
\begin{equation}\label{base vanishing cycles ag}
    \begindc{\commdiag}[18]
    \obj(-70,-10)[1]{$\bgm{S}$}
    \obj(-20,-10)[2]{$\ag{S}$}
    \obj(20,-10)[3]{$S.$}
    \obj(-70,10)[4]{$\bgm{S} \times_{\ag{S}}\ag{S}$}
    \obj(-20,10)[5]{$\ag{S}$}
    \obj(20,10)[6]{$S$}
    \obj(-140,10)[7]{$\bgm{S}$}
    \mor{1}{2}{$i_0$}[\atright,\solidarrow]
    \mor{3}{2}{$j_0$}
    \mor{4}{1}{$$}
    \mor{5}{2}{$\Theta_n$}
    \mor{6}{3}{$id$}
    \mor{4}{5}{$$}
    \mor{6}{5}{$j_{0,n}$}[\atright,\solidarrow]
    \mor{7}{4}{$t_{0,n}$}
    \mor{7}{1}{$\Theta_{0,n}$}[\atright,\solidarrow]
    \cmor((-140,14)(-140,18)(-138,20)(-70,20)(-22,20)(-20,18)(-20,14)) \pdown(-70,25){$i_{0,n}$}
    \enddc
\end{equation}
Notice that even if $t_{0,n}:\bgm{S} \rightarrow \bgm{S}\times_{\ag{S}}\ag{S}$ is not an equivalence, it shows $\bgm{S}\times_{\ag{S}}\ag{S}$ as a nilpotent thickening of $\bgm{S}$. More explicitly, objects of $\bgm{S}$ are pairs $(p:X\rightarrow S,\Lc_X)$, while those of $\bgm{S} \times_{\ag{S}}\ag{S}$ are triplets $(p:X\rightarrow S,\Lc_X,s_X)$, where $s_X$ is a $n$-torsion global section of $\Lc_X$. What is important for us is that pullback along $t_{0,n}$ induces an equivalence in \'etale cohomology. Moreover, $\Theta_n$ is a finite morphism of stacks. In particular, the base change theorem should be valid for the cartesian squares above. However, even if the $6$-functors formalism for the \'etale cohomology of stacks has been developed (\cite{lz15}, \cite{lz17i}, \cite{lz17ii}), the proper base change theorem has been proved in the case of representable morphisms. The morphisms $\Theta_n$ are not representable. For example, the pullback of $\Theta_n$ along the canonical atlas $\aff{1}{n}\rightarrow \ag{S}$ is $\aff{1}{S}/\mu_n$.

For any $X$-point $(p:X\rightarrow S, \Lc_X,s_X)$ of $\ag{S}$ (e.g. for any twisted LG model over $S$), we can consider the following diagram, cartesian over (\ref{base vanishing cycles ag})
\begin{equation}
    \begindc{\commdiag}[18]
    \obj(-70,-10)[1]{$X_0$}
    \obj(-20,-10)[2]{$X$}
    \obj(30,-10)[3]{$X_{\mathcal{U}}.$}
    \obj(-70,10)[4]{$X_0 \times_{X}X_n$}
    \obj(-20,10)[5]{$X_n$}
    \obj(30,10)[6]{$X_{\mathcal{U}}$}
    \obj(-140,10)[7]{$X_{0,n}$}
    \mor{1}{2}{$i$}[\atright,\solidarrow]
    \mor{3}{2}{$j$}
    \mor{4}{1}{$$}
    \mor{5}{2}{$\Omega_n$}
    \mor{6}{3}{$id$}
    \mor{4}{5}{$$}
    \mor{6}{5}{$j_{n}$}[\atright,\solidarrow]
    \mor{7}{4}{$t_n$}
    \mor{7}{1}{$\Omega_{0,n}$}[\atright,\solidarrow]
    \cmor((-140,14)(-140,18)(-138,20)(-70,20)(-22,20)(-20,18)(-20,14)) \pdown(-70,25){$i_{n}$}
    \enddc
\end{equation}
The same observations we made for the base diagram remain valid in this case too.
\begin{notation}
For a stack $Y$, and a torsion abelian group $\Lambda$ (we will only be interested in the case $\Lambda=\Z/\ell^n\Z$, $\ell \in \Oc_S^{\times}$) let $\shv(Y,\Lambda)$ the $\oo$-category of $\md_{\Lambda}$-valued \'etale sheaves on $Y$.
\end{notation}
We can then propose the following definition.
\begin{defn}
Let $\Fc \in \shv(X,\Lambda)$. Then
\begin{equation}
    \Psi_{n}(\Fc):=i_n^*j_{n*}j^*\Fc\simeq i_n^*j_{n*}j_n^*\Omega_n^*\Fc \in \shv(X_{0,n},\Lambda).
\end{equation}
Moreover, let 
\begin{equation}
    sp_n:i_n^*\Omega_n^*\Fc\rightarrow \Psi_n(\Fc)
\end{equation}
be the morphism induced by the unit of $(j_n^*,j_{n*})$.
\end{defn}
We can then consider the images of these morphisms in $\shv(X_0,\Lambda)$
\begin{equation}
    \Omega_{0,n*}(sp_n:i_n^*\Omega_n^*\Fc\rightarrow \Psi_n(\Fc)).
\end{equation}
Consider the diagram
\begin{equation}
    \begindc{\commdiag}[18]
    \obj(-50,30)[1]{$X_{0,nm}$}
    \obj(0,30)[2]{$X_0\times_X X_{nm}$}
    \obj(50,30)[3]{$X_{nm}$}
    \obj(85,30)[4]{$X_{\mathcal{U}}$}
    \obj(-50,0)[5]{$X_{0,n}$}
    \obj(0,0)[6]{$X_0\times_X X_{n}$}
    \obj(50,0)[7]{$X_{n}$}
    \obj(85,0)[8]{$X_{\mathcal{U}}$}
    \obj(0,-30)[9]{$X_{0}$}
    \obj(50,-30)[10]{$X$}
    \obj(85,-30)[11]{$X_{\mathcal{U}}.$}
    \mor{1}{2}{$$}
    \mor{2}{3}{$$}
    \mor{4}{3}{$j_{nm}$}[\atright,\solidarrow]
    \mor{5}{6}{$$}
    \mor{6}{7}{$$}
    \mor{8}{7}{$j_n$}
    \mor{9}{10}{$i_0$}
    \mor{11}{10}{$j_0$}
    \mor{1}{5}{$\Omega_{0,m}^n$}[\atright,\solidarrow]
    \mor{2}{6}{$$}
    \mor{3}{7}{$\Omega_m^n$}
    \mor{4}{8}{$id$}
    \mor{5}{9}{$\Omega_{0,n}$}[\atright,\solidarrow]
    \mor{6}{9}{$$}
    \mor{7}{10}{$\Omega_n$}
    \mor{8}{11}{$id$}
    \cmor((-50,33)(-50,37)(-47,40)(0,40)(47,40)(50,37)(50,33)) \pdown(0,45){$i_{nm}$}
    \cmor((-60,30)(-70,30)(-75,25)(-75,0)(-70,-15)(-45,-30)(-7,-30)) \pright(-85,0){$\Omega_{0,nm}$}
    \cmor((-45,5)(-43,10)(-40,11)(0,11)(40,11)(43,10)(45,5)) \pdown(-10,15){$i_n$}
    \cmor((55,25)(60,24)(65,21)(66,0)(65,-21)(60,-24)(55,-25)) \pleft(73,10){$\Omega_{nm}$}
    \enddc
\end{equation}
Then we have
\begin{equation}
    \begindc{\commdiag}[18]
    \obj(-50,15)[1]{$sp_n:i_n^*\Omega_n^*\Fc$}
    \obj(50,15)[2]{$i_n^*j_{n*}j_n^*\Omega_n^*\Fc$}
    \obj(-60,-15)[3]{$\Omega_{0,m*}^n(sp_{nm}):\Omega_{0,m*}^ni_{nm}^*\Omega_{nm}^*\Fc$}
    \obj(50,-15)[4]{$\Omega_{0,m*}^ni_{nm}^*j_{nm*}j_{nm}^*\Omega_{nm}^*\Fc$}
    \mor{1}{2}{$$}
    \mor(-40,14)(-40,-14){$1\rightarrow \Omega_{m*}^n\Omega_m^{n*}$}
    \mor{2}{4}{$\simeq$}
    \mor{3}{4}{$$}
    \obj(-40,-25)[5]{$\simeq i_n^*\Omega_{m*}^n\Omega_m^{n*}\Omega_n^*\Fc$}
    \obj(50,-25)[6]{$\simeq i_n^*j_{n*}j_n^*\Omega_n^*\Fc.$}
    \enddc
\end{equation}
We can then consider 
\begin{equation}
    \varinjlim_{n\in \N^{\times}}\Omega_{0,n*}i_n^*\Omega_n^*\Fc\rightarrow i_0^*j_{0*}j_0^*\Fc \in \shv(X_0,\Lambda).
\end{equation}
For $\Lambda=\Z/\ell^d\Z$ one can then consider the induced morphism on $\ell$-adic sheaves obtained by taking the limit over $d$ and then tensor with $\Ql{}$. It is expected that in this way one is able to recover monodromy-invariant vanishing cycles. In order to explain way the first Chern class of the line bundle appears in the computation, it might be useful to look at the base diagram. The $\ell$-adic cohomology of $\bgm{S}$ is $\Ql{}[c_1]$, where $c_1$ is the universal first Chern class and lies in degree $2$. We expect that 
\begin{equation}
 \Ql{}\otimes_{\Z_{\ell}}\bigl (\varprojlim_{d}   \varinjlim_{n\in \N^{\times}}\Theta_{0,n*}i_n^*\Theta_n^* \Z/\ell^d\Z{}(\beta)\bigr )\simeq cofib(\Ql{}(\beta)\xrightarrow{c_1} \Ql{}(\beta))
\end{equation}
and one recovers $cofib(\Ql{,X_0}(\beta)\xrightarrow{c_1(\Lc_{X_0})}\Ql{,X_0}(\beta))$
by a formula of the kind
\begin{equation}
    \Ql{}[c_1]\simeq \textup{H}^*(\bgm{S},\Ql{})\rightarrow \textup{H}^*(X_0,\Ql{})
\end{equation}
\begin{equation*}
    c_1\mapsto c_1(\Lc_{X_0}),
\end{equation*}
where $\Lc_{X_0}$ is the line bundle which determines the morphism $X_0\rightarrow \bgm{S}$.

\section{Some remarks on the regularity hypothesis}\label{about the regularity hypothesis}
In the theorems about the $\ell$-adic realizaiton of the dg category of a (twisted, $n$-dimensional) LG model the regularity assumption on the ambient scheme is crucial. Indeed, we are not able to compute the motivic realization (and hence, the $\ell$-adic one) of $\cohbp{X_0}{X}$. However, this dg category sits in the following pullback diagram
\begin{equation}\label{pullback cohbp}
    \begindc{\commdiag}[18]
    \obj(-30,15)[1]{$\cohbp{X_0}{X}$}
    \obj(30,15)[2]{$\perf{X}_{X_0}$}
    \obj(-30,-15)[3]{$\cohb{X_0}$}
    \obj(30,-15)[4]{$\cohb{X}_{X_0}.$}
    \mor{1}{2}{$\mathfrak{i}_*$}
    \mor{1}{3}{$$}[\atright, \injectionarrow]
    \mor{2}{4}{$$}[\atright, \injectionarrow]
    \mor{3}{4}{$\mathfrak{i}_*$}
    \enddc
\end{equation}
It is well known after Quillen's d\'evissage for G-theory that 
\begin{equation}
    \Mv(\cohb{X_0})\simeq \Mv(\cohb{X}_{X_0}).
\end{equation}
Therefore, to say that $\Mv(\mathfrak{i}_*:\cohbp{X_0}{X}\rightarrow \perf{X}_{X_0})$ is an equivalence or to say that the image of square (\ref{pullback cohbp}) via $\Mv$ is still a pullback square are equivalent statements. We believe that this is the case, even though this matter will be investigated elsewhere. Notice that to know that such an equivalence holds true would allow to compute the motivic realization of $\cohbp{X_0}{X}$ (under the additional hypothesis $X_0\simeq \pi_0(X_0)$). Indeed, the localization sequence for K-theory would allow us to compute it as the fiber
\begin{equation}
\begindc{\commdiag}[18]
    \obj(0,20)[1]{$\Mv_X(\cohbp{X_0}{X})\simeq \Mv_X(\perf{X}_{X_0})$}
    \obj(0,0)[2]{$\Mv_X(\perf{X})$}
    \obj(0,-20)[3]{$\Mv_X(\perf{X-X_0}).$}
    \mor{1}{2}{$$}
    \mor{2}{3}{$$}
\enddc
\end{equation}
In other words,
\begin{equation}
    \Mv_X(\cohbp{X_0}{X})\simeq i_*i^!\bu_X.
\end{equation}
In this case, the proof of Theorem \ref{main theorem} would work without any changes.

\begin{center}
\textit{Acknowledgements}
\end{center}
This paper is part of the author's PhD thesis written under the supervision of B.~To\"en and G.~Vezzosi, and to them owes a lot. Their generosity in sharing their ideas and insights is the backbone upon which this work relies on.

The author wishes to thank F.~D\'eglise, T.~Dyckerhoff, B.~Keller, S.~Scherotzke and J.~Tapia for their comments and remarks on the work presented in this article.

Many thanks to D.~Beraldo, T.~Moulinos and M.~Robalo for many useful conversations.

The author thanks an anonymous referee for many useful remarks and comments.

This project has received funding from the European Research Council (ERC) under the European Union Horizon 2020 research and innovation programme (grant agreement NEDAG ADG-741501).

This project has received funding from the European Research Council (ERC) under the European Union Horizon 2020 research and innovation programme (grant agreement No.725010). 

More precisely, this work was carried out while the author was supported by the NEDAG PhD grant ERC-2016-ADG-741501 as a "doctorant contractuel employ\'e par le CNRS \`a l'unit\'e UMR5219 IMT", and partially written while he was supported by the BG-BB-AS ERC-2016-COG-725010 as a Reseach Fellow in the UCL Department of Mathematics in the service of University College London 'UCL'.

m.pippi@ucl.ac.uk

Department of Mathematics, University College London, Gordon Street 25, London WC1H 0AY, United Kingdom
\end{document}